%% file: main.tex
\documentclass[11pt]{amsart}
\usepackage{mine} 
    \makeatletter
    \def\@seccntformat#1{%
    \protect\textup{%
        \protect\@secnumfont
        \expandafter\protect\csname format#1\endcsname 
        \csname the#1\endcsname
        \protect\@secnumpunct
    }%
    }

    \makeatother

\title[Finiteness of heights in isogeny classes of motives]{Finiteness of heights in isogeny classes of motives with semistable reduction}

\author{Alice Lin}
\address{Department of Mathematics\\Harvard University}
\email{alicelin@math.harvard.edu}
\date{2025}

\begin{document}
\begin{abstract}
    Using integral p-adic Hodge theory, Kato and Koshikawa define a generalization of the Faltings height of an abelian variety to motives defined over a number field. Assuming the adelic Mumford-Tate conjecture, we prove a finiteness property for heights in the isogeny class of a motive, where the isogenous motives are not required to be defined over the same number field. This expands on a result of Kisin and Mocz for the Faltings height in isogeny classes of abelian varieties.
\end{abstract}

\maketitle
\tableofcontents

\numberwithin{equation}{section}



\input{newintro}


\input{prelim}

\input{BKF}

\input{motivehts}

\input{mtconj}

\input{case1}

\input{case2}


\bibliographystyle{alpha}
\bibliography{refs}

\end{document}

%% file: newintro.tex
\section{Introduction}

Heights are a useful tool to prove finiteness theorems in arithmetic geometry over global fields, as they provide a method of measuring the size of a given arithmetic object defined over a global field in terms of the various absolute values of its defining coordinates. For example, in proving the Mordell conjecture, Faltings introduces and utilizes a height $h_{\Fal}(A)$ for any abelian variety $A$ defined over a number field \cite{faltings1983germanmordell,faltings1986finiteness}. One major input is the Shafarevich conjecture for abelian varieties over number fields: for any number field $F$ and finite set $S$ of finite places of $F$, there are only finitely many isomorphism classes of abelian varieties of a fixed dimension defined over $F$ with good reduction outside of $S$. In particular, Faltings shows that the number of abelian varieties defined over $K$ in a fixed isogeny class is finite \cite[proof of Theorem 6]{faltings1986finiteness}.

We use the following to denote the isogeny class over $\Qbar$ of an abelian variety $A$ defined over a number field.
\[
    \mscr I(A)_{\ol{\Q}} := \{ B_{/F'}\mid \text{$\exists$ an isogeny $A \to B$ defined over some field $F' \subset \ol{\Q}$} \}.
\]
This set contains infinitely many isomorphism classes.
However, explicit computations in the case of elliptic curves suggest the following conjecture, formulated by Mocz: For any $c > 0$ the geometric isogeny class of a given abelian variety contains only finitely many isomorphism classes of Faltings height less than $c$. This is true in the case of CM elliptic curves, where one has a formula for $h_{\Fal}(E)$ due to Deligne \cite{deligne1985preuve} and extended by Nakkajima--Taguchi \cite{nakkajima1991generalization}. Szpiro and Ullmo prove the case for non-CM elliptic curves using Arakelov intersection theory \cite{szpiro1999variation}. In \cite{mocz2018new}, Mocz considers the case of CM abelian varieties using integral $p$-adic Hodge theory.

Kisin and Mocz prove the following result for abelian varieties satisfying the Mumford--Tate conjecture.
\begin{thmsimple}[\cite{kisinmocznorthcott}] \label{thm:kisinmocz}
    Let $A$ be an abelian variety defined over a number field. Assume that the Mumford--Tate conjecture holds for $A$. Then for any $c > 0$, the set
    \[
        \{ B \in \mscr I(A)_{\ol{\Q}} \mid h_{\Fal}(B) < c\}
    \]
    contains finitely many isomorphism classes of abelian varieties over $\ol{\Q}$. 
\end{thmsimple}

The main goal of this paper is to prove an analogous result for the Kato--Koshikawa height, which is a generalization of the Faltings height to motives. The setup is somewhat technical, so let us first mention an application to Shimura varieties. 

As a result of \Cref{thm:kisinmocz}, Kisin and Mocz prove a finiteness result for the Faltings height in the Hecke orbits of Galois-generic points on Hodge type Shimura varieties. Namely, Let $\mc H$ be the Hecke orbit of an algebraic point $x_0$ on a Shimura variety of Hodge type $\Sh(G,X)$. If the image of the Galois representation attached to $x_0$ is Zariski dense in $G$, then for any $c> 0$, the following set is finite:
    \[
        \{ x \in \mc H \mid h_{\Fal}(x) < c\}.
    \]
The Faltings height of a point $x \in \Sh(G,X)$ is well-defined because Hodge type Shimura varieties can be viewed as moduli spaces of abelian varieties. 
As a corollary of our main theorem, we prove the following for any Shimura variety. 
\begin{thmintro}\label{mainthmA}
    Let $\mc H$ be the Hecke orbit of an algebraic point $x_0$ on a Shimura variety $\Sh(G,X)$. Suppose that the image of the Galois representation attached to $x_0$ is open in the finite adelic points $G(\A_f)$. Then for any 
    $c> 0$, the following set is finite:
    \[
        \{ x \in \mc H \mid h(x) < c\}.
    \]
\end{thmintro}
This extends the corollary of Kisin--Mocz to Shimura varieties without a moduli interpretation in terms of abelian varieties, namely the exceptional Shimura varieties. 

To prove \Cref{mainthmA}, we use the height defined by Kato \cite{kato2014heights} and Koshikawa \cite{koshikawa2015heights} for a \emph{Koshikawa motive}, by which we mean any system of realizations coming from a Grothendieck motive over a number field with an integral structure on the adelic \'etale realization. This is referred to as a ``pure $\Z$-motive'' in \cite{koshikawa2015heights}. See \S \ref{section:heightsmotives} for a precise definition. Such families of Koshikawa motives exist for all Shimura varieties, including exceptional ones. 



The Kato--Koshikawa height is a direct generalization of the Faltings height, in that both are defined using an integral structure on the associated graded of the de Rham realization together with a norm at each infinite place of $F$ defined using the Hodge polarization. \Cref{mainthmA} is an application of the following theorem.

\begin{thmintro} \label{mainthmB}
    Let $ M$ be a Koshikawa motive
    defined over a number field $F$ with semistable reduction everywhere. 
    Assume the adelic Mumford--Tate conjecture for $ M$ and that the Hodge classes of $ M$ are absolute Hodge and de Rham (see \Cref{conj:nonHodgemaxadelicMT}, \ref{conj:blasius}, \ref{conj:absolutehodge}). Then the set
    \[
    \{  M' \in \mscr{I}(M)_{\ol{\Q}} \mid \; h(M')<c \}
    \]
    attains only finitely many height values.
\end{thmintro}

\begin{rmkunnumbered}
    If we consider an isogeny class of Koshikawa motives coming from fibers of a finite type family $\mc X \to S$ of smooth projective varieties over an $F$-variety $S$ satisfying some mild conditions \cite[\S 6]{koshikawa2015heights} and a quasifinite Hodge-theoretic period map, then one can upgrade the statement of the theorem to finiteness of isomorphism classes of Koshikawa motives rather than finiteness of height values. 
\end{rmkunnumbered}

Here, an isogeny $M' \to M$ defined over $F'$ refers to a finite-index inclusion of the integral structure of $M'$ into that of $M$ which is $\Gal(\ol{F'}/F')$-equivariant. The underlying vector spaces for $M$ and $M'$ are otherwise identical. This mirrors the situation of abelian varieties: one can define an isogeny $A\to B$ defined over $F'$ as a $\Gal(\ol{F'}/F')$-equivariant finite-index inclusion of $\hat{\Z}$-modules $\prod_p T_p(A) \to \prod_p T_p(B)$. 

We remark that the adelic Mumford--Tate conjecture is a stronger version of the Mumford--Tate conjecture, but the two are equivalent in the case of abelian varieties \cite{cadoret2020integral}. The conditions on the Hodge classes of $M$ are known to be true for abelian varieties due to Deligne \cite{deligne1982hodge} and Blasius \cite{blasius1994padic}, respectively. These conditions on the Hodge classes are also satisfied for the Hodge structures coming from algebraic points on a Shimura variety. 

\subsection{Proof outline}

The proof of \Cref{mainthmB} centers around a formula due to Koshikawa for the change of height $h(M')-h(M)$ when $M' \to M$ is an isogeny of motives. The framework of this argument follows \cite{kisinmocznorthcott}, which studies the analogous isogeny formula for the Faltings height. 
Any isogeny factors into a sequence of prime-power isogenies, so without loss of generality, we may assume that the isogeny is determined by the finite-index inclusion $T' \subsetneq T \subset M_{\et,p}$ of $\Zp$-lattices in the $p$-adic \'etale realization $M_{\et,p}$. The formula for $h(M')-h(M)$ is in terms of the torsion Galois representation $T/T'$ \cite[8.1]{koshikawa2015heights}.
Suppose that the $\Zp$-module $T'$ is only stable under $\Gal(\ol{F}/F')$ for a finite extension $F'/F$. If $T$ was $\Gal(\ol{F}/F)$-stable, then such a finite extension $F'$ exists for any finite-index sublattice $T'$. We can then adapt the isogeny formula of Koshikawa to show 
    \begin{align}
        h({ M}') - h({{M}})
        &= \log \#(T/T')  \left(  \frac{w}{2} - \nu(T, T')\right),
        \\
    	\nu(T,T')
        &=  \sum_{\substack{v\mid p \\ \text{of $F$}}} \frac{[F_v:\Qp]}{[F:\Q]} \int_{\sigma \in \Gal_F} \mu \left(\frac{\mf M_{v}(T)}{\mf M_{v}(\sigma T')}\right) d\sigma.\label{eqn:nuGaloisorbitsintro}
    \end{align}
Here, $w$ is the weight of the motive $M$, and $\nu(T,T')$ is a weighted average of invariants (slopes) attached to torsion Breuil--Kisin--Fargues modules $\mf M_v(T)/\mf M_v(T')$. As the notation suggests, $\mf M_v(T)$ and $\mf M_v(T')$ are functorially attached to $\Zp$-lattices in semistable $p$-adic Galois representations after choosing an embedding $\ol{F} \into \ol{\Qp}$ extending a place $v \mid p$ of $F$ \cite{kisin2006crystalline}. The cokernel of $\vphi: \mf M_v(T) \to \mf M_v(T)$ contains the data of the Hodge--Tate decomposition of the Galois representation $T \otimes \Qp$.
For a torsion quotient of finite free Breuil--Kisin--Fargues modules $\mf M$ and $\mf M'$, the slope $\mu(\mf M/\mf M')$ measures the relative size of the cokernels of Frobenius on $\mf M$ and $\mf M'$.

In the formula (\ref{eqn:nuGaloisorbitsintro}) for $\nu$, observe that there is an integral over the Galois orbit, and there is a sum over different places $v \mid p$. If there are infinitely many $M'$ in the isogeny class of $M$ with $h(M')< c$, one of the following two cases must occur: infinitude of $p$-power isogenies for a fixed prime, and infinitude of $p$-torsion isogenies for infinitely many primes $p$. 

\subsubsection{Fixed prime case}
In the first case, it is always possible to produce an infinite subset $\{T'_i \to T\}_{i \in \N}$ of isogenies which are of the form $T'_i = L + p^{n_i}T$ for some lower-rank $\Zp$-submodule $L \subset T$ with torsion-free quotient, and $n_i \to \infty$. A priori, $L$ does not have to satisfy any Galois-stability condition. For simplicity we can just consider $n_i = i$ for all $i$. We prove the following.
\begin{propsimple}[\S \ref{subsec:fixedprimeheightbound}]
    Let $M$ be as in \Cref{mainthmB}, and let $L$ and $T$ be as above. Then 
    \[
        \lim_{i \to \infty} \nu(T, L + p^iT) \le \frac{w}{2},
    \]
    where $w$ is the weight of the Koshikawa motive $M$. Moreover, if equality is achieved, then $L$ is stable by the action of $\Gal(\ol{F}/F')$ for some finite extension $F'/F$. 
\end{propsimple}

If the inequality is strict, it follows that the height $h(M_i')$ increases without bound as $i \to \infty$. If the inequality is an equality, then all the $M_i'$ are motives defined over some number field $E$. We can thus apply \cite[Theorem 9.8]{koshikawa2015heights} to conclude that the $h(M_i')$ for all $i$ attain only a finite set of values. 

To prove the Proposition, we first construct $\mf M(L)$ as a finite free Breuil--Kisin--Fargues module, and show that $\mf M(L + p^iT) = \mf M(L) + p^i \mf M(T)$. For any $\sigma \in \Gal(\ol{F}/F)$, we can bound 
$$
    \lim_i \mu (\mf M(T)/\mf M(\sigma L) \bmod{p^i})
$$
in terms of the length of the Frobenius cokernel $\mf M(L)/\vphi^*\mf M(L) \otimes \BdR^+$. However, as remarked earlier, the Frobenius cokernel of $\mf M(L)$ contains the data of the location of $L \otimes_{\Zp} C$ in the Hodge--Tate decomposition on $T \otimes_{\Zp}C$.

What's more, the Mumford--Tate assumption implies $\rho(\Gal_F) \subset G_{MT}(\Zp)$ is open, so we can replace the integral over $\sigma \in \Gal_F$ in (\ref{eqn:nuGaloisorbitsintro}) with one over $g \in G_{MT}(\Zp)$. The integral will equal the ``generic'' slope value, i.e.~the one attained by a measure one (or Zariski open) set of $g \in G_{MT}(\Zp)$.

The absolute Hodge class and \'etale-de Rham compatibility assumptions imply a Hodge symmetry on the Hodge--Tate decompositions for conjugate pairs of places $v \mid p$, which allows us to bound the sum of generic slope values for conjugate pairs of $v \mid p$ in terms of the weight $w$. Because $\nu$ is a sum over all places $v \mid p$, this Hodge symmetry relation implies the Proposition.

\subsubsection{Varying prime case}
Suppose there exist $p$-power isogenies $M' \to M$ with $h(M')< c$ for infinitely many primes $p$. Unlike the fixed-prime case, the proof in the varying-prime case of \cite{kisinmocznorthcott} still applies in our situation if we assume the adelic Mumford--Tate conjecture.
\begin{propsimple}[\Cref{cor:ptorsionheightbound}]
    Assume $M$ is as in \Cref{mainthmB}. For any prime $p$, let $T_p$ denote the $\Zp$-integral structure chosen on the $p$-adic \'etale realization $M_p$ for $M$. After replacing $F$ with a finite extension, there exists a constant $c_M > 0$ such that for $p \gg 0$, and for any proper sublattice $T_p' \subset T_p$ such that the $\Gal(\ol{F}/F)$-orbit of $T_p'$ generates $T_p$, the following inequality holds. 
    \[
        \frac{w}{2} - \nu(T_p,T_p') \ge \frac{1}{2[F:\Q](\rk(T_p)-1)}\left( 1 - \frac{c_M}{p}\right).
    \]
\end{propsimple}
The result of Koshikawa \cite[Theorem 9.8]{koshikawa2015heights} implies that there must be infinitely many prime-power isogenies of the form specified in the Proposition. However, for $p \gg 0$, the Proposition shows that any such $p$-power isogeny $M' \to M$ must have $h(M') > c$, where $c$ is as in \Cref{mainthmB}. 

We prove the Proposition by reducing to the situation of $p$-torsion isogenies, i.e.~$p$-power isogenies defined by a finite-index inclusion of $\Zp$-modules $T_p' \into T_p$ such that $T_p/T_p'$ is an $\Fp$-vector space. The proof in the $p$-torsion case is due to \cite[Proposition 3.6.3]{kisinmocznorthcott}, and we briefly summarize it here. The adelic Mumford--Tate conjecture for $M$ implies that for $p \gg 0$, we can replace the integral over $\Gal_F$ in (\ref{eqn:nuGaloisorbitsintro}) with an average over the $G_{MT}(\Fp)$-orbit acting on the subspace $T'_p \otimes \Fp$ in the vector space $T_p \otimes \Fp$. Allowing $p \gg 0$ also ensures that the associated $p$-torsion BKF modules are Fontaine--Laffaille, which obey a good structure theory. This, combined with a similar Hodge-symmetry inequality used in the fixed-prime case, also gives a lower-bound for $\nu(T_p, T'_p)$ whenever $T_p \supset T_p' \supset pT_p$. 

\td{With the adelic Mumford--Tate conjecture assumption in \Cref{mainthmB}, we can upgrade the finiteness of heights result of Koshikawa \cite[Theorem 9.8]{koshikawa2015heights} to finiteness of isomorphism classes of Koshikawa motives. In particular, under the adelic Mumford--Tate assumption, the Galois group acts through a reductive group and hence acts semisimply on the $p$-adic \'etale realization $M_p$ for every $p$. One can then show with the same argument as Faltings that for every $p$, there are only finitely many isomorphism classes of $\Gal(\ol{F}/F)$-stable $\Zp$-lattices in $M_p$. Moreover, when $p \gg 0$, the Galois group surjects onto $G_{MT}(\Zp)$ and hence onto $G_{MT}(\F_p)$. This implies that there is only one isomorphism class (up to scaling by powers of $p$) of Galois-stable $\Zp$-lattice in $M_p$, because the $\F_p$-representation is semisimple. }


\subsection{Organization of the paper}
In \S \ref{section:BKF}, we review some background on Breuil--Kisin and Breuil--Kisin--Fargues modules. In \S \ref{section:heightsmotives}, we define Koshikawa motives, define the Kato--Koshikawa height, and give the formula for height change under isogeny. In \S \ref{section:MTconj}, we state the assumptions used in the statement of Theorem \ref{mainthmB}. In \S \ref{section:case1}, we prove the case of having infinitely many $p$-power isogenies of bounded height for a fixed prime $p$. In \S \ref{section:varyingprime}, we prove the case of having $p$-power isogenies of bounded height for infinitely many primes $p$. 

\subsection{Acknowledgements}
This work was supported by the National Science Foundation Graduate Research Fellowship Program under Grant No. DGE 2140743. Any conclusions or recommendations expressed in this material are those of the author and do not necessarily reflect the views of the National Science Foundation.

We thank Teruhisa Koshikawa, Alexandra Hoey, Aaron Landesman, Frank Lu, Dylan Pentland, Sasha Petrov, Kush Singhal, Ananth Shankar, Naomi Sweeting, Yunqing Tang, Salim Tayou, Rosie Shen, Ziquan Yang, and Jit Wu Yap. We are particularly grateful to Mark Kisin for suggesting this problem and for many helpful discussions.

%% file: prelim.tex
\section{Notation}\label{section:notation}

\subsection{Number fields, local fields, and Galois groups}
\begin{itemize}
    \item Let $F \subset \ol{\Q}$ be a number field. $K$ will always denote a field extension of $\Qp$. 
    \item We denote by $\Gal_F := \Gal(\ol{F}/F)$ the absolute Galois group of $F$. 
    \item We will denote by $C := \widehat{\ol{\Qp}}$ the $p$-adic completion of the algebraic closure of $\Qp$, and $\C$ denotes the usual complex numbers. 
    \item Let $v$ denote a finite place of $F$ above $p$. We use the same notation to denote the corresponding embedding $v: F \into C$. We can make a choice of embedding $\ol{v}: \ol{\Q} \to C$ extending $v$, but when it's clear, we abuse notation to write $v: \ol{\Q}\to C$ instead.  
    \item We fix a field isomorphism $C \cong \C$. 
    \item $\A_f$ denotes the ring of finite adeles over $\Q$.
\end{itemize}

\subsection{$p$-adic Hodge theory}
\begin{itemize}
    \item The cyclotomic character has Hodge--Tate weight $-1$.
    \item $\Ainf$ denotes the ring $W(\mcO_C^\tilt)$. 
    \item  $\mu := [\veps] - 1 \in \Ainf$, where $\veps \in \mcO_C^\tilt$ is a compatible choice of $p^n$th roots of unity $(1,\zeta_p,\zeta_{p^2},\ldots)$. 
    \item $\xi$ is a generator of the kernel of $\theta: \Ainf \to \mcO_C$.
    \item We will write $\tilde{\xi} := \vphi(\xi) \in \Ainf$.
    \item Let $\mf S := W(k) [[u]]$, with the associated Eisenstein polynomial $E(u)$ with root $\pi \in \ol{\Qp}$. We write $f: \mf S \to \Ainf$ for the natural map defined by $u \mapsto [\pi^\tilt]$. 
    \item For a choice of compatible $p^n$th roots of the uniformizer $\pi \in K$, $K_\infty =\bigcup_n K (\pi^{1/p^n})$. 
\end{itemize}

\subsection{Miscellaneous}
An $R$-lattice in a $\Frac(R)$-vector space $V$ is a projective $R$-submodule $M$ of $V$ such that $M \otimes_R \Frac(R)= V$.

%% file: BKF.tex
\section{Breuil--Kisin--Fargues modules}
\label{section:BKF}

Breuil--Kisin--Fargues modules generalize Breuil--Kisin modules from ``classical'' integral $p$-adic Hodge theory. For a complete discretely valued field $K/\Qp$, Breuil--Kisin modules classify $\Gal_K$-stable $\Zp$-lattices in crystalline Galois representations. However, it is difficult to make these modules compatible with base change under ramified extensions $K'/K$. The BK modules do agree, however, after base-changing to $\Ainf$. This motivates considering Breuil--Kisin--Fargues modules.

\subsection{Breuil--Kisin modules}

Let $K$ be a complete discretely valued field extension of $\Qp$ with perfect residue field $k$, and let $\mcO_K$ be the ring of integers. Fix a uniformizer $\pi \in \mcO_K$ whose minimal polynomial is an Eisenstein polynomial $E(u) \in W(k)[u]$.
Define $\mf S := W(k)[[u]]$. 
We define the Frobenius endomorphism $\vphi: \mf S \to \mf S$ by $u \mapsto u^p$, and the natural Frobenius on $W(k)$. 
There is a natural map $\mf S \to \mcO_K$ where $u \mapsto \pi$. 

\begin{defn}
    A \emph{Breuil--Kisin (BK) module} is a $E(u)$-torsion free, finitely generated $\mf S$-module $\mf M$ equipped with an isomorphism
    \[
        \vphi_{\mf M}: \mf M \otimes_{\mf S, \vphi} \mf S[1/E(u)] \xto{\sim} \mf M[1/E(u)].
    \]
    We will often denote $\vphi^* \mf M := \mf M \otimes_{\mf S, \vphi} \mf S$.
\end{defn}

\subsubsection{}
Denote by $\Mod^{\vphi}_{/\mf S}$ (resp.~$\Mod^{\vphi,t}_{/\mf S}$) the category of finite free (resp.~$p$-power torsion) BK modules. We say that $\mf M$ is \emph{effective} if $\vphi_{\mf M}(\vphi^*\mf M) \subset \mf M$ inside the codomain $\mf M[1/E(u)]$. 

Choose compatible $p^n$th roots $\pi^{1/p^n}$ of $\pi$, and define $K_\infty := \bigcup_n K(\pi^{1/p^n})$. $\Gal_{K_\infty} \subset \Gal_K$ is its absolute Galois group.

\begin{theorem}[\cite{kisin2006crystalline}, \cite{liu2008lattices} Lemma 3.4.5, \cite{kisin2010integral} Theorem 1.2.1(2)]
    \label{thm:BKmodequivalence}
    Let $V$ be a semistable representation with Hodge--Tate weights in $[-w,0]$. For any $\Gal_{K_\infty}$-stable $\Zp$-lattice $T \subset V$, there is an effective module $\mf M \in \Mod^{\vphi}_{/\mf S}$ such that $T_{\mf S}(\mf M) := (\mf M \otimes \Ainf)^{\vphi = 1} \subseteq V$ equals $T$. Moreover, the functor $T_{\mf S}: \Mod^{\vphi}_{/\mf S} \to \Rep_{\Zp}(\Gal_{K_\infty})$ is fully faithful. There is also a canonical $\vphi, \Gal_{K_\infty}$-equivariant identification $\mf M \otimes_{\mf S} W(C^\tilt) \cong T \otimes_{\Zp} W(C^\tilt) $. 
\end{theorem}

\subsection{Breuil--Kisin--Fargues modules and Fargues' equivalence}

\begin{definition}[\cite{BMS2018integral}, Definition 4.22]
    A \emph{Breuil--Kisin--Fargues (BKF) module} is a $\tilde{\xi}$-torsion free, finitely presented $\Ainf$-module $\mf M$ with an isomorphism 
    \[
        \vphi_{\mf M}: \mf M \otimes_{\Ainf, \varphi} \Ainf[1/\tilde{\xi}] \xto{\sim} \mf M[1/\tilde{\xi}]
    \]
    such that $\mf M[1/p]$ is a finite projective (equivalently, free) $\Ainf[1/p]$-module. 
\end{definition}

\subsubsection{}
Denote by $\Mod^{\vphi, fp}_{/\Ainf}$ (resp.~$\Mod^{\vphi }_{/\Ainf}$, resp.~$\Mod^{\vphi, t}_{/\Ainf}$) the category of (finitely presented) BKF modules (resp.~finite free BKF modules, resp.~finite $p$-power torsion BKF modules). We say that $\mf M$ is \emph{effective} if $\vphi_{\mf M}(\vphi^* \mf M) \subset \mf M$ inside the codomain $\mf M[1/\tilde{\xi}]$. 

We denote by $\Mod^{\vphi,\le h}_{/\Ainf}$ (resp.~$\Mod^{\vphi, t, \le h}_{/\Ainf}$) the subcategory of $\Mod^{\vphi }_{/\Ainf}$ (resp.~$\Mod^{\vphi, t}_{/\Ainf}$) consisting of effective objects $\mf M$ such that $\mf M / \vphi^* \mf M$ is $\tilde{\xi}^{h}$-torsion.


\begin{theorem}[\cite{scholzeweinstein2020berkeley}, Theorem 14.1.1, originally due to Fargues]
    \label{thm:farguesequivalence}
    The following categories are equivalent:
    \begin{enumerate}[(1)]
        \item Pairs $(T,\Xi)$, where $T$ is a finite free $\Zp$-module and $\Xi \subset T \otimes_{\Zp} \BdR$ is a $\BdR^+$-lattice,
        \item Finite free Breuil-Kisin-Fargues modules $\mf M$.
    \end{enumerate}
\end{theorem}

\begin{remark}\label{rmk:XiMrelation}
    By the proof of the above theorem in \cite{scholzeweinstein2020berkeley}, $\Xi = \mf M \otimes_{\Ainf} \BdR^+$, and $T= (\mf M \otimes_{\Ainf} W(C^\tilt))^{\vphi_{\mf M} =1}$. However, the quasiinverse functor from $(T,\Xi)$ to the associated $\mf M$ is less explicit because it goes through auxiliary categories, in particular the category of modifications of vector bundles on the Fargues--Fontaine curve. However, the following lemma of \cite{BMS2018integral} gives some further understanding of the relation between $\mf M$ and $(T,\Xi)$.
\end{remark}

\begin{lemma}[\cite{BMS2018integral}, Lemma 4.26]\label{lemma:BMS4.26}
    Let $\mf M \in \Mod^{\vphi}_{/\Ainf}$. Let $T= (\mf M \otimes_{\Ainf} W(C^\tilt))^{\vphi_{\mf M} =1}$. There is a Frobenius-equivariant identification $T \otimes_{\Zp} W(C^\tilt) \cong \mf M \otimes_{\Ainf} W(C^\tilt)$. Moreover, under this isomorphism, one can identify the submodules $\mf M \otimes_{\Ainf}\Ainf[1/\mu] = T \otimes_{\Zp}\Ainf[1/\mu]$.
\end{lemma}

\begin{remark}[Tate twist and relation to Fargues' classification of free BKF modules]
    \cite[Example 4.24]{BMS2018integral} shows that $\Ainf\{1\}:= \frac{1}{\mu}(\Ainf \otimes_{\Zp}\Zp(1))$ is spanned by $e$ such that $\vphi(e)=e/\vphi(\xi)$. Going backwards, one can recover $\Zp(1) = (\Ainf\{1\} \otimes W(C^\tilt))^{\vphi=1}$ is $\Zp(e \otimes \mu) = \Zp(1)$. We can also compute $\Xi = \Ainf\{1\} \otimes_{\Ainf} \BdR^+$, so the pair $(T,\Xi)$ corresponding to $\Ainf\{1\}$ is $(\Zp(1), \frac{1}{\xi}(\Zp(1) \otimes \BdR^+))$. 
\end{remark}

\subsubsection{Base-changing from BK modules to BKF modules}\label{subsubsec:basechangefromBKtoBKF}
We will use the convention of \cite[\S 4]{BMS2018integral}, where the Frobenius on Breuil--Kisin modules has cokernel killed by a power of $E(u)$, and the Frobenius on Breuil--Kisin--Fargues modules has cokernel killed by a power of $\vphi(\xi)$. Let $\mf M_{BK}$ denote a BK module. As is the convention in \cite[Proposition 4.32]{BMS2018integral}, we can define a BKF module
$$\mf M_{BKF}:= \mf M_{BK} \otimes_{\mf S, \vphi \circ f} \Ainf,$$ 
where $f: \mf S \to \Ainf$ is given by $u \mapsto [\pi^\tilt]$, where $\pi$ is a root of $E(u)$, so that $\vphi\circ f(u) = [\pi^\tilt]^p$.

\subsection{Slope of torsion BKF modules}
\subsubsection{}\label{subsubsec:normalizedlength}
For a finitely presented $\mcO_C^\tilt$-module $N$, we may write $N \cong \bigoplus_i \mcO_C^\tilt/a_i\mcO_C^\tilt$ \cite[Proposition 2.10]{scholze2013padic}. We define the normalized length of $N$ to be
$$\lambda(N) := \sum_i v(a_i),$$ 
where the valuation $v$ on $\mcO_C^\tilt$ is normalized so that $v(\ul{p}) = 1$. This is not the same as the standard notion of length in commutative algebra. 
Let $M$ be a finitely presented $\Ainf$-module which is killed by a power of $(p,\tilde{\xi})$. Then we can write $M$ as a finite number of extensions of finitely presented $\mcO_C^\tilt$-modules, and we can extend the definition of $\lambda$ on $\mcO_C^\tilt$-modules to $M$ by additivity \cite[2.2.13]{cornut2019harder}.

\begin{definition}[\cite{kisinmocznorthcott}, 3.1.4] \label{def:slope}
    Let $\mf M \in \Mod^{\vphi, t}_{/\Ainf}$ be an effective object. Then $\mf M/ \vphi^*(\mf M)$ is a finitely presented $\Ainf$-module which is killed by a power of $(p,\tilde{\xi})$. We define
    \begin{align}
        \deg \mf M &:= \lambda(\mf M/ \vphi^* \mf M) \\
        \rk \mf M &:= \ell_{W(C^\tilt)}(\mf M \otimes_{\Ainf} W(C^\tilt)),
    \end{align}
    where $\ell_{W(C^\tilt)}(\cdot)$ denotes the usual notion of length of a $W(C^\tilt)$-module from commutative algebra. Then, define the slope of $\mf M$ as 
    \begin{align}
        \mu(\mf M) := \frac{\deg \mf M}{\rk \mf M}. 
    \end{align}
\end{definition}

This definition of slope induces a Harder--Narasimhan theory on effective torsion BKF modules $\Mod^{\vphi,t}_{/\Ainf}$ \cite[\S 2.4]{cornut2019harder}. For an effective torsion BK module $\mf M \in \Mod^{\vphi,t}_{/\mf S}$, \cite{LWE2020harder} define the degree to be 
\begin{align}\label{eqn:LWEdeg}
\deg \mf M := \frac{1}{[K:\Qp]} \ell_{\Zp}(\mf M/\vphi^* \mf M),
\end{align}
where $\ell_{\Zp}$ denotes usual length as a finite $\Zp$-module. One can check that this agrees with $\deg(\mf M \otimes_{\mf S, \vphi} \Ainf)$. Both \cite{cornut2019harder} and \cite{LWE2020harder} define rank in the same way, so the two slope definitions agree.

Here, we list some properties of degree and slope which follow from the definitions and are proved in \cite[Lemma 3.1.5]{kisinmocznorthcott}. 
\begin{lemma} \label{lemma:slopeproperties}
    \begin{enumerate}[(1)]
        \item $\deg(-)$ and $\rk(-)$ are both additive in short exact sequences in $\Mod^{\vphi,t}_{/\Ainf}$. This implies that for short exact sequences 
        \[
            0 \to \mf M' \to \mf M \to \mf M'' \to 0,
        \]
        we have the inequalities 
        \begin{align*}
            \min(\mu(\mf M'), \mu(\mf M'')) \le \mu (\mf M) \le \max(\mu(\mf M'), \mu(\mf M'')).
        \end{align*}
        \item If $h: \mf M_1 \to \mf M_2$ is a map in $\Mod^{\vphi,t}_{/\Ainf}$ which is an isomorphism after inverting $\tilde{\xi}$, then $\mu(\mf M_1) \ge \mu(\mf M_2)$, and equality holds iff $h$ is an isomorphism.
        \item For $\mf M \in \Mod^{\vphi,t}_{/\mcO_C^\tilt}$, $\deg(\mf M) = \deg(\wedge^{r}_{\mcO_C^\tilt} \mf M)$, where $r = \rk_{\mcO_C^\tilt} \mf M$. 
    \end{enumerate}
\end{lemma}
\subsubsection{}\label{subsubsec:mumaxHdefn}
Suppose $\mf M \in \Mod^{\vphi}_{/\mcO_C^\tilt}$ is effective, and write $\mf M / \vphi^* \mf M \cong \bigoplus_{i=1}^n \mcO_C^\tilt / a_i \mcO_C^\tilt $. Using the same valuation $v$ on $\mcO_C^\tilt $ as in \ref{subsubsec:normalizedlength}, Kisin--Mocz define 
\[
    \mu_{\max}(\mf M) := \max_{i} v(a_i)
\]
and
\[
    \mc H(\mf M) := \{ v(a_1),\ldots, v(a_n)\}
\]
as a multiset \cite[3.1.7]{kisinmocznorthcott}.

\subsection{Nygaard filtration and $\mf M_{\dR}$}\label{subsec:nygaard}
In order to study the change in height under an isogeny of motives, we will use a Nygaard filtration on torsion BKF modules of the form $\mf M/\mf M' \in \Mod^{\vphi, t}_{/\Ainf}$, where $\mf M, \mf M' \in \Mod^{\vphi}_{/\Ainf}$ are finite free BKF modules.

\begin{defn}
    [Nygaard filtration on finite free BKF modules] Let $\mf M \in\Mod^{\vphi, \le w}_{/\Ainf}$. We can define a filtration on $\vphi^* \mf M$ by
    \[
        \Fil_{\mc N}^r \vphi^* \mf M = \{ x \in \vphi^* \mf M \; | \; \vphi(x) \in \vphi(\xi)^r \mf M \}.
    \]
\end{defn}

\begin{rmk}
    When $\mf M \in \Mod^{\vphi, \le w}_{/\Ainf}$, i.e.~when the cokernel of $\vphi$ on $\mf M$ is $\vphi(\xi)^w$-torsion, observe that $\vphi\left( \Fil_{\mc N}^w \vphi^* \mf M \right) = \vphi(\xi)^w \mf M \subset \mf M$.
\end{rmk}

This is the base change of the Nygaard filtration defined for Breuil-Kisin modules in \cite[\S 7.1]{koshikawa2015heights} if we base change under the composed map $\mf S= W(k)[[u]] \xto{f} \Ainf \xto{\vphi} \Ainf$ sending $u \mapsto [\pi^\tilt]^p$, which is the map used to relate Breuil--Kisin to Breuil--Kisin--Fargues modules in \ref{subsubsec:basechangefromBKtoBKF}. The identification of the two Nygaard filtrations is due to the following lemma.

\begin{lemma}
	Let $E(u)$ be an Eisenstein polynomial in $W(k)[u]$, and let $f: W(k)[[u]] \to \Ainf$ be the ring map determined by $u \mapsto [\pi^\tilt]$, where $\pi$ a root of $E(u)$. Then $f(E(u))$ is a generator of the principal ideal $\ker \theta \subset \Ainf$. 
\end{lemma}
\begin{proof}
	Since $E(u)$ is Eisenstein, it can be written in the form $E(u) = u^n + p^2(a_{n-1}u^{n-1} + a_{n-2}u^{n-2} + \cdots + a_1 u) + pv$ for some $n$, where $a_i \in W(k)$ and $v \in W(k)^\times$. If we write
\begin{align}
	f(E(u)) = \sum_{i\ge 0} [b_i]p^i,\qquad b_i \in O_C^\flat,
\end{align}
	then we see that $b_0 = (\pi^\flat)^n \in \mcO_C^\flat$ and $b_1 = v \in (\mcO_C^\flat)^\times$. In particular, we can use \cite[Prop.~4.4.3(2)]{brinon2009cmi} to conclude that $f(E(u))$ is a generator of $\ker \theta$. 
\end{proof}
    

\subsection{Filtered lattice in de Rham cohomology}
The Nygaard filtration above induces a filtered lattice which will be useful in analyzing the change in Kato--Koshikawa height under isogenies.

\begin{defn}
    For an effective BK module $\mf M \in \Mod^{\vphi}_{/\mf S}$ coming from a $\Gal_K$-stable $\Zp$-lattice $T$ in a crystalline representation, define $\mf M_{\dR} := \vphi^* \mf M \otimes_{\mf S} \mcO_K$, under the usual map $\mf S\to \mcO_K$ where $u \mapsto \pi$. The image of $\Fil^* \vphi^*\mf M$ in $\mf M_{\dR}$ will define a filtration on $\mf M_{\dR}$.
\end{defn}

\begin{remark}
    One has the containment $\mf M_{\dR}\subset D_{\dR}(T[1/p])$, and the filtration on $\mf M_{\dR}$ agrees with the filtration on $D_{\dR}(T[1/p])$ \cite{koshikawa2015heights}.
\end{remark}

\subsubsection{}
Continuing the notation from \S \ref{subsubsec:basechangefromBKtoBKF}, when $\mf M$ is a BKF module, we can also define a filtered $\mcO_C$-module $\mf M_{\dR}$ to be compatible with the definition for BK modules.
 $\Fil^i \vphi^* \mf M_{BK}$ is the pullback of the diagram $\vphi^* \mf M_{BK} \xto{\vphi}\mf M_{BK} \xleftarrow{\supseteq} E(u)^i \mf M_{BK}$. 
 Since the map $\vphi\circ f: \mf S \to \Ainf$ is flat, and because $f(E(u))$ generates $\ker\theta = \xi\Ainf$, we can tensor the pullback square by $\vphi\circ f: \mf S \to \Ainf$ to get another pullback square. 
 In other words, $(\Fil^i \vphi^* \mf M_{BK}) \otimes_{\mf S, \vphi \circ f}\Ainf = \Fil^i \vphi^* \mf M_{BKF}$, where the filtration on $\vphi^* \mf M_{BKF}$ is defined to be $\Fil^i \vphi^* \mf M_{BKF} := \vphi^{-1}(\vphi(\xi)^i \mf M_{BKF})$.

To get to the $\OC$-lattice, recall that we get the $\mcO_K$-module $(\mf M_{BK})_{\dR}$ as $(\mf M_{BK})_{\dR} = \vphi^* \mf M_{BK} \otimes_{\mf S, \theta\circ f} \mcO_K$, where we abuse notation to denote $\theta\circ f: \mf S \to \Ainf \to \mcO_C$ as $u \mapsto \pi$. In particular, 
\begin{align}
    (\mf M_{BK})_{\dR} \otimes_{\mcO_K}\OC 
    &= (\vphi^* \mf M_{BK} \otimes_{\mf S, \vphi \circ f} \Ainf) \otimes_{\Ainf, \theta \circ \vphi^{-1} } \mcO_C \\
    &= \vphi^*\mf M_{BKF} \otimes_{\Ainf, \theta \circ \vphi^{-1}} \OC \\
    &=\vphi^*\mf M_{BKF} \pmod{\vphi(\xi)}.
\end{align}
In other words, if we define
\begin{align*}
    (\mf M_{BKF})_{\dR} := \vphi^*\mf M_{BKF} \otimes_{\Ainf, \theta \circ \vphi^{-1}} \OC,
\end{align*}
then we get the compatibility $(\mf M_{BK})_{\dR}\otimes_{\mcO_K}\mcO_C = (\mf M_{BKF})_{\dR}$.

\begin{rmk}\label{rmk:quotfil}
    After the present author noticed that the definition of the Nygaard filtration on torsion BK modules in \cite[\S 7.1]{koshikawa2015heights} does not behave will in short exact sequences, T.~Koshikawa suggested to use the following quotient filtration instead. 
\end{rmk}

\begin{defn}
    When $\mf M/\mf M'$ is a torsion Breuil--Kisin--Fargues module, one can define a filtration $\Fil'$ on $\vphi^*(\mf M/\mf M')$ such that $\Fil'^i \vphi^*(\mf M / \mf M')$ is precisely the quotient filtration from $\Fil^i_{\mcN}\vphi^* \mf M$, so 
    \[
        \Fil'^i \vphi^*(\mf M/\mf M') = \frac{\Fil^i_{\mcN} (\vphi^* \mf M) + \vphi^* \mf M'}{\vphi^* \mf M'} \cong \frac{\Fil^i_{\mcN} \vphi^* \mf M}{\Fil^i_{\mcN} \vphi^* \mf M \cap \vphi^* \mf M' }.
    \]
\end{defn}

Since $\Fil^i_{\mcN} \vphi^* \mf M \cap \vphi^* \mf M' = \Fil_{\mc N}^i\vphi^* \mf M'$, we have that for all $i$, the following sequence is exact.
\begin{align} \label{eqn:NygaardSES}
    0 \to \Fil^i_{\mcN} \vphi^* \mf M' \to \Fil^i_{\mcN} \vphi^* \mf M \to \Fil'^i \vphi^* (\mf M/\mf M') \to 0
\end{align}
Moreover, the same is true for the graded pieces $\gr^i_{\mc N}$ and $\gr'^i$. 

\subsubsection{} \label{subsubsec:filtrations}
For free Breuil-Kisin-Fargues modules $\mf M$, we can define a filtration $F^\bullet$ on $\Fil^r_{\mcN} \vphi^* \mf M$ as in \cite[Lemma 7.5]{koshikawa2015heights} such that $F^0 \Fil^r_{\mcN} \vphi^* \mf M= \Fil^r_{\mcN} \vphi^* \mf M$ and $F^{r+1} \Fil^r_{\mcN} \vphi^* \mf M = \Fil^{r+1}_{\mcN} \vphi^* \mf M$, namely
\[
    F^j (\Fil^r \vphi^* \mf M) := (\vphi^* \mf M \cap \vphi(\xi)^{r+1} \mf M) + (\vphi(\xi)^j \vphi^* \mf M \cap \vphi(\xi)^r \mf M).
\]
This induces a filtration $F^\bullet$ on $\gr^r_{\mcN} \vphi^* \mf M$. 
We also get an injection 
\[
    F^j \gr^r_{\mcN} \vphi^* \mf M' \into F^j \gr^r_{\mcN} \vphi^* \mf M 
\]
which induces a quotient filtration $F'^\bullet$ on $\gr'^r\vphi^* (\mf M/\mf M')$. We can check by construction of $F^\bullet$ that we also have an injection $F'^{j+1}\gr'^r\vphi^* (\mf M/\mf M') \into F'^{j}\gr'^r\vphi^* (\mf M/\mf M')$. This means that for all $r$ and $j$, there is a short exact sequence
\begin{align}\label{eqn:sesgrjgrr}
    0 \to \gr_F^j \gr^r_{\mc N} \vphi^* \mf M' \to \gr_F^j \gr^r_{\mc N} \vphi^* \mf M \to \gr_{F'}^j \gr'^r \vphi^*( \mf M/\mf M') \to 0.
\end{align}

Also, by definition of $F^\bullet$, there is a natural isomorphism
\[
\gr_F^j \gr^r_{\mc N} \vphi^* \mf M \cong \gr^{r-j}\mf M_{\dR},\]
where $\mf M_{\dR} :=\vphi^* \mf M \otimes_{\Ainf, \tilde{\theta}} \mcO_C$, where the base change map $\tilde{\theta}:\Ainf \to \mcO_C$ has kernel $(\vphi(\xi))$ \cite[Lemma 7.5]{koshikawa2015heights}.

Because (\ref{eqn:NygaardSES}) is exact, and its third term is $\vphi(\xi)$-torsion free, we also get the following short exact sequence for all $r$.
\[
    0 \to \Fil^r \mf M'_{\dR} \to \Fil^r \mf M_{\dR} \to \Fil'^r (\mf M/\mf M')_{\dR} \to 0
\]
Moreover, one can check that the filtration on $\vphi^*(\mf M/\mf M')$ induces an honest filtration on $(\mf M/\mf M')_{\dR} := \vphi^*(\mf M/\mf M') \otimes \mcO_C$, i.e.~the maps from $\Fil'^{i+1} \to \Fil'^i$ are injective, by checking that $\Fil'^i/\Fil'^{i+1} \cong (\Fil'^i / \vphi(\xi))/ (\Fil'^{i+1} / \vphi(\xi))$. Therefore for all $r$, we get short exact sequences 
\begin{align}\label{eqn:sesgrdR}
    0 \to \gr^r \mf M'_{\dR} \to \gr^r \mf M_{\dR} \to \gr'^r (\mf M/\mf M')_{\dR} \to 0.
\end{align}
Combining (\ref{eqn:sesgrjgrr}) and (\ref{eqn:sesgrdR}), we have
\[
    \gr_{F'}^j \gr'^r \vphi^*( \mf M/\mf M') \cong \gr^{r-j} \mf M_{\dR}/ \gr^{r-j}\mf M'_{\dR} \cong \gr'^{r-j} (\mf M/\mf M')_{\dR}.
\]

%% file: motivehts.tex
\section{Heights on motives}
\label{section:heightsmotives}

\subsection{Motives as systems of realizations}
In this section, we will use the same definitions as used by \cite{koshikawa2015heights}. Roughly, a motive is the data of the $w$th Betti, de Rham, and \'etale cohomologies of a smooth projective variety $X$ over a number field $F$, all of which are compatible under various comparison isomorphisms with their various structures (e.g.~polarized Hodge structure, filtration, or Galois representation). 

Rather than directly working with the category of motives when studying heights in isogeny classes, we instead work with varying choices of integral structures on a fixed system of realizations \cite[\S 7]{deligne1989groupe} coming from a single Grothendieck motive. A system of realizations is a cohomological avatar of a motive, remembering the various cohomological realizations of a motive and the extra data and relations among them. 
\begin{defn}
    Let $F$ be a number field. A \emph{system of realizations} of weight $w$ defined over $F$ is the tuple ${\bf M} = (M_{B,\sigma}, M_{\dR}, M_{\et}, i_{\sigma,\dR}, i_{\sigma, B,\et}, \{i_{v,\et,\dR}\}_v)$, where
\begin{itemize}
    \item for each $\sigma : F \into \C$, the Betti realization $M_{B,\sigma}$, which is a polarized $\Q$-Hodge structure of weight $w$,
    \item the de Rham realization $M_{\dR}$, which is a filtered $F$-vector space,
    \item the adelic \'etale realization $M_{\et}$, which is an $\A^f_{\Q}$-module with a continuous $\Gal_F$-action,
    \item $i_{\sigma,\dR}: M_{B,\sigma} \otimes_{\Q} \C \xto{\sim} M_{\dR} \otimes_{F,\sigma} \C$,
    \item $ i_{\sigma, B,\et}: M_{B,\sigma} \otimes_{\Q}\A^f_{\Q} \xto{\sim} M_{\et} $
    \item For every embedding $v: \ol{F} \to \ol{\Q}_p$ extending a place $v \mid p$ of $F$, an isomorphism $i_{v,\et,\dR}:M_{p} \otimes_{\Qp} \BdR \xto{\sim} M_{\dR} \otimes_{F, v} \BdR$ respecting filtrations and $\Gal(\ol{F}_v/F_v)$-actions. Here, $M_{p}$ denotes the $p$-adic projection of $M_{\et}$. 
    \end{itemize}
\end{defn}

\begin{rmk}
    One might consider the $\A_f$-topology to be dependent on some choice of integral structure on $M_{\et}$. However, any choice of complex embedding $\sigma$ and any choice of $\Z$-lattice inside the $\Q$-Hodge structure $M_{B,\sigma}$ will define a $\hat{\Z}$-integral structure on $M_{\et}$ under the comparison isomorphism $i_{B,\et}$, which will allow one to define the restricted product structure on $\prod_p M_p$ to make $M_{\et}$ an $\A_f$-module. Regardless of the choice of $\Z$-lattice, the topology on $M_{\et}$ will be the same. 
\end{rmk}

\begin{defn}[\cite{koshikawa2015heights}, 4.2]\label{defn:goodsemistablered}
    Fix a finite place $v$ of $F$ above $p$, and let ${\bf M}$ be a system of realizations defined over $F$. 
    \begin{enumerate}
        \item ${\bf M}$ has \emph{good reduction} at $v$ if the $\ell$-adic \'etale realization $M_\ell$ is unramified at $v$ for all $\ell \neq p$, and the $p$-adic \'etale realization $M_p$ is a crystalline representation at $v$. 
        \item ${\bf M}$ has \emph{semistable reduction} at $v$ if the inertia $I_v$ acts unipotently on $M_\ell$ for any $\ell \neq p$, and the $p$-adic \'etale realization $M_p$ is a semistable representation at $v$. 
    \end{enumerate}
\end{defn}

\begin{remark}\label{rmk:geomrepsdeRham}
    Suppose we have smooth proper variety $X/F$.
    Faltings showed that at places $v \mid p$ of $F$, the Galois representation on $p$-adic \'etale cohomology of $X_{F_v}$ is de Rham, and when $X$ has good reduction at $v$, the representation is crystalline \cite{faltings1989crystalline}.
    Any smooth proper $X/F$ will have good reduction at almost all places of $F$. At the finitely many places of bad reduction, each of those de Rham representations are potentially semistable, as proved by Berger \cite{berger2002representations} and separately by Andr\'e, Mebkhout, and Kedlaya. Therefore we may replace $F$ with a finite extension so that the system of realizations corresponding to the $w$th cohomology of $X$ has semistable reduction everywhere as defined above. 
\end{remark}

\begin{defn}
    \label{def:Koshikawamotive}
    A \emph{Koshikawa motive} $M = ({\bf M},\hat{T})$ of weight $w$ defined over a number field $F$ is a system of realizations ${\bf M}$ over $F$ associated to a Grothendieck motive over $F$, together with the choice of a $\Gal_F$-stable $\hat{\Z}$-submodule $\hat{T} \subset M_{\et}$ which is open in the adelic topology. We will always assume that ${\bf M}$ has semistable reduction everywhere. 
\end{defn}

\begin{rmk}
    In what follows, one might ask whether the Kato--Koshikawa height could be defined for only the data of a system of realizations with a choice of integral structure, i.e.~discarding the assumption that the data arises as the image of some actual correspondence of smooth projective varieties defined over $F$. However, the construction of the Kato--Koshikawa height requires the various $\mcO_{F_v}$-lattices in $M_{\dR}\otimes_F F_v$ produced via the choice of $\hat{T}$ and integral $p$-adic Hodge theory to be coherent. That is, for almost all $v$ of $F$, we need the $\mcO_{F_v}$-lattice to be the base-change of a single $\mcO_F$-lattice in $M_{\dR}$. In \cite{koshikawa2015heights}, Koshikawa uses some results that follow from having integral models of the varieties involved in producing the motive to guarantee coherence of the $\mcO_{F_v}$-lattices. Without putting some unnatural extra conditions on the system of realizations, we do not know how to replace this use of geometry in ensuring the height is well-defined. Therefore we will require that a Koshikawa motive must have an underlying system of realizations coming from geometry. 
\end{rmk}

\td{When you come back to this, it would be good to explain the relation between actual motives and the system of realizations that I am using. Is there a fully faithful functor?}

\subsection{Construction of a lattice in the de Rham realization}
Suppose that $M = ({\bf M},\hat{T}) $ is a Koshikawa motive such that ${\bf M}=H^w(X)$ for a smooth projective variety $X$ defined over $F$, and $\hat{T}$ agrees with the natural $\Z_p$-integral structure on $M_p$ for almost all $p$.

\begin{rmk}
    In defining the height, we consider the case of ${\bf M} = H^w(X)$. If ${\bf M}$ is a system of realizations coming from a motive which is not $H^w(X)$ for some smooth projective variety $X/F$ but rather the weight $w$ piece of the projection from some cycle in $X \times X$, we can take an integral model of the cycle over $\mcO_{F,S}$ for some $S$ large enough so that we get a natural torsion-free $\mcO_{F,S}$-structure on the de Rham realization $M_{\dR}$ and its graded pieces. 
\end{rmk}

Analogously to the N\'eron model used to define the Faltings height, one wants to choose a good $\mcO_F$-integral structure on the motive's de Rham realization.
In \cite[\S 4.1]{koshikawa2015heights}, Koshikawa uses integral $p$-adic Hodge theory to produce a filtered $\mcO_{F_v}$-lattice 
$$
M_{\dR}(T_p)_{v} \subset M_{\dR}\otimes_{F} F_v
$$
from the $\Gal_F$-action on the $\Zp$-module $T_p$ restricted to $\Gal_{F_v}$ for all places $v \mid p$ of $F$ and all $p$. 
He then uses a comparison to log de Rham cohomology to show that these various filtered $\mcO_{F_v}$-lattices will combine to give a filtered $\mcO_F$-lattice in $M_{\dR}$. He provides an alternate justification in \cite[\S 5.2]{koshikawa2016hodge} using the new technique of $\Ainf$-cohomology developed in \cite{BMS2018integral}, which we will explain here.

\subsubsection{} \label{subsubsec:Sassumptions}
We have a smooth integral model $\mc X \to \Spec \mcO_{F,S}$ of $X$ defined over the $S$-integers of $F$. We can enlarge $S$ so that the following hold:
\begin{enumerate}[(1)]
\item the finitely generated $\mcO_{F,S}$-modules $H^w_{\dR}(\mc X/\mcO_{F,S})$ and $H^{w+1}_{\dR}(\mc X/\mcO_{F,S})$ are torsion-free,
\item the finitely generated $\mcO_{F,S}$-modules $H^{w-i}(\mc X , \Omega^i_{\mc X/\mcO_{F,S}})$ are torsion-free for all $i$, and
\item for all $p$ such that some $v \not\in S$ satisfies $v \mid p$, the choice of $\Zp$ lattice in $M_p$ coming from $\hat{T}$ is the natural one: $T_p = H^w_{\et}(X_{\ol{F}}, \Zp)\subset M_p$. 
\end{enumerate}

Let $\mf X$ be a proper smooth formal scheme over $\mcO_C$ with generic fiber $X_C$.
There is a perfect complex of $\Ainf$-modules $R\Gamma_{\Ainf}(\mf X)$ such that all cohomology groups are finitely presented BKF modules. This $\Ainf$-cohomology complex is a way to geometrically construct the BKF modules associated to the \'etale cohomology of $X$, and satisfies various comparison isomorphisms to de Rham, \'etale, and crystalline cohomology. 
The construction and comparison theorems are originally proved by Bhatt--Morrow--Scholze in \cite{BMS2018integral}. Since then, Bhatt and Scholze have shown that these results can be proven using the formalism of the prismatic site \cite{bhatt2022prisms}, \cite[Theorem 9.1]{guo2024prismatic}.
One may then combine the $\Ainf$-cohomology comparison theorems with Fargues' classification of finite free BKF modules to prove the following, under a torsion-free assumption. 

\begin{theorem}[\cite{BMS2018integral}, 14.5(iii)]
    Let $\mf X$ be the formal scheme as above, and fix $i \ge 0$. Assume that $H^i_{\dR}(\mf X)$ is torsion-free. Then $H^i_{\Ainf}(\mf X):= H^i(R\Gamma_{\Ainf}(\mf X))$ is a finite free BKF module. In particular, it is the BKF module associated to $H^i_{\et}(X_C, \Zp)$.
\end{theorem}

The $\Ainf$-de Rham comparison theorem is the following form \cite[Theorem 1.8(ii)]{BMS2018integral}:
\[
R\Gamma_{\Ainf}(\mf X) \otimes^{\mb L}_{\Ainf} \mcO_C \simeq R\Gamma_{\dR}(\mf X).
\]
With the choice of $S$ satisfying the torsion-free conditions above, for all $v \not\in S$, we can thus conclude that
$$H^w_{\Ainf}(\mf X) \otimes_{\Ainf}\mcO_C \cong H^{w}_{\dR}(\mc X/ \mcO_{F,S}) \otimes_{\mcO_{F,S}, v} \mcO_C .$$
The tensor product on both sides are non-derived. In other words, we have found a $\mcO_C$-lattice inside de Rham cohomology $M_{\dR} \otimes_{F,v} C$. Moreover, the above theorem shows that $H^w_{\Ainf}(\mf X)$ depends only on $X_C$ and not on the choice of integral model.

Let's see how this produces at least a choice of $\mcO_{F,S}$-lattice inside $M_{\dR}$ that is independent of choice of integral model of $X$. For all $v \not\in S$, we get the $\mcO_{F_v}$-module structure on $M_{\dR} \otimes_F F_v$ as
\[
    (H^w_{\Ainf}(\mf X) \otimes_{\Ainf}\mcO_C) \cap( H^w_{\dR}(X/F)\otimes F_v) = H^w_{\dR}(\mc X/ \mcO_{F,S}) \otimes_{\mcO_{F,S}} \mcO_{F_v}. 
\]
So we have found an $\mcO_{F,S}$-lattice in $M_{\dR}= H^w_{\dR}(X/F)$ such that for all $v \not\in S$, the base-change to $\mcO_{F_v}$ of the lattice is independent of the choice of integral model, and therefore so is this $\mcO_{F,S}$-lattice.

\subsubsection{}
Now that we have at least a good choice of an $\mcO_{F,S}$-lattice in $M_{\dR}$. Suppose that for each $v \in S$, we could construct an $\mcO_{F_v}$-lattice $M_{\dR}(T_p)_v$ in $M_{\dR}\otimes_F F_v$. Then we can combine these data to get a global $\mcO_F$-lattice $M_{\dR}(\hat{T}) \subset M_{\dR}$ locally of the correct rank.

With the choice of $\Gal_{F}$-stable $\Zp$-lattice $T_p \subset M_p$ in the semistable representation $M_p$, we can associate a finite free BK module $\mf M_v(T_p)$ as follows. Choose a place $v \mid p$ of $F$. Fix an extension of $v$ to $\ol{F}$, i.e.~an embedding $v: \ol{F} \to C$ extending $v: F \into C$. Choose a uniformizer $\pi$ of $F_v$ and compatible $p^n$th roots $\pi^{1/p^n}$ in $\ol{F_v}$ for all $n$. Write $K := F_v$, and let $K_\infty := \bigcup_{n} F_v(\pi^{1/p^n})$. This defines inclusions of absolute Galois groups $\Gal_{K_\infty} \subset \Gal_{K} \subset \Gal_F$. One can then use the equivalence between in \cite[Lemma 2.1.15]{kisin2006crystalline} to define $\mf M_v(T_p)$. 

One may be concerned that different choices of $\pi^{1/p^n}$ defining $K_{\infty}$ may result in different BK modules $\mf M_v(T_p)$, and thus not having a well-defined way to construct the $\mcO_F$-lattice in $M_{\dR}$ used to define the height. However, the resulting $\mcO_{F_v}$ lattice in $M_{\dR} \otimes_F F_v$ will be independent of the choice of $\pi^{1/p^n}$ \cite[Theorem 1.0.1]{liu2018compatibility}.

Given a BK module $\mf M_v(T_p)$, one has a natural map $\vphi^* \mf M_v(T_p) \otimes_{\mf S} \mcO_K \to D_{\dR}(M_p) = M_{\dR} \otimes_F F_v$ \cite[\S 2.3, \S 4.2]{liu2018compatibility}. Here, $\mf S \to \mcO_K$ is given by $u \mapsto \pi$, where $\pi$ is a root of the Eisenstein polynomial $E(u)\in \mf S$.
The image of the Nygaard filtration on $\vphi^* \mf M$ defines a filtration on $\vphi^* \mf M \otimes_{\mf S}\mcO_K$. Define the filtered $\mcO_{F_v}$-lattice
$$
M_{\dR}(T_p)_v := \vphi^* \mf M \otimes_{\mf S}\mcO_{F_v}.
$$

We can then intersect the filtered $\mcO_{F,S}$-lattice $H^w_{\dR}(\mc X/\mcO_{F,S})$ in $M_{\dR}$ and the filtered $\mcO_{F,v}$-lattices $M_{\dR}(T_p)_v$ for the finitely many $v \in S$ to get a filtered $\mcO_F$-lattice $M_{\dR}(\hat{T})$ in $M_{\dR}$.

Now we have a filtered $\mcO_F$-lattice $M_{\dR}(\hat{T})$ in $M_{\dR}$. For the $v \not\in S$, the torsion-free assumption (2) from \ref{subsubsec:Sassumptions} on the relative Hodge cohomology means that the relative Hodge-to-de Rham spectral sequence computing $H^w_{\dR}(\mc X/ \mcO_{F,S})\otimes \mcO_{F_v}$ degenerates at the $E_1$ page, because the $E_1$ differentials are 0 after tensoring with $F_v$.
\td{Ask someone about this to double check.}
Therefore, for $v \not\in S$, $(\gr^r M_{\dR}(\hat{T})) \otimes \mcO_{F_v}$ is already torsion-free. Define $\gr^r( M_{\dR}(\hat{T}))_{\fr}$ to be the unique $\mcO_F$-module such that 
\[
\gr^r( M_{\dR}(\hat{T}))_{\fr} \otimes \mcO_{F_v} = \begin{cases}
    \gr^r M_{\dR}(\hat{T}) \otimes \mcO_{F_v} & v \not\in S \\
    (\gr^r M_{\dR}(\hat{T}) \otimes \mcO_{F_v})/\text{torsion}& v \in S.
\end{cases}
\]
This defines an $\mcO_F$-lattice $\gr^r( M_{\dR}(\hat{T}))_{\fr}$ in $\gr^r(M_{\dR})$. 

\begin{remark}
    In fact, the process outlined in \S 2.6 of \cite{koshikawa2015heights} to produce $M_{\dR}(T)_v$ would also work for the places $v \not\in S$ [\textit{ibid.}, Theorem 3.5(3)], but a priori it would not be clear that the intersection of all these lattices with $M_{\dR}$ would produce an $\mcO_F$-submodule of the correct rank. Some further argument is needed, either using $\Ainf$-cohomology as above or log de Rham cohomology as is used in \cite{koshikawa2015heights}, to show that these $\mcO_{F_v}$-lattices are compatible. 
\end{remark}

\td{Have to explain why the $r$th filtered pieces also fit together to get $O_F$-modules, and the graded pieces also. Also address the $N_{\fr}$ notation for a module $N$. Actually, seems like we define it in 4.3.1.}

\nts{Stacks project 0FM0 shows that relative algebraic de Rham cohomology for proper smooth morphisms is well-behaved under any base change, so that we can get the de Rham cohomology over $\mcO_{F,S'}$ just by localizing from $\mcO_{F,S}$. Also, flat base change allows us to get the new sheaf cohomology just by localizing. Also, since we are in a proper setting over a Noetherian base, both cohomology groups are finitely generated.}


\subsection{Kato--Koshikawa height of a Koshikawa motive}

Using the $\mcO_F$-lattices $\gr^r(M_{\dR}(\hat{T}))_{\fr}$ for all $r$ from above, we can now define an
$\mcO_F$-integral version of the Griffiths line bundle, which Koshikawa then uses to define the height in a way analogous to the Faltings height.
\begin{defn}[\cite{koshikawa2015heights}, \S 4.2]
    Define the locally free rank 1 $\mcO_F$-module 
    \[
    \mc L ({ M}) := \bigotimes_{r\in \Z}(\det (\gr^rM_{\dR}(\hat{T}))_{\fr})^{\otimes r}.
    \]
    The tensor products and determinants are all taken over $\mcO_F$. 
\end{defn}

We actually do need to keep track of the torsion of the $\gr^r(M_{\dR}(\hat{T}))$ in the definition of the height, to ensure that the height change is well-behaved under isogenies. 
\begin{defn}
    Let $\Sigma_F^{\text{fin}}$ denote all finite places of $F$. For a finitely generated $\mcO_{F_v}$-module $M$, let $M_{\tor}$ denote the torsion submodule of $M$. Define 
    \[
        \# \mc L({ M})^{\tor} := \prod_{v \in \Sigma_F^{\text{fin}}} \prod_{r \in \Z} \# \gr^r(M_{\dR}(\hat{T})_v)_{\tor}
    \]
    We use the superscript $\tor$ on the left-hand side so as not to suggest that we are taking the torsion submodule of the torsion-free $\mcO_F$-module $\mc L({ M})$. 
\end{defn}

\subsubsection{}\label{ssc:normforheight}
Fix an embedding $\sigma: F \into \C$, and let $\mc L({ M})_{\sigma}:= \mc L({ M}) \otimes_{\mcO_F, \sigma} \C$. Let $c \in \Aut(\C)$ denote complex conjugation. Using the Hodge symmetry on the Hodge decomposition, there is a natural isomorphism
\begin{align}
    \mc L({ M})_{\sigma} \otimes \mc L({ M})_{c \circ \sigma}
    & \cong 
    \left(\bigotimes_{r \in \Z} \det \gr^r (M_{\dR} \otimes_{F,\sigma} \C )^{\otimes r}\right) 
    \otimes
    \left(\bigotimes_{r \in \Z} \det \gr^{r} (M_{\dR} \otimes_{F,c \circ \sigma} \C )^{\otimes r}\right) \\
    &\cong 
    \left(\bigotimes_{r \in \Z} \det \gr^r (M_{\dR} \otimes_{F,\sigma} \C )^{\otimes r}\right) 
    \otimes
    \left(\bigotimes_{r \in \Z} \det \gr^{w-r} (M_{\dR} \otimes_{F, \sigma} \C )^{\otimes r}\right) \\
    &\cong \bigotimes_{r \in \Z}\left( \det \gr^r (M_{\dR} \otimes_{F,\sigma} \C)\right)^{\otimes w}\\
    &\cong \det (M_{\dR}\otimes_{F,\sigma} \C)^{\otimes w}.
\end{align}

The Betti-\'etale comparison isomorphism, together with the integral structure given by $\hat{T}$, defines a $\Z$-module inside $\det (M_{B} \otimes_{\Q}\C )^{\otimes w}$. Choosing a generator $s_{B,\sigma}$ of this free rank 1 $\Z$-module and using the Betti-de Rham comparison isomorphism $\det(M_{\dR} \otimes_{F}\C)^{\otimes w} \cong \det(M_{B}\otimes_{\Q}\C )^{\otimes w}$ then allows us to define an isomorphism 
of 1-dimensional $\C$-vector spaces
\begin{align*}
    \det(M_{\dR}(\hat{T}))^{\otimes w} \otimes_{\mcO_F}\C \cong \det(M_{\dR} \otimes_F \C)^{\otimes w} &\xto{\sim} \C\\
    zs_{B,\sigma} &\mapsto z.
\end{align*}
For a section $s \in \mc L({ M})$, the image of $s \otimes \ol{s}$ in $\det (M_{\dR}\otimes_{F,\sigma} \C)^{\otimes w}$ will be some multiple $zs_{B,\sigma}$ of the generator $s_{B,\sigma}$, for $z \in\C$. Then define 
\begin{align} \label{eqn:norm}
    \abs{s}_\sigma:= |z|^{1/2}_{\C},
\end{align}
where $|\cdot|_{\C}$ is the usual absolute value on $\C$. 

\begin{defn}
    Let ${ M}:= ({\bf M}, \hat{T})$ be a Koshikawa motive over $F$ such that ${ \bf M
    } = {H}^w(X)$ for $X$ a smooth projective variety defined over $F$. 
    Choose any $s \in \mc L({ M})$.
    The \emph{logarithmic Kato--Koshikawa height} is the sum
    \[
    h({ M}):= \frac{1}{[F:\Q]}\left( \log \#\mc L({ M})^{\tor} + \#( \mc L({ M}) / \mcO_F \cdot s ) - \sum_{\sigma: F \into \C} \log \abs{s}_\sigma  \right).
    \]
\end{defn}

\begin{remark}
    The height is independent of the choice of $s$. Indeed, any two choices $s,s'$ will satisfy $s = \alpha s'$ for $\alpha \in F^\times$, and the product formula for $\alpha$ will ensure the heights computed using $s$ and $s'$ are the same. 
\end{remark}

\begin{remark}[\cite{koshikawa2015heights}, Proposition 4.8]
    The height of a Koshikawa motive ${ M}$ defined over $F$ is independent of base change to a finite field extension of $F$.
\end{remark}

\subsection{Comparison to other heights}

\subsubsection{Arakelov height}
Consider smooth projective scheme $\mc X \to\Spec \mcO_F$. Consider the Koshikawa motive ${ M} = ({ \bf M}, \hat{T})$, where ${ \bf M} = {H}^w(\mc X_F)$ with integral structure $\hat{T}:= \prod_p H^w_{\et}(\mc X_{\ol{F}},\Zp)_{\fr}$. 
Then $\mc L({ M})$ is a metrized line bundle on $\Spec \mcO_F$ with Hermitian metric at the infinite place $\sigma: F \into \C$ given by (\ref{eqn:norm}), which is invariant under complex conjugation.
Let $\widehat{\deg}(\mc L({ M}))$ denote the normalized Arakelov degree of this metrized line bundle \cite[(2.1.15), \S 3.1.4]{BGS1994heights}.
\begin{prop}
    [\cite{koshikawa2015heights}, Proposition 4.10]
    The difference $\lvert h({ M}) - \widehat{\deg}(\mc L({ M})) \rvert$ is bounded by a constant that depends only on $\dim X$ and the Hodge numbers of ${ \bf M}$. 
\end{prop}

The Kato---Koshikawa height also recovers the Faltings height when applied to $M = { H}^1(A)$ when $A/F$ is an abelian variety.
\begin{prop}
    [\cite{koshikawa2015heights}]
    Let $h_{\Fal}(A)$ denote the stable Faltings height of the abelian variety $A$. Then $|h_{\Fal}(A) - h(M)|$ is bounded by a constant depending only on $\dim A$. 
\end{prop}


\nts{I am removing the section about heights in families, but I think I would need it to make the Shimura varieties argument.}

\subsection{Formula for height change under isogeny}


Let ${ M} = ({\bf M}, \hat{T})$ and ${ M}' = ({\bf M}, \hat{T}')$ be two pure Koshikawa motives over $F$ defined by putting two different $\hat{\Z}$-structures on the same rational Grothendieck motive ${\bf M}$. 


\begin{defn}
    The data of ${ M}$, ${ M}'$ as above together with a finite-index inclusion $f: \hat{T}' \into \hat{T}$ of $\hat{\Z}$-modules is called an \emph{isogeny}. If the inclusion $f$ is $\Gal_{F'}$-equivariant for some extension $F'/F$, then the isogeny is said to be defined over $F'$. We will sometimes abuse notation and write $f: { M}' \to { M}$ for the isogeny. 
\end{defn}

\begin{remark}
    If an isogeny $f: { M}' \to { M}$ is such that $\hat{T}/f(\hat{T}')$ is a $p$-group, then we say that $f$ is a \emph{$p$-power isogeny}. Any isogeny defined over $F'$ can be factored into a finite composition of $p_i$-power isogenies, all defined over the same $F'$. The $p_i$ are the primes which divide $\# \hat{T}/f(\hat{T}')$. 
\end{remark}

\subsubsection{} \label{subsubsec:isogenysetup}
Let $f: { M}' \to { M}$ be an isogeny of pure Koshikawa motives as above which is defined over $F$. It suffices to write down a formula for $h({ M'}) - h({ M})$ in the case that $f$ is a $p$-power isogeny for some prime $p$. 
We may twist the motive such that the Hodge--Tate weights of the \'etale realization of ${ \bf M}$ are nonpositive, which allows us to assume that any BK modules associated to lattices in $M_p$ are effective.

Let $v \mid p$ be a finite place of $F$. Fix an extension of $v$ to $\ol{F}$, i.e.~an embedding $v: \ol{F} \to C$ extending $v: F \into C$. Choose a uniformizer $\pi$ of $F_v$ and compatible $p^n$th roots $\pi^{1/p^n}$ in $\ol{F_v}$ for all $n$. Write $K := F_v$, and let $K_\infty := \bigcup_{n} F_v(\pi^{1/p^n})$. This defines inclusions of absolute Galois groups $\Gal_{K_\infty} \subset \Gal_{K} \subset \Gal_F$. 

Using \Cref{thm:BKmodequivalence}, a $\Gal_{K_\infty}$-stable $\Zp$-lattices $T'_p \subset T_p$ inside the semistable rational representation $M_p$ corresponds to an inclusion of finite free effective BK modules $\mf M_v(T'_p) \into \mf M_v(T_p)$, where $\mf M_v(T'_p) \otimes_{\Zp} \Qp = \mf M_v(T_p) \otimes_{\Zp} \Qp$.

\begin{remark}
    \cite{koshikawa2015heights} uses a contravariant functor $\mf M(-): \Rep^{free}_{\Zp}(\Gal_{K}) \to \Mod^{\vphi}_{/\mf S}$. We will use the covariant one, but the two approaches differ only by taking duals on both sides \cite[Remark 11.2.5]{brinon2009cmi}.
\end{remark}

\subsubsection{}
    The following isogeny formula is due to \cite[Proposition 8.1]{koshikawa2015heights}. The original proof uses a slightly different definition of the Nygaard filtration on the torsion BK module $\mf M_v(T_p)/\mf M_v(T_p')$. See \Cref{rmk:quotfil}. For the sake of completeness, we rewrite the proof of Koshikawa's isogeny formula using the quotient Nygaard filtration defined in \S \ref{subsec:nygaard}, but the idea of the proof is otherwise entirely the same as the original one given by Koshikawa. 

\begin{proposition}\label{prop:isogenyformulaslope}
    Use the setup of a $p$-power isogeny ${ M}' \to { M}$ as in \ref{subsubsec:isogenysetup}. Recall from \Cref{def:slope} the definition of the slope $\mu(\mf N)$ of an effective torsion BK module $\mf N \in \Mod^{\vphi, t}_{/\mf S}$. Then 
    \begin{align*}
        h({ M}') - h({{M}})
        &= \log \#(T_p/T_p')  \left(  \frac{w}{2} - \nu(T_p, T_p')\right),
    \end{align*}
    where we define
    \begin{align*}
        \nu(T_p,T_p')&:= \frac{1}{[F:\Q]} \sum_{v \mid p} [F_v:\Qp]\:\mu\left(\frac{\mf M_v(T_p)}{\mf M_v(T_p')} \right).
    \end{align*}
\end{proposition}

\begin{proof}
    Suppose that for each place $v \mid p$, the inclusion $T_p' \into T_p$ of $\Zp[\Gal_{F_v}]$-modules gives rise to inclusions $\mf M_v(T_p') \to \mf M_v(T_p)$.
    Recall that 
    \[
    \mf M_v(T_p)_{\dR}:= \vphi^*   \mf M_v(T_p) \otimes_{\mf S} \mcO_{F_v} = M_{\dR}(T_p)_v
    \]
    and the analogously defined $\mf M_v(T_p')_{\dR}$
    are filtered $\mcO_K$-modules.
    By definition of $\mc L({ M})$, one has that 
    \begin{align}\label{eqn:LMLMprimeratiointermsofdR}
        \frac{\# \mc L({ M})_{\tor}}{\# \mc L({ M}')_{\tor}} \# \frac{\mc L({ M})}{\mc L({ M}')} 
        = \prod_{v \mid p} \prod_{r \in \Z} \# \left(  \frac{\gr^r (\mf M_v)_{\dR}}{\gr^r (\mf M'_v)_{\dR}}\right)^r.
    \end{align}
    The inclusion ${ M}' \to { M}$ induces an inclusion $\mc L({ M}') \subset \mc L({ M})$. The $\Z$-structures $\Z \cdot s'_{B,\sigma} \subset (\det M_{\dR}(\hat{T}'))^{\otimes w} \otimes_{\mcO_F, \sigma} \C$ and $\Z \cdot s_{B,\sigma}\subset (\det M_{\dR}(\hat{T}))^{\otimes w}\otimes_{\mcO_F, \sigma} \C$ from the respective Betti-\'etale comparisons satisfy 
    \[
    \# (\Z s_{B,\sigma })/(\Z s'_{B,\sigma }) = \# (\hat{T}/\hat{T}')^w. 
    \]
    These generators were used to define the norms $\abs{\cdot}_{\sigma}$ and $\abs{\cdot}_{\sigma}'$ which show up in the definition of $h({ M})$ and $h({ M}')$, respectively. 
    For $s \in \mc L({ M'}) \subset \mc L({ M})$, the relation above implies that
    \begin{align}\label{eqn:normisogeny}
        \abs{s}_\sigma = \# (\hat{T}/\hat{T}')^{w/2} \abs{s}_{\sigma}' 
    \end{align}
    Together, (\ref{eqn:LMLMprimeratiointermsofdR}) and (\ref{eqn:normisogeny}) for a $p$-power isogeny show that
    \[
        h({ M}') - h({{M}})
        = - \frac{1}{[F:\Q]}\sum_{v \mid p} \sum_{r \in \Z} r \cdot \log \#  \left(  \frac{\gr^r (\mf M_v)_{\dR}}{\gr^r (\mf M'_v)_{\dR}}\right) + \frac{w}{2} \log \# (T_p/T_p').
    \]
    In order to use the filtrations in \ref{subsubsec:filtrations} to rewrite the change in height in terms of slope, we need to convert these BK modules to BKF modules, which involves a twist by Frobenius. As stated after (\ref{eqn:LWEdeg}), we have that for an effective torsion BK module $\mf N_{BK}$, 
    \begin{align}\label{eqn:lengthtonormlength}
        \log_p \# \left( \mf N_{BK} / \vphi^* \mf N_{BK} \right) = [K:\Qp] \lambda \left( \mf N_{BKF} / \vphi^* \mf N_{BKF}\right),
    \end{align}
    where $\mf N_{BKF} := \mf N_{BK} \otimes_{\mf S, \vphi \circ f} \Ainf$, where $\vphi\circ f(u) := [\pi^{\tilt}]^p$ as in \ref{subsubsec:basechangefromBKtoBKF}.

    From now on, for a fixed (and omitted) $v \mid p$, we will write 
    \begin{align}
        \mf M &:= (\mf M_v(T_p))_{BKF} = \mf M_v(T_p) \otimes_{\mf S, \vphi \circ f} \Ainf\\
        \mf M_{\dR} &:= M_{\dR}(T_p)_v \otimes_{\mcO_{F_v}} \mcO_C = (\vphi^* \mf M_v(T_p) \otimes_{\mf S} \mcO_{F_v}) \otimes_{\mcO_{F_v}}\mcO_C = \vphi^* \mf M \otimes_{\Ainf} \frac{\Ainf}{\vphi(\xi)},
    \end{align}
    and we define $\mf M'$, $\mf M'_{\dR}$ analogously. 
       From now on in the proof, we will be working with BKF modules. We hope this does not cause the reader too much confusion.

    Recall that $\Fil'^w \vphi^*(\mf M/ \mf M') = \vphi(\xi)^w \vphi^* (\mf M/ \mf M')$, and there are inclusions $\Fil'^w \vphi^*(\mf M/ \mf M') \subset \vphi^* (\mf M/\mf M') \subset \mf M/ \mf M'$. The second inclusion is an inclusion because $\vphi$ is injective, since the quotient $\mf M/ \mf M'$ is $\vphi(\xi)$-torsion free. By taking quotients, we then have the relation 
\begin{align}    
        \lambda \left(  \frac{\mf M/\mf M'}{\vphi^* \mf M/\mf M'}\right) 
        &= \lambda \left( \frac{\mf M/\mf M'}{\vphi(\xi)^w \mf M/ \mf M'}\right) - \sum_{r=0}^{w-1}\lambda( \gr'^r \vphi^*(\mf M/\mf M')) \\
        &= \lambda \left( \frac{\mf M/\mf M'}{\vphi(\xi)^w \mf M/ \mf M'}\right) - \sum_{r=0}^{w-1}\sum_{j=0}^{r}\lambda( \gr^j_{F'}\gr'^r \vphi^*(\mf M/\mf M')) \\
        &=  \lambda \left( \frac{\mf M/\mf M'}{\vphi(\xi)^w \mf M/ \mf M'}\right) - \sum_{r=0}^{w-1}\sum_{j=0}^{r}\lambda(\gr'^{r-j} \vphi^*(\mf M/\mf M'))\\
        &= \lambda \left( \frac{\mf M/\mf M'}{\vphi(\xi)^w \mf M/ \mf M'}\right) - \sum_{r=0}^{w}(w-r)\lambda\left( \frac{\gr^r \mf M_{\dR}}{\gr^r \mf M'_{\dR}}\right) \\
        &= w \cdot \lambda\left( \frac{\mf M_{\dR}}{\mf M'_{\dR}} \right) -  \sum_{r=0}^{w}(w-r)\lambda\left( \frac{\gr^r \mf M_{\dR}}{\gr^r \mf M'_{\dR}}\right)  \\
        &= \sum_{r=0}^w r \lambda\left( \frac{\gr^r \mf M_{\dR}}{\gr^r \mf M'_{\dR}}\right).
\end{align}
    The second-to-last equality comes from the fact that $\mf M/\mf M'$ is $\vphi(\xi)$-torsion free, and the last equality follows from the fact that $(\mf M/\mf M')_{\dR} = \mf M_{\dR}/\mf M'_{\dR}$, and that $\lambda(\mf M_{\dR}/\mf M'_{\dR})= \sum_{r=0}^w \lambda(\gr^r \mf M_{\dR}/ \gr^r \mf M'_{\dR})$. Note that $\gr^r \mf M_{\dR}/ \gr^r \mf M'_{\dR}$ comes from a finitely generated $\Zp$-module which is $p$-power torsion, so the relation (\ref{eqn:lengthtonormlength}) also holds for this module, and so $[F_v:\Qp]\lambda\left( \frac{\gr^r \mf M_{\dR}}{\gr^r \mf M'_{\dR}}\right) = \log_p \#\frac{\gr^r \mf M_{\dR}}{\gr^r \mf M'_{\dR}} $.
    Then we get
    \[
        h({ M}') - h({{M}})
        = \log \# (T_p/T_p')\left( \frac{w}{2} - \frac{1}{[F:\Q]}\sum_{v \mid p}  [F_v:\Qp] \frac{  \lambda \left(  \frac{\mf M/\mf M'}{\vphi^* \mf M/\mf M'}\right)    }{\log_p \# (T_p/T_p')} \right),
    \]
    which becomes the desired formula after using the definition of slope. 

\end{proof}

\subsubsection{}
Analogously to the isogeny formula given by Faltings in \cite{faltings1986finiteness}, the height change under isogeny has two contributions: one from the ``finite'' contribution to the height, namely from the $\mc L({ M})/(\mcO_F \cdot s)$ term, and the other from the ``infinite'' contribution to the height, namely the metric at each infinite place of $F$. These correspond to the following quantities for the Kato--Koshikawa height.
    \[
        \Deg_{\dR, v} := \# \left( \frac{\mf M_v(T_p)/
        \mf M_v(T'_p)}{\vphi^*(\mf M_v(T_p)/
        \mf M_v(T_p'))} \right), \quad \quad
        \Deg_{\et, p} := \# (T_p/T_p').
    \]
    Then the height change from \Cref{prop:isogenyformulaslope} can be rewritten as
    \[
        h({ M}') - h({{M}}) = \frac{w}{2} \log \Deg_{\et,p} - \frac{1}{[F:\Q]} \sum_{v \mid p} \log \Deg_{\dR, v}.
    \]
This is the way the formula is written in \cite[Proposition 8.1]{koshikawa2015heights}. 

\subsection{Writing isogeny formula as integral over a Galois group}
To make the isogeny formula more amenable to taking limits over an infinite set of isogenies in $\mscr I({ M})_{\ol{F}}$, we want to rewrite the isogeny formula in such a way that it can accommodate situations when $T_p'$ is not a $\Gal_F$-stable submodule of $T_p$. 
\begin{lemma}\label{lemma:nuintegral}
    Suppose that $T_p$ is $\Gal_F$-stable, but $T_p'\subset T_p$ is only $\Gal_{F'}$-stable for some finite extension $F'/F$. Without loss of generality, we may assume $F'/F$ is Galois. 
    Then
    we can write $\nu(T_p,T_p')$ as 
    \begin{align}
    	\nu(T_p,T_p')
        =  \sum_{v\mid p} \frac{[F_v:\Qp]}{[F:\Q]} \int_{\Gal_F} \mu \left(\frac{\mf M_{v}(T_p)}{\mf M_{v}(\sigma T_p')}\right) d\sigma\label{eqn:nuGaloisorbits}
    \end{align}
    We use $d\sigma$ to denote the Haar measure of volume 1 on $\Gal(\ol{F}/F)$. For each $v \mid p$, we are implicitly choosing an extension of $v$ to a place of $\ol{F}$. 
\end{lemma}

\begin{proof}
    We start with 
    \begin{align*}
        \nu(T_p,T_p')&:= \frac{1}{[F':\Q]} \sum_{v' \mid p} [F'_{v'}:\Qp]\:\mu\left(\frac{\mf M_{v'}(T_p)}{\mf M_{v'}(T_p')} \right)
    \end{align*}
    where the summation is over all places $v' \mid p$ of $F'$.
    Fix a place $v \mid p$ of $F$, and make a choice of embedding $\ol{v}: \ol{F} \into \ol{\Qp}$ extending $v: F \into \ol{\Qp}$. 
    Recall that we assumed $F'/F$ is Galois, so that $\ol{v}$ restricts to a given place $v'$ of $F'$ above $v$, and the group $\Gal(F'/F)$ acts transitively on the set of places of $F'$ above $v$. Choosing the totally ramified $\Zp$-extension $F_{v,\infty}$ of $F_v$ also gives a totally ramified $\Zp$-extension $F'_{v',\infty}:= F_{v,\infty}F'_{v'}$. 
        \td{See Brinon-Conrad, you might need to pass to a finite extension of $F'$ to get a $\Zp$-extension of $F'$. Actually, I'm not sure why we do this.}

    We can rewrite this as 
    \begin{align*}
        \nu(T_p,T_p')&:= \frac{1}{[F':\Q]} \sum_{v' \mid p} [F'_{v'}:\Qp]\mu\left(\frac{\mf M_{v'}(T_p)}{\mf M_{v'}(T_p')} \right) \\
        &= \frac{1}{[F':\Q]} \sum_{v \mid p} [F_v:\Qp] \sum_{\sigma \in \Gal(F'/F)}\mu\left(\frac{\mf M_{\sigma v'}(T_p)}{\mf M_{\sigma v'}(T_p')} \right).
    \end{align*}
    The redundancy of the action of $\Gal(F'/F)$ on the places $v' \mid v$ is accounted for by dividing by $[F'_{v'}:F_v]$. 

    Now, instead changing the choice of subgroup $ \Gal_{F_{v'}} \subset \Gal_{F'}$ acting on the fixed sublattice $T_p'$, we can still get an isomorphic torsion BK module by changing the sublattice $T_p' \subset T_p$ itself.


For each $\sigma \in \Gal(F'/F)$, choose a lift $\tilde{\sigma}$ to $\Gal(\ol{F}/F)$. 
The map $\ol{v} \circ \tilde{\sigma}^{-1}: \ol{F} \to \ol{\Qp}$, when restricted to the subfield $F'$, corresponds to the place $\sigma v'$ of $F'$. Starting with the global representation 
\[
    \rho: \Gal_{F} \to \GL(T),
\]
such that the image of $\rho|_{\Gal_{F'}}$ preserves the sublattice $T_n'$. The following commutative diagram shows that the action of the subgroup $\Gal_{F'_{\sigma v'}}$ on $T_n'$ induced by the restriction of $\rho$ is isomorphic to the action of $\Gal_{F'_{\sigma v'}}$ on a distinct sublattice $\rho(\tilde{\sigma})^{-1}T_n'$ induced by the restriction of $\rho$. 
\[
    \begin{tikzcd}
        \Gal_{F'_{\sigma v'}} \arrow[r, "\rho|_{\Gal_{F'_{\sigma v'}}}"] \arrow[d,"\tau \mapsto \tilde{\sigma}^{-1}\tau\tilde{\sigma}"']& \GL(T_n') \arrow[d,"f \mapsto \rho(\tilde{\sigma})^{-1} f\rho(\tilde{\sigma}) "] \\
        \Gal_{F'_{v'}} \arrow[r, "  \rho|_{\Gal_{F'_{v'}}}"] & \GL(\rho(\tilde{\sigma})^{-1} T_n')
    \end{tikzcd}
\]
The purpose of the summation over places $v'\mid p$ of $F'$ in the change-of-height formula is to consider the Galois representations $T_p$ and $T_p'$ as $\Gal_{F_{v'}}$-representations for the different $v' \mid p$. Instead of considering one lattice $T_n'$ acted on by distinct conjugate subgroups of $\Gal_{F'}$, we can equivalently fix one subgroup $\Gal_{F'_{v'}}$ and consider its action on conjugate sublattices $\rho(\tilde{\sigma})^{-1}T_n'$ to get an isomorphic representation.
We write $\sigma T_n'$ for $\sigma \in \Gal(F'/F)$ in (\ref{eqn:nuGaloisorbits}) to denote the lattice $\rho(\tilde{\sigma})T_n'$ generated by any lift $\tilde{\sigma}$ of $\sigma$ to $\Gal_F$. 
\end{proof}

%% file: mtconj.tex
\section{Adelic Mumford--Tate conjecture and other assumptions}
\label{section:MTconj}

\subsection{Mumford--Tate conjecture}
The following setup and statement of the Mumford--Tate conjecture is adapted from the versions given in \cite{cadoret2020integral} and in 

Fix a subfield $k$ of $\C$. Let $X$ be a smooth projective variety over $k$. Fix an integer $w$. Let $ M$ be the Koshikawa motive of weight $w$ corresponding to the $w$th cohomology of $X$. Using the inclusion $k \subset \C$, we can take the Betti realization $M_{B} := H^w_{\Betti}(X(\C),\Z)/\text{torsion}$, which is a polarized $\Z$-Hodge structure. Let $G_{MT} \subset \GL(M_B)$ denote the associated Mumford--Tate group over $\Z$. 
This is the Zariski closure of the usual definition of the Mumford--Tate group $G_{MT} \times_{\Z} \Q$ as an algebraic group over $\Q$.

For every prime $\ell$, let $M_{\ell}:= H^w_{\et}(X_{\ol{k}}, \Z_\ell)/\text{torsion}$ be the $\ell$-adic \'etale realization. There is a continuous Galois representation
\[
    \rho_{{ M}, \ell}: \Gal(\ol{k}/k) \to \GL(M_{\ell}).
\]

\begin{defn}
    Fix $ M$ as above. The \emph{$\ell$-adic algebraic monodromy group} $G_\ell$ for $ M$ is the algebraic group over $\Z_\ell$ which is the Zariski closure of the image $\rho_{{ M}, \ell}(\Gal(\ol{k}/k)) $ in $\GL(M_{\ell})$.
\end{defn}

\nts{If you want to, add an explanation of how to define this even if the variety is defined over $\C$. But we don't really need to worry since we always have something defined over a number field. }

\begin{remark}
    Usually, the $\ell$-adic algebraic monodromy group is an algebraic group over $\Q_\ell$, but we want an integral version so that we can state an adelic version of the Mumford--Tate conjecture. 
\end{remark}

\begin{remark}\label{rmk:MTandMongroupscanvary}
    The groups $G_{MT}$ and $G_\ell$ depend on the choice of complex embedding of $k$ and on $\ell$, respectively.
\end{remark}

Despite not knowing whether the groups $G_\ell$ are independent of $\ell$, one can replace $k$ with a finite extension so that the $\ell$-adic algebraic monodromy groups $G_\ell$ are connected for all $\ell$. The following is originally due to Serre.
\begin{prop}[6.14, \cite{larsen1992independence}]
    The finite index open subgroup $\rho_{{ M},\ell }^{-1}(G_\ell^\circ(\Q_\ell)) \subset \Gal(\ol{k}/k)$ is independent of $\ell$, and therefore the groups $$G_\ell(\Q_\ell)/G^\circ_\ell(\Q_\ell) \cong \Gal(\ol{k}/k) / \rho_{{ M},\ell }^{-1}(G_\ell^\circ(\Q_\ell))$$ are independent of $\ell$. 
\end{prop}
Without loss of generality, we will assume from now on that the $G_\ell$ are connected. 
By \Cref{rmk:geomrepsdeRham}, if the field of definition $k$ is a number field $F$, then for every $v \mid p$ of $F$, the $p$-adic Galois representation $\Gal(\ol{F_v}/F_v)$ acting on $M_p$ is de Rham, hence Hodge--Tate. The following is due to Bogomolov.

\begin{theorem}[Th\'eor\`eme 1, \cite{bogomolov1980algebricite}] \label{thm:bogomolov}
    Let $F$ be a number field, and let $\rho: \Gal_F \to \GL_n(\Q_p)$ be a continuous $p$-adic representation of $\Gal_F$. Let $G$ be the algebraic subgroup over $\Qp$ which is the Zariski closure of $\rho(\Gal_F)$ in $\GL_n$. For each place $\ol{v}$ of $\ol{F}$ extending $v \mid p$ of $F$, suppose that the restriction of $\rho$ to $\Gal(\ol{F}_{\ol{v}}/F_v)$ is a Hodge--Tate representation. Then the image $\rho(\Gal_F)$ is an open subgroup of $G(\Qp)$ in the $p$-adic topology.
\end{theorem}

Given an embedding $\tilde{\sigma}: \ol{k} \into \C$, there is a canonical \'etale-Betti comparison isomorphism $M_B \otimes_{\Z} \A_f \to \prod_{\ell} M_\ell \otimes \Q$ of $\A_f$-modules \cite[Expos\'e XI]{SGA4Tome3}.
Therefore we may view both $G_{MT} \times_{\Z}\Z_\ell$ and $G_\ell$ as subgroup schemes of $\GL(M_\ell)$.

\nts{The Grothendieck motive has an adelic realization, but implicitly there is a choice of $\Z_p$-module structure at every $p$ right? Actually, I don't think there is. So you only know the $\Qp$ structure. You need the choice of $\hat{T}$ to get Koshikawa's $\Z$-motive to get the $\Z_p$-structure. Does the Grothendieck $\Q$-motive have a $\Z$-structure on the Betti realization? NO. It seems like you can get the $\Z$-structure from the Betti-etale comparison using the integral structures on etale cohom.}

\begin{conj}
    [Mumford--Tate conjecture]\label{conj:MumfordTate}
    Using the comparison isomorphism as above for some $\tilde{\sigma}$ and $\ell$, $G_{MT} \times_{\Z}\Z_\ell = G_{\ell}$ as connected subgroup schemes of $\GL(M_\ell)$. 
\end{conj}

\td{
The known results of the Mumford--Tate conjecture are:

- Pink (abelian varieties $X$ with $\End(X) = \Z$ and some numerical conditions on $\dim X$). Other types of endomorphism algebras have been done by Banaszak, Gajda, Krason I, II, III 2006 and 2010.

- Moonen 2017 On the Tate and Mumford--Tate conjectures (Section 2. Motives of K3 type.)

- MTC is true for CM type abelian varieties, Pohlmann 

- simple AV of dimension 1 or prime dimension. Then MTC is true for all powers of X (Serre in dim 1, Tankeev 1982, 1988 for prime dimension)

- Vasiu 2008 proves some cases of MTC under assumption about Lie types of Shimura data. 

- In general, it's not even known that $G_\ell$ is a subgroup of $G_{MT}$.
}

As mentioned in \Cref{rmk:MTandMongroupscanvary}, the Mumford--Tate conjecture a priori depends on the choice of $F \into \C$ and choice of $\ell$.

\td{
Known results for independence of $\tilde{\sigma}, \ell$:

- Larsen and Pink: MTC for AVs is independent of $\ell$.

- Commelin has shown that for abelian motives, the MTC is independent of $\ell$. 

- In general, independence of $\ell$ is not known. Partial results in Larsen and Pink.
}

\subsection{The adelic Mumford--Tate conjecture}
The conjecture may be stated when working over subfield of $\C$ of finite type, but let us work with a motive ${ M} = {\bf H}^w(X)$ for smooth proper $X$ over a number field $F$, and fix a complex embedding $\sigma: F \into \C$. 

Suppose that the Mumford--Tate conjecture (\Cref{conj:MumfordTate}) is true for some choice of $\ell$. By \Cref{thm:bogomolov}, the map $\rho_\ell: \Gal_F \to G_{MT}(\Z_\ell)$ has open image in the $\ell$-adic topology. 

If we consider all $\ell$-adic representations together, and we assume that the Mumford--Tate conjecture is true for $ M$ for every $\ell$, we get a map 
\[
    \hat{\rho}: \Gal_F \to G_{MT}(\hat{\Z}).
\]
Serre suggested the following extension of the Mumford--Tate Conjecture.

\begin{conj}[Adelic Mumford--Tate conjecture, 11.4, \cite{serre1994proprietesconjecturales}] \label{conj:adelicMT}
    Suppose that the Betti realization $M_B$ is Hodge-maximal. Then the image of $\hat{\rho}$ is an open subgroup of $G_{MT}(\hat{\Z})$ in the adelic topology. In particular, for large enough $\ell$, $\rho_\ell: \Gal_{F} \to G_{MT}(\Z_\ell)$ is surjective. 
\end{conj}

\begin{defn}
    Suppose that $M_B$ is a polarized $\Z$-Hodge structure, with the Hodge structure given by the map $h: \mb S \to (G_{MT})_{\R}$. We say that $M_B$ is \emph{Hodge-maximal} if there does not exist a nontrivial isogeny of connected $\Q$-groups $G' \to (G_{MT})_{\Q}$ such that the Hodge structure lifts along the isogeny to a map $h': \mb S \to G'_{\R}$. 
\end{defn}

\begin{remark}
    By \cite[Remark 2.6]{cadoret2020integral}, it is a necessary condition that $M_B$ is Hodge-maximal in order for \Cref{conj:adelicMT} to hold. Otherwise, the fact that $G'(\A_f) \to G_{MT}(\A_f)$ does not have open image is an obstruction to $\hat{\rho}(\Gal_F)$ having open image in $G_{MT}(\hat{\Z})\subset G_{MT}(\A_f)$. 
\end{remark}

The following remark closely follows \cite[\S 2.2]{cadoret2020integral}, which in turn closely follows \cite[\S 0.2]{wintenberger2002extension}. 
\begin{remark} \label{rmk:pi1}
    Once again, let the Hodge structure on $M_B$ be given by the map 
    $$h: \mb S \to (G_{MT})_{\R} \subset \GL(M_B)_{\R}.$$
    Let $h_{\C}: \mb S_{\C} \to (G_{MT})_{\C}$ denote the base change to $\C$, and let $\bm \mu': (\Gm)_{\C} \to (G_{MT})_{\C}$ denote the Hodge cocharacter, given by $\bm \mu'(z):= h_{\C}(z,1)$. 
    By definition of the Mumford--Tate group as the smallest group defined over $\Q$ containing the image of $h$,
    the images of $\sigma^*h_{\C}$ for $\sigma \in \Aut_{\Q}(\C)$ must generate $G_{MT}$ as an algebraic group. One can check that the images of $\sigma^*\bm{\mu'}_{\C}$ also generate $G_{MT}$ as an algebraic group, which is defined in \cite[\S 1]{borovoi1998abelian} as $\pi_1(G_{MT}) = X_*/\angles{R^\vee}$, the quotient of the cocharacter lattice $X_*$ by the coroot lattice $\angles{R^\vee}$. Each cocharacter $\sigma^* \bm \mu ' $ defines a $G_{MT}$-conjugacy class of cocharacters and equivalently a Weyl group orbit of cocharacters landing in a specific maximal torus. The Weyl group acts trivially on the quotient $\pi_1(G_{MT})$, so the $\sigma^* \bm \mu ' $ define elements of $\pi_1(G_{MT})$. The fact that the images of these cocharacters generate $G_{MT}$ as a group means that they generate a full-rank sublattice of $X_*$, and therefore the images $(\sigma^* \bm \mu ')_* (\pi_1(\Gm))$ generate a finite-index subgroup of $\pi_1(G_{MT})$. 
    
    One can check that the condition that $h$ is Hodge-maximal is equivalent to the condition that the subgroups $\sigma^*\bm \mu ' (\pi_1(\Gm))$ for all $\sigma \in \Aut_{\Q}(\C)$ generate the fundamental group $\pi_1(G_{MT})$.
\end{remark}

\subsubsection{} \label{subsubsec:centralisog}
We do not necessarily know that our motive $ M$ is Hodge-maximal. There exists a $G'$ over $\Q$ and a central isogeny $G' \to G_{MT}$ such that the lift $h': \mb S \to G'_{\R}$ is itself Hodge-maximal. Indeed, the group $G'$ corresponds to the finite cover of $G_{MT}$ defined by the finite-index subgroup of $\pi_1(G_{MT})$ generated by the $\sigma^*\bm \mu ' (\pi_1(\Gm))$ for all $\sigma \in \Aut_{\Q}(\C)$ (see \Cref{rmk:pi1}). Moreover, the cocharacters $\sigma^* \bm \mu'$ (and hence $\sigma^* h$) all lift to $G'$, giving us the Hodge structure for $G'$.

\subsubsection{}\label{subsubsec:lifting}
One may lift the Galois representations to $G'$ as well. Indeed, once the Hodge cocharacter $\bm \mu'$ lifts to $G'$, the following result of Wintenberger shows that the $\ell$-adic Galois representations for all $\ell$ simultaneously lift to be valued in $G'$, up to restricting to the subgroup $\Gal_{E} \subset \Gal_F$ for a finite extension $E/F$. Note that applying this result requires the assumption that the Hodge classes defining the Mumford--Tate group are absolute Hodge, and that these are compatible with the $p$-adic \'etale-de Rham comparison isomorphism. 

\begin{theorem}
    [\cite{wintenberger1995relevement}, Th\'eor\`eme 2.1.7]
    Let $F$ be a number field, let $G/\Q$ be a linear algebraic group, and let $(\rho_\ell)_\ell$ be a system of $\ell$-adic Galois representations $\rho_\ell: \Gal_F \to G(\Q_\ell)$ for all primes $\ell$. 
    Suppose that $(\rho_\ell)_\ell$ comes from the cohomology of a smooth proper variety defined over $F$. \td{What is the actual condition?}
    Suppose further that the collection of Hodge cycles defining the Mumford--Tate group are absolute Hodge, defined over $F$, and that they are compatible with the $p$-adic \'etale-de Rham comparison isomorphism theorem.
    Let $\pi: G' \to G$ be a central isogeny of linear algebraic groups defined over $\Q$. 
    Suppose that the Hodge cocharacter $\bm \mu': \Gm \to G_{\R}$ lifts to $G'_{\R}$. 
    Then there exists a finite extension $E/F$, a finite set of places $S'$ of $E$, and a system of $\ell$-adic Galois representations $(\rho'_\ell)$ with $\rho'_{\ell}: \Gal_E \to G'(\Q_\ell)$ lifting $(\rho_\ell)$. 
    These have good reduction outside of $S'$ in the sense of \Cref{defn:goodsemistablered}.  
\end{theorem}

\nts{Wintenberger is working with $\Q_\ell$-points of these groups, not $\Z_\ell$. Is the lift unique in some sense? Yes}
\td{Need to justify landing in the integral points of $G'$ (see discussion with Mark). Also, it's not known that this is a compatible system in the sense of someone (?), but some partial results are known. For example the lifts can be chosen to be crystalline? Not sure if it's better to put it here, but also need to explain the mod $p$ version, and why you should still have a complex conjugation relation.}

Serre conjectures that there exists a motive $ M'$ which gives rise to the Hodge structure and Galois representations landing in $G'$ \cite[11.8]{serre1994proprietesconjecturales}. 

Granting the adelic MT conjecture for this new motive $ M'$, one expects that the image of $\Gal_F$ in $G'(\A_f)$ is open. Therefore, the best that one can hope for $ M$ is that the image $\hat{\rho}_{ M}(\Gal_F)$ is open in the image of $G'(\hat{\Z})\to G_{MT}(\hat{\Z})$. We take this to be our adelic Mumford--Tate conjecture.

\begin{conj}[non-Hodge-maximal adelic Mumford--Tate conjecture]\label{conj:nonHodgemaxadelicMT}
    Suppose that ${ M} = { H}^w(X)$ is a motive coming from a smooth projective variety $X$ over a number field $F$. Fix a complex embedding $\sigma: F \into \C$. Let $G_{MT}$ be the Mumford--Tate group of the Betti realization $M_B$ with respect to $\sigma$. 
    Let $\pi: G' \to G_{MT}$ be the central isogeny constructed as in \ref{subsubsec:centralisog}. 
    Assume that the Mumford--Tate conjecture (\Cref{conj:MumfordTate}) holds for ${ M}$ with respect to $\sigma$ and for all $\ell$.
    Then, for a finite extension $E/F$, the image of $\hat{\rho}(\Gal_{E})$ is open in the image of $\pi: G'(\hat{\Z}) \to G_{MT}(\hat{\Z})$. 
\end{conj}




\subsection{Hodge and Hodge--Tate cocharacters}
Let $X$ be a smooth projective variety over a number field $F$. The following subsection repeats notation used in \cite{kisinmocznorthcott}.

Let $C$ denote the completion of an algebraic closure of $\Qp$. We fix an isomorphism $C \cong \C$. We also fix a finite place $v \mid p$ of $F$ and extend it to an embedding $v: \ol{F} \to C$. Using $C \cong \C$, this gives an embedding $\tilde{v}: \ol{F} \to \C$.

\subsubsection{Hodge cocharacter}
As in \Cref{rmk:pi1}, let the Hodge structure on $M_{B, \tilde{v}} = H^w_B(X(\C),\Z)/\text{torsion}$ be given by the map $h: \mb S \to (G_{MT})_{\R} \subset \GL(M_{B,\tilde v})_{\R}$. 
Let $h_{\C}: \mb S_{\C} \to (G_{MT})_{\C}$ denote the base change to $\C$, and let $\bm \mu'_v: (\Gm)_{\C} \to (G_{MT})_{\C}$ denote the Hodge cocharacter. In particular, for an element $x \in (M_{B,\tilde v} \otimes \C)^{p,q}$ and $z \in \C^\times$, $\bm \mu'(z)(x) = z^{-p}x$ \cite[\S 2, p.~26]{milne2005introduction}.
This gives the Hodge decomposition
\begin{align}\label{eqn:hodgedecomp}
    H^w_B(X(\C), \Q)\otimes_{\Q} \C \cong H^w_{HT}(X/F) \otimes_F \C. 
\end{align}
Here, we define the Hodge cohomology
\[
    H^w_{HT}(X/F):= \gr^* H^w_{\dR}(X/F)
\]
as the graded $F$-vector space associated to the filtered de Rham cohomology.

Because we have fixed the embedding $\tilde v: \ol{F} \to \C$, we get a canonical isomorphism \cite[Expos\'e XI]{SGA4Tome3}
\[
    H^w_{\et}(X_{\ol{F}}, \Qp) \otimes_{\Qp} \C \cong H^w_B(X(\C), \Q)\otimes_{\Q} \C.
\]
So we may view $\bm \mu'_v : (\Gm)_{\C} \to \GL( H^w(X_{\ol{F}}, \Qp) \otimes \C )$.

Let $c \in \Aut(\C) \cong \Aut(C)$ denote complex conjugation. 
The $\C$-linear map
\begin{align*}
    (H^w_{B}(X(\C),\Q)\otimes\C)\otimes_{\C, c}\C 
    & \xto{c}
    H^w_{B}(X(\C),\Q)\otimes\C\\
    (x \otimes z) \otimes z'& \mapsto x \otimes \ol{z} z. 
\end{align*}
puts a new grading on the left-hand side using the Hodge decomposition on the right-hand side. 

We can identify the left-hand side $(H^w_{B}(X(\C),\Q)\otimes\C)\otimes_{\C, c}\C$ with $H^w_{B}(X(\C),\Q)\otimes\C$ as $\Q$-vector spaces in a set-theoretic way, where $(x \otimes z)\otimes z'$ corresponds to $x \otimes z\ol{z}'$. This new grading on the left-hand side, now viewed as $H^w_{B}(X(\C),\Q)\otimes\C$, defines a new cocharacter which we denote 
\begin{align}
    c^* \bm \mu'_v: (\Gm)_{\C}\to \GL(H^w_{B}(X(\C),\Q)\otimes\C).
\end{align}
Hodge symmetry dictates that the $p$th graded piece on $H^w_{B}(X(\C),\Q)\otimes\C$ with respect to the grading from $ \bm \mu'_v$ corresponds to the $(w-p)$th graded piece on $H^w_{B}(X(\C),\Q)\otimes\C$ with respect to the grading from $ c^*\bm \mu'_v$. In other words, we have the relation
\begin{align}\label{eqn:hodgecocharsymmetry}
    \bm \mu'_v \cdot c^*\bm \mu'_v = \bm w^{-1},
\end{align}
where $\bm w^{-1}: (\Gm)_{\C}\to \GL(H^w_B(X(\C),\Q)\otimes \C)$ is the inverse of the weight cocharacter and acts on $H^w_{B}(X(\C),\Q)\otimes\C$ via $\bm w^{-1}(z): v \mapsto z^{-w}v$.

\subsubsection{Hodge--Tate cocharacter}\label{subsubsec:HTcochar}
By \cite[Theorem 8.1]{faltings1989crystalline} (and ignoring the Galois action), there is a canonical isomorphism of graded $C$-vector spaces
\begin{align}
    H^w_{\et}(X_{\ol{F}}, \Qp)\otimes_{\Qp} C \cong H^w_{\et}(X_{\ol{F_v}}, \Qp)\otimes_{\Qp} C \cong H^w_{HT}(X/F) \otimes_{F, v}C,
\end{align}
which gives rise to what is known as the Hodge--Tate decomposition on the left-hand side.
The first isomorphism follows from smooth base change. \td{check this!!}
Let
\begin{align}
    \bm \mu_v: (\Gm)_C \to \GL(H^w_{\et}(X_{\ol{F}}, \Qp)\otimes_{\Qp} C)
\end{align}
denote the cocharacter corresponding to the Hodge--Tate decomposition.

\subsection{Absolute Hodge conjecture and compatibility with $D_{\dR}$}

Another assumption we need to make is to ensure that we can compare the Hodge cocharacter with respect to $\tilde v$ and Hodge--Tate cocharacter with respect to $v$. In particular, we want these two characters to be conjugate by an element in $G_{MT}(C)$. Moreover, we want to translate the relation of Hodge cocharacters from (\ref{eqn:hodgecocharsymmetry}) to a relation on Hodge--Tate cocharacters. This relation is what allows us to lower-bound the change in height in both the fixed prime and large prime cases of the proof of the main theorem. 

\subsubsection{}
Fix a choice $v: \ol{F} \into C$ and an isomorphism $C \cong \C$, which yields $\tilde v: \ol{F} \into \C$.
The choice of $\tilde v |_{F}: F \into \C$ defines a comparison isomorphism
\begin{align*}
    H^w_B(X(\C), \Q) \otimes_{\Q} \C \cong H^w_{\dR}(X/F) \otimes_F \C.
\end{align*}
The choice of $\tilde v: \ol{F}\to \C$ defines a comparison isomorphism
\begin{align*}
    H^w_B(X(\C), \Q) \otimes_{\Q} \A_f \cong \left(\prod_p H^w_{\et}(X_{\ol{F}}, \Z_p)\right) \otimes_{\Z} \Q.
\end{align*}
For convenience, we will write 
\[  
    H^w_{\et}(X_{\ol{F}},\A_f) := \left(\prod_p H^w_{\et}(X_{\ol{F}}, \Z_p)\right) \otimes_{\Z} \Q.
\]
These together define the isomorphism
\begin{align}
    \label{eqn:bettiderhametale}
     H^w_B(X(\C), \Q) \otimes_{\Q} (\A_f \times \C) \cong   H^w_{\et}(X_{\ol{F}},\A_f) \times \left( H^w_{\dR}(X/F) \otimes_F \C \right). 
\end{align}
\nts{Might consider including the twists to make everything weight 0.}

\nts{Define Hodge classes.}

\begin{conj}
    [Hodge classes are absolute Hodge, cf.~Theorem 2.11, \cite{deligne1982hodge}]\label{conj:absolutehodgegeneral}
    \label{conj:absolutehodge}
    Let $X$ be a smooth projective variety over a algebraically closed field $k = \ol{k}$ of finite transcendence degree over $\Q$. Suppose that a pair $(s_{\A_f},s_{\dR}) \in  (H^w_{\et}(X,\A_f) \times  H^w_{\dR}(X/k))^{\otimes}$ comes from a Hodge class in $(H^w_B(X(\C),\Q))^\otimes$ for a given embedding $\sigma: k \into \C$. Then for every embedding $k \into \C$, $(s_{\A_f},s_{\dR})$ comes from a Hodge class under the comparison isomorphism with Betti cohomology. 
\end{conj}
\begin{rmk}
    We say that such pair $(s_{\A_f},s_{\dR})$ satisfying the statement of the conjecture is an \emph{absolute Hodge class}. The conjecture is saying that all Hodge classes for any embedding $k \into \C$ are absolute Hodge classes. 
\end{rmk}

The absolute Hodge conjecture was proved by Deligne for abelian varieties \cite{deligne1982hodge}.
\td{Add other implications/known results about AH conjecture, it would make Tate conjecture imply Hodge conjecture?}

\subsubsection{}\label{subsubsec:dRalgebraic}
Let's return to our setting of $X$ being defined over a number field $F$. For our fixed $\tilde v: \ol{F} \to \C$, another definition of the Mumford--Tate group $(G_{MT,\tilde v})_{\Q}$ of $H^w_B(X(\C), \Q)$ is the group stabilized by a finite set of Hodge cycles $\{s_\alpha\}$ in $(H^w_B(X(\C),\Q))^{\otimes}$, the direct sum of all tensor powers of $H^w_B(X(\C),\Q)$ and its dual. 
Under the isomorphism (\ref{eqn:bettiderhametale}), these cycles in Betti cohomology give rise to pairs $\{(s_{\alpha,\A_f}, s_{\alpha,\dR})\}$.
If the absolute Hodge conjecture is true for $X/F$, then it follows that there is a finite extension $F'/F$ such that the corresponding $s_{\alpha,\dR}$ all live in $H^w_{\dR}(X_{F'}/F') \subset H^w_{\dR}(X_{\ol F}/\ol F)$ \cite[Proposition 2.9(b)]{deligne1982hodge}, or see \cite[2.4.1]{kisinmocznorthcott} for additional explanation.

In addition, $p$-adic Hodge theory provides an \'etale--de Rham comparison isomorphism. As cited in \ref{subsubsec:HTcochar}, Faltings proves the following. 
\begin{theorem}[Theorem 8.1, \cite{faltings1989crystalline}]\label{thm:faltingsetaledR}
    Let $X$ be smooth and separated over a finite extension $K/\Qp$. Then there is an isomorphism
    \[
        H^*_{\et}(X_{\ol{K}},\Qp)\otimes_{\Qp} \BdR \cong H^*_{\dR}(X/K) \otimes_K \BdR
    \]
    respecting the Galois action and filtration on both sides. 
\end{theorem}
In the case of abelian varieties over $\ol{\Q}$, Blasius proves that absolute Hodge classes $s \mapsto (s_{\A_f}, s_{\dR})$ are compatible under Faltings' isomorphism as well. 
\begin{theorem}[\cite{blasius1994padic}]
    Let $X$ be an abelian variety defined over $\ol{\Q}$. Fix a prime $p$ and an embedding $v: \ol{\Q} \to C$, and isomorphism $C \cong \C$, and $\tilde v: \ol{\Q}\to \C$.
    Let $s \in (H^1(X(\C), \Q))^\otimes$ be a Hodge class with respect to $\tilde{v}$ (by Deligne, also an absolute Hodge class). Under the Betti-\'etale and Betti-de Rham comparison isomorphisms, suppose that $s$ corresponds to the cocycles $s_{p} \in (H^1_{\et}(X,\Qp))^\otimes$ and $s_{\dR} \in (H^1_{\dR}(X/F'))^\otimes$, where $F'$ is some number field containing the field of definition for $X$ (cf.~\ref{subsubsec:dRalgebraic}). Then under the comparison isomorphism of Faltings in \Cref{thm:faltingsetaledR} applied to $X_{F'_v}$, the cocycle $s_{p}$ is sent to $s_{\dR}$. 
\end{theorem}

We formulate the analogous conjecture for arbitrary smooth projective varieties $X$ defined over a number field.

\begin{conj}[Hodge classes compatible with $D_{\dR}$]\label{conj:blasius}
    Let $X$ be a smooth projective variety over a number field $F$.
    Let $(s_{\A_f}, s_{\dR}) \in (H^w_{\et}(X_{\ol{F}},\A_f) \times H^w_{\dR}(X/F'))^\otimes$ be an absolute Hodge class. 
    Then the corresponding $p$-adic \'etale cohomology class $s_{p}\in H^w_{\et}(X_{\ol{F}},\Qp)$ maps to $s_{\dR}$ under Faltings' isomorphism in \Cref{thm:faltingsetaledR}. 
\end{conj}

Various important properties of the Hodge--Tate cocharacter follow from \Cref{conj:blasius}, if we know that the $\{s_\alpha\}$ defined in \ref{subsubsec:dRalgebraic} are absolute Hodge classes.

Once we have assumed Conjectures \ref{conj:blasius} and \ref{conj:absolutehodge}, the proofs of \cite[Proposition 2.4.4]{kisinmocznorthcott} also apply in our more general setting. We state the analogous properties here.

\begin{prop}[Properties following from the conjectures, \cite{kisinmocznorthcott}]
    Suppose we have $X$ and $\{s_\alpha\}$ as above. For two $G_{MT}$-valued cocharacters $\mu_1$ and $\mu_2$, we write $\mu_1 \sim \mu_2$ if the two cocharacters are conjugate over some field containing the fields of definition of $\mu_1$ and $\mu_2$. 
    \begin{enumerate}[(1)]
    \item \Cref{conj:absolutehodge} implies that, up to replacing $F$ with a finite extension independent of the choice of prime $p$, the Galois representation $\rho: \Gal_F \to \GL(H^w_{\et}(X_{\ol{F}},\Qp))$ is contained in $G_{MT}(\Qp)$, where we view $(G_{MT})_{\Qp} \subset \GL(H^w_{\et}(X_{\ol{F}},\Qp))$ as the stabilizer of the $\{s_{\alpha,p}\}$. 
    \item \Cref{conj:absolutehodge} and \Cref{conj:blasius} imply the following.
    \begin{enumerate}
        \item The Hodge--Tate cocharacter $\bm \mu_v$ is $G_{MT}$-valued.
        \item For $\sigma \in \Aut(C)$, $\bm \mu_{\sigma \circ v} \sim \sigma^*(\bm \mu_v)$.
        \item Let $c \in \Aut(\C) \cong \Aut(C)$ denote complex conjugation. Then $\bm \mu_{c \circ v} \cdot \bm \mu_v \sim \bm w^{-1}$, where $\bm w^{-1}$ is defined in (\ref{eqn:hodgecocharsymmetry}). 
        \item The images of the $G_{MT}(C)$-conjugates of the cocharacters $\bm \mu_v$, for all embeddings $v: \ol{F} \to C$, generate the group $G_{MT}$. \td{check this. Also don't we need adelic MT assumption? To say its dense }
    \end{enumerate}
\end{enumerate}
\end{prop}

\td{!!! Also say the crucial statement, which is about the Galois conjugates of the HT cocharacter generating the MT group if you assume the conjectures. }

\nts{Be explicit about what follows from \Cref{conj:absolutehodge} and \Cref{conj:blasius}. 

- $\sigma^* \bm \mu_v \sim \bm \mu_{\sigma \circ v}$

- the cocharacter $\bm \mu_v$ is $G_{MT}$-valued because the Blasius result allows you to say that the $s_{\alpha,p}$ live in the 0th graded piece always, since that's where the $s_{\alpha,\dR}$ live. 

- Then, since the HT cochar is $G$-valued and so is the Hodge cochar, they are conjugate by $G$. 

- This allows us to get the relation we want between conjugate embeddings and HT cochars.

Following from the absolute Hodge conjecture:

- the galois action lands in the MT group (see 2.4.1, still don't quite understand)

- we wanted the Hodge classes defining $G_{MT}$ to be absolute Hodge classes, to apply Blasius' result?? To use faltings, probably need a finite extension.
}

\td{Check and add corollary 2.4.7}

\td{CHAR p VERSION OF HT COCHARACTER}

%% file: case1.tex
\section{Main result: fixed prime case}
\label{section:case1}
For the rest of this section, fix a pure Koshikawa motive $M = ({\bf M},\hat{T})$ defined over a number field $F$. Fix a prime $p$. We will write $T = T_p$ for the $\Zp$-module quotient of $\hat{T}$.

The goal of this section is to prove the following result. 

\begin{theorem}\label{thm:fixedprimecase}
    Fix $c > 0$. Suppose that for a fixed prime $p$, there are infinitely many distinct $p$-power isogenies $M' \in \mscr{I}(M)_{\ol{F}}$ such that $h(M')< c$. Then the values $h(M')$ for all such $M'$ form a finite set. 
\end{theorem}

\td{First reduce to the case of $L + p^nT$.}
Fix a saturated $\Zp$-submodule $L \subset T$, that is, a submodule such that $T/L$ is torsion-free. For $n \ge 0$, define
$$T_n':= L + p^nT.$$
Define the $\hat{\Z}$-sublattice $\hat{T}'_n \subset \hat{T}$ to be equal to $\hat{T}$ at all $\ell$-adic factors for $\ell \neq p$, and to equal $T_n'$ at the $p$-adic factor.

\subsection{Construction of $\ML$}\label{ss:constructionML}
\begin{defn}\label{defn:latticepairsfinalobject}
    Let $\mcC$ be the category of pairs $(T,\Xi)$, where $T$ is a finite free $\Zp$-module and $\Xi \subset T \otimes_{\Zp} \BdR $ is a $\BdR^+$-lattice. 
    A morphism $f:(T',\Xi') \to (T,\Xi)$ in $\mcC$ is given by a map $T' \to T$ of $\Zp$-modules which, after base-changing to $\BdR$, restricts to a $\BdR^+$-module homomorphism $\Xi' \to \Xi$. 
    Fix an object $(T,\Xi)$ in $\mcC$, and fix a $\Zp$-submodule $L \subset T$ such that $T/L$ is $p$-torsion free. 
\end{defn}

\subsubsection{}
Define the category $\mcC_{L,T}$ as the category of pairs $((T',\Xi'), f: (T',\Xi') \to (T,\Xi))$ where $(T',\Xi')$ is an object of $\mcC$, and $f$ is a morphism in $\mcC$ such that the image of $T'$ inside $T$ is contained inside $L$. Morphisms in $\mcC_L$ are defined as morphisms $(T',\Xi') \to (T'',\Xi'')$ which commute with the given maps to $(T,\Xi)$. From the definition of $\mcC_{L,T}$, it has a final object $((L,\Xi_L), i:(L,\Xi_L) \into (T,\Xi))$, where
\[
    \Xi_L := \Xi \cap (L \otimes_{\Zp}\BdR),
\]  
and the map $i$ is induced by inclusion $L \into T$. 



\begin{prop}\label{prop:MLisfinitefree}
	Let $(T,\Xi)$ be a choice of a finite free $\Zp$-module $T$ and a full rank $\BdR^+$-lattice $\Xi \subset T \otimes_{\Zp} \BdR$. By Fargues (\Cref{thm:farguesequivalence}), this is equivalent to the data of a finite free BKF module, which we denote $\MT$. Let $L \subset T$ be a $\Zp$-submodule such that $T/L$ is torsion-free. We can define a finite free BKF module $\ML$ such that $\ML$ corresponds to the final object $(L,\Xi_L)$ in $\mcC_{L,T}$ as in \Cref{defn:latticepairsfinalobject}, where $\Xi_L:= \Xi \cap L \otimes_{\Zp} \BdR$.
\end{prop}


\begin{proof}
	WLOG we may assume that $\MT$ is effective, i.e.~the image of $\vphi^* \MT$ is in $\MT$. 
    
    Define
    \begin{align}\label{eqn:defML}
        \ML := \MT \cap (L \otimes_{\Zp}\Ainf[1/\mu]).
    \end{align}
    If we show $\ML$ is a finite free $\Ainf$-module and stable under the Frobenius induced by $\vphi_{\MT}$, then it is indeed a finite free BKF module. Moreover, the relation in \Cref{rmk:XiMrelation} shows that $\ML$ must be the object corresponding to $(L,\Xi_L)$ under Fargues' equivalence. By \Cref{lemma:MLstableunderFrob} and $\MT$ being effective, we know that $\ML$ is $\vphi$-stable. It remains to show that $\ML$ is finite free. 
    

    We claim that it suffices to show that $\mf M(L)$ defines a vector bundle after restricting to $U:= \Spec \Ainf \bs \{\mf m\}$, where $\mf m$ is the unique maximal ideal of $\Ainf$. Let $\iota: U \to \Spec \Ainf$ denote the inclusion. By \cite[Lemma 14.2.3]{scholzeweinstein2020berkeley}, if we can show that $\mf M(L)|_U$ is a vector bundle on $U$, then $\iota_*(\mf M(L)|_U)$ is in fact a finite free $\Ainf$-module. There is a natural map $\mf M(L) \to \iota_*(\mf M(L)|_U)$ which commutes with $\mf M(T) \to \iota_* (\mf M(T)|_U)$. This latter map is the identity map since $\mf M(T)$ is already finite free. Of course, the map $\mf M(L) \to \iota_*(\mf M(L)|_U)$ is an isomorphism over $U$, so the kernel and cokernel are $\Ainf$-modules supported at the closed point $\mf m$. Since $\mf M(L)$ is torsion-free, the map is injective. If the cokernel is nonzero and supported at $\mf m$, then there exists a nonzero element $x \in \iota_*(\mf M(L)|_U)$ such that $\mu x \in \mf M(L) \subset L \otimes \Ainf[1/\mu]$. However, $x$ is also in $\iota_*(\mf M(T)|_U) = \mf M(T)$, so this implies that $x \in \mf M(T) \cap (L \otimes \Ainf[1/\mu])= \mf M(L)$. Therefore $\mf M(L)$ is itself a finite free $\Ainf$-module.


	To show that $\mf M(L)|_U$ is a vector bundle, we use the open cover $\Spec \Ainf[1/\mu] \cup \Spec \ainf[1/p] = \Spec \Ainf \bs \{ \mf m\}$. By \Cref{lemma:BMS4.26}, we will then have that $\ML$ restricted to $\Spec \Ainf[1/\mu]$ is finite free already, since it becomes $L \otimes_{\Zp} \Ainf[1/\mu]$. 

	Now we treat the case of $\ML[1/p]$.
	Let $i \ge 0$. By Lemma \ref{lemma:MLlocalisomorphism}, $\vphi$ on $\MT$ induces an isomorphism on the $\vphi^{-i}(\xi)$-adic completions
	\[
		(\vphi^* \ML[p^{-1}])^{\wedge}_{(\vphi^{-i}(\xi))} \xto{\sim} \ML[p^{-1}]^{\wedge}_{(\vphi^{-i}(\xi))}.
	\]
	In particular, this identifies the completed stalk of $\ML[1/p]$ at $(\vphi^{-i-1}(\xi))$ with the completed stalk of $\ML[1/p]$ at $(\vphi^{-i}(\xi))$. 
    By \Cref{lemma:MLonegoodstalk}, we know that the completion $\ML[1/p]^{\wedge}_{(\xi)}$ is a free $\Ainf[1/p]^{\wedge}_{(\xi)} = \BdR^+$-module of rank equal to $\rk_{\Zp}L$,
    and therefore the same is true for $\ML[1/p]^{\wedge}_{(\vphi^{-i}(\xi))}$ for all $i\ge 1$.

	By \cite[Lemma 3.23]{BMS2018integral}, we know that $V((\mu)) = \bigcup_{i \ge 0} V(\vphi^{-i}(\xi)) \subset \Spec \Ainf[1/p]$, and $(\vphi^{-i}(\xi))$ is a maximal ideal for all $i$. Moreover, we have that for all $i \ge 0$,
    $$
        \Spec \Ainf\left[ \frac{1}{p\vphi^{-i-1}(\mu)} \right] = \Spec \Ainf \left[ \frac{1}{p\vphi^{-i}(\mu)} \right] \cup \{ (\vphi^{-i}(\xi)) \}
    $$
    inside of $\Spec \Ainf[1/p]$. Therefore, by applying Beauville-Laszlo repeatedly, we can combine the fact that $\ML[1/p\mu]$ is free of rank equal to $\rk_{\Zp}(L)$, and the completions $\ML[1/p]^{\wedge}_{(\vphi^{-i}(\xi))}$ are free of the same rank for all $i \ge 0$, to conclude that $\ML[1/p]$ defines a vector bundle on $\Spec \Ainf[1/p]$. 
\end{proof}

\begin{lemma}\label{lemma:MLonegoodstalk}
    Using $(T,\Xi)$, $\MT$, $L$, and $\ML$ as defined in \Cref{prop:MLisfinitefree}, we have that $\ML[1/p]^{\wedge}$ is a finite free $\BdR^+$-module of rank equal to $\rk_{\Zp}L$, where the ${}^\wedge$ denotes $\xi$-adic completion.
\end{lemma}

\begin{proof}
    Consider the following commutative diagram, where the two horizontal maps are the natural ones, and the vertical maps are induced by the inclusion $\ML \into\MT$. 
    \[
    \begin{tikzcd}
        \ML[1/p] \otimes_{\Ainf[1/p]} \BdR^+ \arrow[r] \arrow[d] & \ML[1/p]^{\wedge} \arrow[d]
        \\
        \MT[1/p] \otimes_{\Ainf[1/p]} \BdR^+ \arrow[r,"\sim"] & \MT[1/p]^{\wedge} 
    \end{tikzcd}
    \]
    Note that the bottom horizontal map is an isomorphism because $\MT[1/p]$ is a finite free $\Ainf[1/p]$-module\footnote{This allows us to commute the tensor product with finite intersection.}, and the left map is an injection. By the definition of $\ML$ and the fact that $\Ainf \to \BdR^+$ is a flat map, we can check that $\ML[1/p] \otimes_{\Ainf[1/p]} \BdR^+ = \Xi \cap L \otimes_{\Zp} \BdR^+$, which is a finite free $\BdR^+$-module of the correct rank.  
    
    Therefore, to show that $\ML[1/p]^\wedge$ is finite free of the correct rank, it now suffices to show that the right vertical map is injective. Indeed, the map
    \[
        \frac{\ML[1/p]}{\xi^n \ML [1/p]} \to \frac{\MT[1/p]}{\xi^n \MT [1/p]}
    \]
    is injective for all $n \ge 1$, so it's injective in the limit. 
\end{proof}

\begin{lemma}\label{lemma:MLstableunderFrob}
	If $\vphi(\MT)\subset \MT$, then $\vphi(\ML )\subset \ML$ as well.  
\end{lemma}
\begin{proof}
	\begin{align}
		\vphi(\ML) &\subset \MT \cap (L \otimes \Ainf[1/\vphi^j(\mu)]_{j \ge 0}) \\
		&=\MT \cap (L \otimes \Ainf[1/\vphi^j(\mu)]_{j \ge 0}) \cap (T \otimes \Ainf[1/\mu]) 
		\\
		&= \MT \cap \left( (L \otimes \Ainf[1/\vphi^j(\mu)]_{j \ge 0}) \cap (T \otimes \Ainf[1/\mu]) \right)\\
		&= \MT \cap (L \otimes \Ainf[1/\mu]) =: \ML  \label{eqn:MLnaivedef}
	\end{align}
    The first equality follows from $\MT \otimes \Ainf[1/\mu] \cong T \otimes \Ainf[1/\mu]$. 
\end{proof}

\begin{lemma} \label{lemma:MTMLtorsionfree}
    $\MT/\ML$ is $\vphi(\xi)$-torsion free.
\end{lemma}
\begin{proof}
    Suppose that $x \in \MT$ such that $\vphi(\xi)x \in \ML$. Recall that $1/\vphi(\xi) = \mu/\vphi(\mu)$, so that
    $x \in \MT \cap (L \otimes_{\Zp} \Ainf[1/\vphi^j(\mu)]_{j \ge 0}) = \ML$, where the equality is shown in the proof of \Cref{lemma:MLstableunderFrob}. 
\end{proof}

\begin{lemma}\label{lemma:MLlocalisomorphism}
	For all $i \ge 0$ and $n \ge 1$, the map $\vphi$ on $\ML$ induces an isomorphism 
	\begin{align} \label{eqn:MLlocalisomorphism}
		\frac{\vphi^*\ML[p^{-1}]}{\vphi^{-i}(\xi)^n \vphi^*\ML[p^{-1}]} \xto{\vphi} \frac{\ML[p^{-1}]}{\vphi^{-i}(\xi)^n \ML[p^{-1}]}.
	\end{align}
\end{lemma}

\begin{proof}
    \td{TODO: check this proof}
	By Lemma \ref{lemma:MLstableunderFrob}, we have that the map is well-defined. 

	To show surjectivity, choose some element $x \in \ML[p^{-1}]$. 
	Viewing $x \in \MT[p^{-1}]$, since $\MT$ is a Breuil--Kisin--Fargues module, there exists some $N$ such that $\vphi(\xi) ^N x$ is in the image of $\vphi$, i.e.~there exists some $y \in \MT[p^{-1}]$ such that $\vphi(y)=\vphi(\xi)^N x$. 
    Moreover, since $\vphi|_{\MT}$ is injective, such a $y$ is unique given $N$ and $x$. 
	On the other hand, there also exists a unique $y' \in L \otimes \Ainf[1/p\vphi^{-1}(\mu)]$ such that $\vphi(y')=\vphi(\xi)^Nx$. 
    Here, the $\vphi$-semilinear map $\vphi: L \otimes_{\Zp} \Ainf[1/p\vphi^{-1}(\mu)] \to L \otimes_{\Zp} \Ainf[1/p\mu]$ is induced by the identity on $L$, and hence is immediately bijective.
	In particular, viewing both $ L \otimes \Ainf[1/p\vphi^{-1}(\mu)]$ and $\MT[p^{-1}]$ as inside $\MT[1/p\mu]$, we have $y = y'$. In particular,
	\[
	y' \in L \otimes \Ainf[1/p\vphi^{-1}(\mu)] \cap \MT \subset L \otimes \Ainf[1/p\mu] \cap \MT[1/p] =: \ML[1/p]
	\]
	where the inclusion is due to $1/\vphi^{-1}(\mu) = \xi/\mu$. In particular, we now have an element $y' \in \ML [1/p]$ such that $\vphi(y')=\vphi(\xi)^N x$. Once we consider the map modulo $\vphi^{-i}(\xi)^n$, the element $\vphi(\xi)^N$ becomes invertible, so we can conclude surjectivity of (\ref{eqn:MLlocalisomorphism}).

    The source and target are $\Ainf[p^{-1}]$-submodules of $\vphi^* \MT[p^{-1}]/\vphi^{-i}(\xi)^n$ and $\MT[p^{-1}]/\vphi^{-i}(\xi)^n$, respectively. Therefore, it suffices to show injectivity of
	\begin{align}
		\vphi: \vphi^* \MT[p^{-1}] /\vphi^{-i}(\xi)^n  \to \MT[p^{-1}]/\vphi^{-i}(\xi)^n .
	\end{align}
    Since these are maps of modules over the local ring $\Ainf[p^{-1}]/\vphi^{-i}(\xi)^n$, we can use the fact that this is also the base change of the localization $(\vphi^*\MT[p^{-1}])_{(\vphi^{-i}(\xi))} \xto{\vphi} ({\MT}[p^{-1}])_{(\vphi^{-i}(\xi))}$, which is an isomorphism since $\vphi(\xi)$ is a unit in $\Ainf[p^{-1}]_{(\vphi^{-i}(\xi))}$. 
	Take $x \in\vphi^* \MT[p^{-1}]$, and assume that $\vphi(x) = \vphi^{-i}(\xi)^ny$ for $y \in \MT[p^{-1}]$. 
    Then there exists $x' \in (\vphi^*\MT[p^{-1}])_{(\vphi^{-i}(\xi))}$ such that $\vphi(x')=y \in ({\MT}[p^{-1}])_{(\vphi^{-i}(\xi))}$, and therefore 
	\begin{align}
		x = \vphi^{-i}(\xi)^n x' \in (\vphi^*\MT[p^{-1}]) _{(\vphi^{-i}(\xi))} \label{eqn:xxprime}
	\end{align}
    We conclude that $x' \in \vphi^* \MT[p^{-1}]$ as well. 
\end{proof}

\subsection{Upper bound on limiting behavior of slope}

\begin{lemma}\label{lemma:LpnTmodulestructure}
    For $n \ge 0$, $\mf M(L + p^nT) = \mf M(L) + p^n \mf M(T)$.
\end{lemma}
\begin{proof}
    The definition of $\mf M(L + p^nT) \subset \mf M(T)$ should be the object among finite free BKF-modules corresponding to the subobject $(L + p^nT, \Xi) \subset (T, \Xi)$ in the category $\mc C$. We can check that this corresponds to the final object among finite free BKF modules with a map to $\mf M(T)$ such that the image after applying the functor $(-\otimes W(C^\tilt)^{\vphi=1}$ is contained in $L + p^nT$. First of all, let's check that the finite free BKF module corresponding to $(L + p^n T, \Xi)$ is given by 
    \begin{align} \label{eqn:defnM(Tn')}
            \mf M(L + p^nT) := \mf M(T) \cap \left( (L + p^nT) \otimes_{\Zp} \Ainf[1/\mu] \right).
    \end{align}
    
    By the same reasoning as \Cref{prop:MLisfinitefree}, we can show that $\mf M(L + p^nT)$ is finite free as an $\Ainf$-module. To show that this corresponds to $(L + p^n T, \Xi)$, we only need to check the following two things: first, that $(\mf M(L + p^nT) \otimes W(C^\tilt))^{\vphi=1} = L + p^nT$, and second, that $\mf M(L + p^nT) \otimes_{\Ainf}\BdR^+ = \Xi_T$. These are straightforward to check if we note that $\Ainf \to W(C^\flat)$ and $\Ainf \to \BdR^+$ are both flat\cite[7.6, 6.1]{cesnaviciuskoshikawa2019cohomology}, and flat base change commutes with finite intersections. 
    
    A priori, it is only clear that we have inclusions $\mf M(L), p^n \mf M(T) \subset \mf M(L + p^nT)$. Consider the quotient map
    \[
        \frac{\mf M(T)}{\mf M(L + p^nT)} \onto \frac{\mf M(T)}{\mf M(L) + p^n\mf M(T)} 
    \]
    We want to check this is an isomorphism. By the expression in (\ref{eqn:defnM(Tn')}), the source and target of the surjection are $\mu$-torsion-free, so we have the following commutative diagram. 
    \[
    \begin{tikzcd}
        \frac{\mf M(T)}{\mf M(L + p^nT)} \arrow[r,twoheadrightarrow]\arrow[d,hook] &\frac{\mf M(T)}{\mf M(L) + p^n\mf M(T)} \arrow[d,hook]
        \\
        \frac{\mf M(T)}{\mf M(L + p^nT)}[1/\mu] \arrow[r,"\sim"] &\frac{\mf M(T)}{\mf M(L) + p^n\mf M(T)} [1/\mu]
    \end{tikzcd}
    \]
    This implies that the top horizontal map must be injective as well as surjective. 
\end{proof}
\begin{lemma}
    The quotient $\MT/\mf M(L + p^nT)$ is $\vphi(\xi)$-torsion free. 
\end{lemma}
\begin{proof}
    Using (\ref{eqn:defnM(Tn')}), the same argument works as in \Cref{lemma:MTMLtorsionfree}. 
\end{proof}

\begin{lemma}
    The Nygaard filtration on $\vphi^* \ML$ is induced by that of $\vphi^* \MT$. That is, for all $ r \ge 0$,
    \[
        \Fil^r \vphi^* \ML = \vphi^* \ML \cap \Fil^r \vphi^* \MT.
    \]
    The same is true for $\mf M(L + p^nT)$:
    \[
    \Fil^r \vphi^* \mf M(L + p^nT) = \vphi^* \mf M(L + p^n T) \cap \Fil^r \vphi^* \MT.
    \]
\end{lemma}
\begin{proof}
    By definition, $\Fil^r \vphi^* \ML = \vphi^{-1}(\vphi(\xi)^r \ML)$. By \Cref{lemma:MTMLtorsionfree}, $\vphi(\xi)^r \ML = \ML \cap \vphi(\xi)^r \MT$. One can then check that $\vphi^{-1}(\vphi(\xi)^r \MT )\cap \ML =\vphi^* \ML \cap \Fil^r \vphi^* \MT$. The same holds for $\mf M(L + p^nT)$. 
\end{proof}

\subsubsection{}\label{subsubsec:ineqofdRlengths}
Using the fact that $\Fil^r \vphi^* \mf M(L + p^nT) \supset \Fil^r \vphi^* \mf M(L) + p^n\Fil^r \vphi^*\mf M(T)$, we also have that $\Fil^r \mf M(L + p^nT)_{\dR} \supset \Fil^r \mf M(L)_{\dR} + p^n\Fil^r \mf M(T)_{\dR}$.

\subsubsection{Preliminaries on $\OC$-modules}
We use the same notation as in \cite[\S 2.2.5]{kisinmocznorthcott}. For an $\mcO_C$-module $M$, we say $M$ is \textit{bounded} if there exists an embedding $M \subset \mcO_C^{\oplus m}$ for some finite $m$. If $M \otimes C$ is finite-dimensional for an $\mcO_C$-module $M$, we say $M$ is finite rank, and we call this dimension the \textit{rank} of $M$. For a bounded $\mcO_C$-module $M$, define
\[
    M^+ := \Hom_{\OC}(\mf m_C, M),
\]
where $\mf m_C\subset \OC$ is the maximal ideal. 
We recall that if $f: M \to N$ is a morphism of $\OC$-modules, then $f$ is called an \textit{almost isomorphism} if the kernel and cokernel of $f$ are $\mf m_C$-torsion.

By \cite[Lemma 2.2.6]{kisinmocznorthcott}, if we change $C$ to be spherically closed complete algebraically closed field over $\Qp$ with valuation $v$ which surjects onto $\R$, then for a bounded $\OC$-module $M$, the module $M^+$ is finite free. One can check that when $M$ is bounded, the natural map $M \to M^+$ is injective and an almost isomorphism. In general, $M \to M^+$ has $\mf m$-torsion kernel. If $M$ was already finite free, then $M \cong M^+$.

\begin{lemma}[\cite{stacks-project}, \href{https://stacks.math.columbia.edu/tag/0ASP}{Tag 0ASP}]\label{lemma:OCmodstructurethm}
    Every finitely presented $\OC$-module $M$ is isomorphic to a finite direct sum 
    \[
        M \cong \bigoplus_{i=1}^n \OC/\pi_i
    \]
    for some $\pi_i \in \OC$, not necessarily in $\mf m_C$. 
\end{lemma}

\begin{remark}
    For any complete algebraically closed field $C$ over $\Qp$, we can still define a normalized length function on finitely presented $\OC$-modules. See \cite[\S 7.10]{cesnaviciuskoshikawa2019cohomology}.
\end{remark}

\begin{lemma} \label{lemma:Nr}
    For each $r = 0,\ldots,w$, there exists a finitely presented $\OC$-module $\mf N_r$ such that
    \begin{align}
          \lambda \left( \frac{\gr^r \mf M(T)_{\dR}}{\gr^r \mf M(L + p^nT)_{\dR}} \right) \le \lambda (\mf N_r \otimes_{\mcO_C} \mcO_C/p^n).
    \end{align}
\end{lemma}
\begin{proof}
Note that $\Fil^r\mf M(T)_{\dR}$, $\Fil^r \mf M(L + p^nT)_{\dR}$, and $\Fil^r\mf M(L)_{\dR} + p^n\Fil^r\mf M(T)_{\dR}$ are all submodules of the finite free $\mcO_C$-module $\mf M(T)_{\dR}$ and hence are bounded $\mcO_C$-modules. Since $T$ and $L + p^nT$ are Galois-stable lattices in a crystalline representation, the $\mcO_C$-modules $\Fil^r\mf M(T)_{\dR}$, $\Fil^r \mf M(L + p^nT)_{\dR}$ are base-changed from $\mcO_{F'_v}$ modules for some finite $F'_v/\Qp$, and thus are finite free modules already. However, $\Fil^r\mf M(L)_{\dR}$ is not necessarily finite free. 
One can check using the snake lemma that the maps 
\begin{align}
    &\frac{\Fil^r \MT_{\dR}}{\Fil^r \mf M(L + p^nT)_{\dR}} \to \frac{(\Fil^r \MT_{\dR})^+}{(\Fil^r \mf M(L + p^nT)_{\dR})^+} \\
    &\frac{\Fil^r \MT_{\dR}}{\Fil^r \mf M(L)_{\dR} + p^n\mf M(T)_{\dR}} \to \frac{(\Fil^r \MT_{\dR})^+}{(\Fil^r \mf M(L)_{\dR} + p^n\mf M(T)_{\dR})^+}
\end{align}
are almost isomorphisms mapping to finitely presented $\mcO_C$-modules. The first map is an isomorphism because the domain was already a finitely presented $\mcO_C$-module, but the second map shows that the quotient $\frac{\Fil^r \MT_{\dR}}{\Fil^r \mf M(L)_{\dR} + p^n\mf M(T)_{\dR}}$ can be given a well-defined normalized length. 

From \ref{subsubsec:ineqofdRlengths} and left-exactness of the $M \mapsto M^+$ functor, we have an inequality of normalized lengths
\begin{align}
    \lambda \left(  \frac{\Fil^r \MT_{\dR}}{\Fil^r \mf M(L + p^nT)_{\dR}} \right) \le \lambda \left(  \frac{(\Fil^r \MT_{\dR})^+}{(\Fil^r \mf M(L)_{\dR} + p^n\mf M(T)_{\dR})^+} \right). 
\end{align}
Moreover, one can check that $(\Fil^r\ML_{\dR})^+ + p^n (\Fil^r\MT_{\dR})^+ \into (\Fil^r \mf M(L)_{\dR} + p^n\mf M(T)_{\dR})^+$ is an almost isomorphism, so
\begin{align}
    \lambda \left(  \frac{\Fil^r \MT_{\dR}}{\Fil^r \mf M(L + p^nT)_{\dR}} \right) \le \lambda \left(  \frac{(\Fil^r \MT_{\dR})^+}{(\Fil^r \mf M(L)_{\dR})^+} \bmod{p^n}\right). 
\end{align}

Define
\begin{align}
    \mf N_r:= \frac{(\Fil^r \MTdR)^+/(\Fil^{r+1} \MTdR)^+}{(\Fil^r \MLdR)^+/(\Fil^{r+1} \MLdR)^+}.
\end{align}
This is a finitely presented $\OC$-module.  Then we have 
\begin{align}
    \lambda \left(\frac{\gr^r \MTdR}{\gr^r(\M(L +p^nT))_{\dR} }\right) &\le \lambda(\mf N_r\bmod{p^n}).
\end{align}
\end{proof}

\subsubsection{} \label{subsubsec:rankofNr}
Since $\mf N_r$ is a finitely presented $\OC$-module, we can use \Cref{lemma:OCmodstructurethm} to write $\mf N_r \cong \OC^{\oplus n_r} \oplus \bigoplus_i \OC/\pi_i$ for $\pi_i \in \mf m_C$. Note that $n_r$ is the rank of $\mf N_r$. Then
\begin{align} \label{eqn:limitisrank}
    \lim_{N\to \infty} \frac{1}{N}\lambda(\mf N_r\bmod{p^N}) = n_r. 
\end{align}
One can use the snake lemma to show that the natural map
\begin{align*}
    \frac{\gr^r \MTdR}{\gr^r\MLdR} \to \mf N_r
\end{align*}
is an almost isomorphism.
Therefore they have the same rank as $\OC$-modules, and we can rewrite (\ref{eqn:limitisrank}) as 
\begin{align}
     \lim_{N\to \infty} \frac{1}{N}\lambda(\mf N_r\bmod{p^N})
    &= 
    \dim_C\left( \frac{\gr^r \MTdR}{\gr^r\M(L)_{\dR} } \otimes_{\OC} C\right).
\end{align}

\subsubsection{}
\label{subsubsec:limitmuineq}
We can apply \Cref{lemma:Nr} and \ref{subsubsec:rankofNr} to get the following inequality. 
\begin{align*}
    \lim_n \mu\left(\frac{\mf M(T)}{\mf M(L + p^nT)}\right)
    &=\lim_n \frac{1}{n(\rk T - \rk L)} \sum_{r=0}^w r \cdot \lambda \left( \frac{\gr^r \MT_{\dR}}{\gr^r \M(L + p^n T)_{\dR}}\right)\\
    &=  \frac{1}{(\rk T - \rk L)} \sum_{r=0}^w r \cdot\lim_n \left(\frac{1}{n} \lambda \left( \frac{\gr^r \MT_{\dR}}{\gr^r \M(L + p^n T)_{\dR}} \right)\right)\\
    &\le \frac{1}{(\rk T - \rk L)} \sum_{r=0}^w r \cdot \dim_C \left(\frac{\gr^r \mf M(T)_{\dR }\otimes C}{\gr^r \MLdR \otimes C}\right)\\
    &= \frac{1}{(\rk T - \rk L)} \sum_{r=1}^w \dim_C(\Fil^r \MT_{\dR}[1/p]) - \dim_C(\Fil^r \ML_{\dR}[1/p]).
\end{align*}

\subsection{Limiting slope in terms of Frobenius cokernels}

Recall that the filtration on $\ML_{\dR}$ is defined as follows.
\begin{definition}\label{defn:FilMLdR}
    The $r$th filtered piece of $\ML_{\dR}$ is the image of the composite map
    \[
        \Fil^r \vphi^* \ML \xto{\subseteq} \vphi^* \ML \to \vphi^* \ML \otimes_{\Ainf} \frac{\Ainf}{\vphi(\xi)\Ainf}=: \ML_{\dR}.
    \]
    In particular, this means that
    $
        \Fil^r \ML_{\dR} \cong \frac{\Fil^r \vphi^* \ML}{\Fil^r \vphi^* \ML \cap \vphi(\xi)\vphi^*\ML}. 
    $
    One can check that the denominator is equal to $\vphi(\xi)\Fil^{r-1}\vphi^*\ML$, so that 
    \begin{align}
        \Fil^r \ML_{\dR} \cong \frac{\Fil^r \vphi^* \ML}{\vphi(\xi)\Fil^{r-1} \vphi^* \ML}.
    \end{align}
\end{definition}

\subsubsection{}\label{subsubsec:QvTQvLdefn}
Let $v: \ol{\Q}\to C$ denote an extension of the embedding $v: F \to C$ determined by a place $v \mid p$ of the number field $F$.
As in the proof of \Cref{lemma:nuintegral}, this determines different representations $\Gal_{F_v} \to \GL(T)$, and hence different BKF modules $\mf M_v(T)$.
For convenience of notation, let's write
\[
     Q_v(T) :=  \frac{\mf M_v(T)}{\vphi^*\mf M_v(T)}[1/p], \quad \quad Q_v(L) :=  \frac{\mf M_v(L)}{\vphi^*\mf M_v(L)}[1/p].
\]
We will drop the subscript $v$ when we are working with a single choice of $v$ in context. 

\subsubsection{}
\label{subsubsec:CDrelatingdRtoQT}
Using the description of $\Fil^r \ML_{\dR}$ from \ref{defn:FilMLdR}, we have the following commutative diagram, where $M[a]$ denotes the $a$-torsion of the module $M$. 
\[
\begin{tikzcd}
{\Fil^{r-1} \vphi^* \ML[1/p]} \arrow[dd, "\times \vphi(\xi)"', hook] \arrow[rd, "\frac{\vphi}{\vphi(\xi)^{r-1}}"] \arrow[rr, bend left] &                                  & {Q(L)[\vphi(\xi)^{r-1}]} \arrow[rd, "\subseteq", hook] \arrow[dd, hook]          &                                     \\
                                                                                                                                        & {\ML[1/p]} \arrow[rr, two heads] &                                                                                                               & {Q(L)} \\
{\Fil^r \vphi^* \ML[1/p]} \arrow[ru, "\frac{\vphi}{\vphi(\xi)^r}"'] \arrow[rr, bend right] \arrow[d, two heads]                         &                                  & {Q(L)[\vphi(\xi)^r]} \arrow[ru, "\subseteq"', hook] \arrow[d, two heads]         &                                     \\
{\Fil^r(\ML_{\dR}[1/p])\cong \frac{\Fil^r \vphi^* \ML}{\vphi(\xi)\Fil^{r-1} \vphi^* \ML}[1/p]} \arrow[rr]                               &                                  & {\frac{Q(L)[\vphi(\xi)^r]}{Q(L)[\vphi(\xi)^{r-1}]}} &                                    
\end{tikzcd}
\]
The two surjective vertical maps in the diagram are taking quotients of the preceding inclusions. 
\begin{remark}
    Note that the same diagram works if we replace $\mf M(L)$ with $\mf M(T)$ and $Q(L)$ with $Q(T)$. 
\end{remark}
\begin{remark}\label{claim:FildRisomtocokertorsion}
    One can check that the bottom horizontal map is an isomorphism, and is compatible with inclusion into $\Fil^{r-1}$ on the LHS and multiplication by $\vphi(\xi)$ on the RHS. 
\end{remark}

\subsubsection{How dimensions of $\Fil^r \ML_{\dR}$ reflect the structure of $Q(L)$}
\label{subsubsec:QLandFildR}
Because the cokernel of Frobenius on $\ML$ and $\MT$ are both finitely generated $\vphi(\xi)$-power torsion $\Ainf[1/p]$-modules, we thus can write 
\[
    Q(L) = \frac{\ML}{\vphi^*\ML}[1/p] \cong \bigoplus_{i=1}^{\rk L} \BdR^+/\mf m^{n_i}, \quad Q(T)= \frac{\MT}{\vphi^*\MT}[1/p] \cong \bigoplus_{i=1}^{\rk T} \BdR^+/\mf m^{m_i}
\]
where $n_i, m_i \ge 0$, and $\mf m \subset \BdR^+$ is the maximal ideal. Using this presentation of the Frobenius cokernels and the fact that the bottom horizontal map in the above commutative diagram is an isomorphism, we have that 
\begin{align}
    \dim_C (\Fil^r\ML_{\dR}[1/p]) &= \#\{ i \;:\; n_i \ge r\}.
\end{align}
Then,
\begin{align}
     \sum_{r=1}^w \dim_C (\Fil^r\ML_{\dR}[1/p])&= \ell_{\BdR^+}\left( Q(L) \right). 
\end{align}

\subsubsection{}\label{subsubsec:QvTlengthinvariant}
    We can rewrite $\ell_{\BdR^+}(Q_v(T))$ as the weighted sum of the Hodge numbers of the motive $M$ which gives rise to $T$:
    \begin{align}
        \ell_{\BdR^+}(Q_v(T)) &= \sum_{r=0}^w r \dim_C \gr^r (\mf M_v(T)_{\dR}[1/p])
        \\
        &= \sum_{r=0}^w r \dim_C \gr^r(D_{\dR}(T[1/p]) \otimes_{F_v} C).
    \end{align}
    A priori, this does depend on the choice of $v: \ol{F}\to C$, but is actually determined by the Hodge numbers of the de Rham realization $M_{\dR}$ of the motive $M$ (recall that $M_{\dR}$ is a filtered $F$-vector space). 
    \begin{align}
        \ell_{\BdR^+}(Q_v(T)) = \sum_{r=0}^w r \dim_F \gr^r(M_{\dR}).
    \end{align}
    By Hodge symmetry, we have that $\dim_F \gr^r (M_{\dR}) = \dim_F \gr^{w-r}(M_{\dR})$, so that 
    \begin{align}
         \ell_{\BdR^+}(Q_v(T)) = \frac{w}{2} \dim_F(M_{\dR}) = \frac{w}{2} \rk_{\Zp}(T).
    \end{align}

\subsubsection{}
\label{subsubsec:limnuusingcokerlength}
Let $T_n' := L + p^nT$.
Combining \ref{subsubsec:limitmuineq}, \ref{subsubsec:QvTlengthinvariant}, and \ref{subsubsec:QLandFildR}, we then have that 
    \begin{align*}
        \lim_{n \to \infty} \nu(T, T_n') &= \sum_{v\mid p} \frac{[F_v:\Qp]}{[F:\Q]} \int_{\Gal_F} \lim_n \mu \left(\frac{\mf M_v(T)}{\mf M_v(\sigma L + p^nT)}\right) d \sigma \\
        &\le \sum_{v\mid p} \frac{[F_v:\Qp]}{[F:\Q]} \int_{\Gal_F} \frac{1}{(\rk T - \rk L)} \sum_{r=1}^w \dim_C(\Fil^r \MT_{\dR}[1/p]) \\
        & \hspace{7cm}- \dim_C(\Fil^r \mf M(\sigma L)_{\dR}[1/p]) \; d \sigma \\
        &= \sum_{v\mid p} \frac{[F_v:\Qp]}{[F:\Q]} \int_{\Gal_F} \frac{1}{(\rk T - \rk L)} \left(  \frac{w}{2} \rk(T) - \ell_{\BdR^+}(Q_v(\sigma L))\right) d\sigma.
    \end{align*}
If we assume that the image of $\Gal_F$ is open in $G(\Qp)$, the value of $\ell_{\BdR^+}(Q_v(\sigma L))$ for $\sigma \in \Gal_F$ in a subset of Haar measure 1 will equal the generic value of $\ell_{\BdR^+}(Q_v(g \cdot L))$ for $g \in G(\Qp)$. 

\subsection{Hodge--Tate cocharacters}

\subsubsection{}
\label{subsubsec:weirdBdRstuff}
Let's set up some notation as labeled in the following commutative diagram. Below, $\alpha: \Ainf[1/p]\to \BdR^+$ is induced by the usual completion of $\Ainf[1/p]$ with respect to $\ker(\theta)$. The map $\beta$ is induced by the completion of $\Ainf[1/p]$ with respect to $\ker(\theta\circ \vphi^{-1})$. As the limit of isomorphisms, the middle vertical map $\gamma$ is an isomorphism. 
\[
\begin{tikzcd}
{\Ainf[1/p]} \arrow[r, "\alpha"'] \arrow[d, "\vphi"] \arrow[rr, "\theta", bend left]          & {\lim_n \frac{\Ainf[1/p]}{\xi^n}} \arrow[d, "\gamma:=\lim_n (\vphi)"] \arrow[r, two heads] & C \arrow[d, equal] \\
{\Ainf[1/p]} \arrow[r, "\beta"] \arrow[rr, "\theta \circ \vphi^{-1}"', two heads, bend right] & {\lim_n \frac{\Ainf[1/p]}{\vphi(\xi)^n}} \arrow[r, two heads]                              & C                               
\end{tikzcd}
\]
By \Cref{lemma:BMS4.26}, we have that $\MT [1/\mu] = T \otimes_{\Zp} \Ainf[1/\mu]$. Observe that $\beta(\mu) \in (\BdR^+)^\times$. Combining these, we have 
\begin{align}\label{eqn:MTvsTusingbeta}
    \MT \otimes_{\Ainf, \beta} \BdR^+ = \MT[1/\mu] \otimes_{\Ainf,\beta} \BdR^+ = T \otimes_{\Zp} \BdR^+.
\end{align}
\begin{remark}
    The LHS of the above equation is not the same as $\Xi_T:= \MT \otimes_{\Ainf,\alpha}\BdR^+$ \cite[Theorem 4.28]{BMS2018integral}. Namely, $\vphi^*(\MT) \otimes_{\Ainf, \beta}\BdR^+ = \MT \otimes_{\Ainf, \beta \circ \vphi} \BdR^+ = \MT \otimes_{\Ainf, \gamma \circ \alpha} \BdR^+ = \Xi_T \otimes_{\BdR^+, \gamma} \BdR^+ $. 
\end{remark}

In particular, we have $\mf M(T) \otimes_{\Ainf, \theta \circ \vphi^{-1}} C \cong T \otimes_{\Zp}C$.

\subsubsection{}\label{subsubsec:gradingonCvsp}
By \ref{subsubsec:weirdBdRstuff}, we have a canonical identification of 
$$
\mf M(T) \otimes_{\Ainf, \theta \circ \vphi^{-1}} C \cong T \otimes C.
$$ 
When $\mf M(T)$ comes from a Breuil--Kisin module, i.e.~$T$ is a $\Gal_{F_v}$-stable $\Zp$-lattice in a crystalline representation, one has the Hodge--Tate decomposition $T \otimes C \cong \gr^* D_{\dR}(T[1/p]) \otimes_{F_v} C$. 
Because $\Fil^r D_{\dR}(T[1/p]) \otimes_{F_v} C = \Fil^r \mf M(T)_{\dR}[1/p]$ \cite[2.18]{koshikawa2015heights}, we can rewrite the Hodge--Tate decomposition as
$$
T \otimes C \cong \bigoplus_{r=0}^w \gr^r(\mf M(T)_{\dR}[1/p]).
$$

\subsubsection{}\label{subsubsec:choiceofbasis}
One can choose a basis $e_1,\ldots,e_s,e_{s+1},\ldots, e_n$ of $\mf M(T) \otimes \BdR^+$ such that the image of $e_1,\ldots, e_s$ in $Q(T)$ form a set of generators of $Q(T)$ of minimal size.
We may write 
\[
Q(T) \cong \bigoplus_{i=1}^s \BdR^+/\mf m^{m_i}, \quad m_i \ge 1.
\]
Then a basis of $\imag(\vphi_{\mf M(T)}) \otimes_{\Ainf} \BdR^+ = \ker(\mf M(T) \otimes \BdR^+ \to Q(T))$ is given by 
\[
\vphi(\xi)^{m_1}e_1,\ldots,\vphi(\xi)^{m_s}e_s, e_{s+1},\ldots, e_n.
\]

\begin{lemma}\label{lemma:gradingintermsofbasis}
    The grading on $\mf M(T) \otimes C$ given by the Hodge--Tate decomposition on $T \otimes C$ can be expressed in terms of the basis $\ol{e}_1,\ldots, \ol{e}_n$ of $\mf M(T) \otimes C$, where $\ol{e}_i$ is the image of $e_i$ under the natural map $\mf M(T) \otimes \BdR^+ \to \mf M(T) \otimes C$. In particular, the grading is given by:
    \[
        \begin{aligned}
            &\gr^0(\mf M(T) \otimes C) = \Span_C(\ol{e}_{s+1},\ldots,\ol{e}_n), &\\
            & \gr^{r}(\mf M(T) \otimes C) = \Span_C(\{\ol{e}_i \; : \; 1 \le i \le s, \; r = m_i\}), &r \ge 1.
        \end{aligned}
    \]  
\end{lemma}

\begin{proof}
    From \ref{subsubsec:gradingonCvsp} and the description of $\Fil^0 \mf M(T)_{\dR}[1/p]$ in \ref{subsubsec:CDrelatingdRtoQT}, $\gr^0(T \otimes C) = \ker(\mf M(T) \otimes C \onto Q(T) \otimes C )$. By \ref{subsubsec:choiceofbasis}, we see that the $\gr^0$ is as desired. Moreover, the fact that $\ker(\mf M(T) \otimes C \to Q(T) \otimes C$ exactly one graded piece of $\mf M(T) \otimes C$, one can view $Q(T) \otimes C$ as a sub-graded vector space of $\mf M(T) \otimes C$, corresponding to the positively graded pieces and spanned by $\ol{e}_1,\ldots, \ol{e}_s$. 

    From the description of $\Fil^r \mf M(T)_{\dR}[1/p]$ in \ref{subsubsec:CDrelatingdRtoQT}, 
    \begin{align*}
        \gr^r \mf M(T)_{\dR}[1/p] &= 
        \frac{\frac{\Fil^r \vphi^*\mf M(T)}{\vphi(\xi)\Fil^{r-1}\vphi^*\MT}}
        {\frac{\Fil^{r+1} \vphi^*\mf M(T)}{\vphi(\xi)\Fil^{r}\vphi^*\MT}}[1/p] 
        \xrightarrow[\sim]{\frac{1}{\vphi(\xi)^r}\vphi_{\MT}(-)} 
        \frac{\frac{Q(T)[\vphi(\xi)^r]}{Q(T)[\vphi(\xi)^{r-1}]}}
        {\frac{\vphi(\xi)Q(T)[\vphi(\xi)^{r+1}]}{\vphi(\xi)Q(T)[\vphi(\xi)^{r}]}} 
    \end{align*}
    The image is
    \[  
        \frac{Q(T)[\vphi(\xi)^r]}{\vphi(\xi)Q(T)[\vphi(\xi)^{r+1}] + Q(T)[\vphi(\xi)^{r-1}]}\subset \frac{Q(T)}{\vphi(\xi)Q(T) + Q_T[\vphi(\xi)^{r-1}]}.
    \]
    By our choice of basis, we can view the ambient space $\frac{Q(T)}{\vphi(\xi)Q(T) + Q_T[\vphi(\xi)^{r-1}]}$
    as a subspace of $Q(T) \otimes C$ and thus a subspace of $\mf M(T) \otimes C$, being generated by the elements $\{ \ol{e}_i \; : \; 1 \le i \le s, \; m_i \ge r\}$. The image of $\gr^r \mf M(T)_{\dR}[1/p]$ is therefore spanned by the set $\{ \ol{e}_i \; : \; 1 \le i \le s, \; m_i = r\}$. 
\end{proof}

\subsubsection{}
\label{subsubsec:CDgLQT}
Recall from \ref{subsubsec:limnuusingcokerlength} that we want to consider the values of $\ell_{\BdR^+}(Q(g \cdot L))$ for $g \in G(\Zp)$.

We have the following commutative diagram. 
\[\begin{tikzcd}
	& {g \cdot L \otimes \BdR^+} & {T \otimes \BdR^+} & {\mf M(T) \otimes \BdR^+} & {Q(T)} \\
	{g \cdot L} \\
	& {g \cdot L \otimes C} & {T \otimes C} & {\mf M(T) \otimes_{\Ainf, \theta\circ \vphi^{-1}} C} & {Q(T)/\vphi(\xi)Q(T)}
	\arrow[hook, from=1-2, to=1-3]
	\arrow[two heads, from=1-2, to=3-2]
	\arrow["\sim", from=1-3, to=1-4]
	\arrow[two heads, from=1-3, to=3-3]
	\arrow[two heads, from=1-4, to=1-5]
	\arrow[two heads, from=1-4, to=3-4]
	\arrow[two heads, from=1-5, to=3-5]
	\arrow[from=2-1, to=1-2]
	\arrow[from=2-1, to=3-2]
	\arrow[hook, from=3-2, to=3-3]
	\arrow["\sim", from=3-3, to=3-4]
	\arrow[two heads, from=3-4, to=3-5]
\end{tikzcd}\]

Suppose that $L$ has rank 1. 
Using the basis $\ol{e}_1,\ldots, \ol{e}_n$ of $\mf M(T) \otimes C$ from \Cref{lemma:gradingintermsofbasis}, one can define projections $\proj_{\gr^r (\MT \otimes C)}(L \otimes C)$ of $L \otimes C$ to $\gr^r (T \otimes C)$ for $r \ge 0$. When $L$ has rank 1, define
\[
    r_{\max}(L\otimes C) := \max \{r \ge 0 \; : \; \proj_{\gr^r (\MT \otimes C)}(L \otimes C) \neq 0\}.
\]
When $\rk(L) = i > 1$, the exterior power $\ext^i_{C}(\mf M(T) \otimes C)$ has a natural grading induced from that of $\mf M(T) \otimes C$, and we will consider the quantity $r_{\max}(\ext^i (L \otimes C))$ in \Cref{lemma:lengthrmaxineq}. 
\begin{remark}
    \label{rmk:rmaxandanglemu}
    When $\rk L = 1$, the quantity $r_{\max}$ is equal to the definition of $\angles{\bm{\mu}_v, L \otimes C}^\circ$ in \cite[3.5.11]{kisinmocznorthcott}, if we take the increasing filtration $\Fil^{\circ \bm{\mu}_v}$ on $\mf M(T) \otimes C$ coming from our grading on $\mf M(T) \otimes C$ from \Cref{lemma:gradingintermsofbasis}. However, note that the cocharacter $\bm{\mu}$ used in in \cite[3.5.11]{kisinmocznorthcott} is valued in $G_{C^\tilt}$ rather than $G_{C}$, but other than changing the cocharacter, the definition is the same. 
\end{remark}
\begin{lemma}\label{lemma:lengthrmaxineq} 
    The following inequality holds for any saturated $\Zp$-submodule $L \subset T$, with $\mf M(T)$ effective. Then
    \[
        \ell_{\BdR^+}(Q(L))= \ell_{\BdR^+}(Q(\ext^{\rk L}_{\Zp} L)) \ge r_{\max}(\ext^{\rk L}_{C} (L \otimes C)).
    \]
\end{lemma}
\begin{proof}
    First suppose that $\rk_{\Zp}(L ) = 1$. 
    If $r_{\max} = 0$, then $L \otimes \BdR^+ \subset \imag(\vphi_{\MT}) \otimes \BdR^+ \subset \MT \otimes \BdR^+$, and $\ell_{\BdR^+}(Q(L)) = 0$. 

    Suppose again that $\rk_{\Zp}(L) = 1$, and that $r_{\max} > 0$. 
    Then $\ell_{\BdR^+}(Q(L)) \ge r_{\max}$, because the fact that $L \otimes C$ has a nonzero projection to $\gr^{r_{\max}}(Q(T) \otimes C)$ means that $Q(L)$ is spanned by an element $\sum_{i=1}^s a_i e_i \in Q(T)$, $a_i \in \BdR^+$, such that for an $i_0$ with $m_{i_0}= r_{\max}$, the coefficient $a_{i_0} \in (\BdR^+)^\times$. In particular, $Q(L)$ is not $\mf m^{r_{\max}-1}$-torsion as a torsion $\BdR^+$-module with a single generator. 

    Now suppose that $\rk_{\Zp}(L) = d > 1$. Then $\ell_{\BdR^+}(Q(L)) = \ell_{\BdR^+}(Q(\ext^d_{\Zp} L))$, since both measure the determinant of the map $\vphi^* \ML \otimes \BdR^+ \to \ML \otimes \BdR^+$. 
    
    On the other hand, we now consider the quantity $r_{\max}(\ext^i (L \otimes C))$. We may choose a basis $v_1,\ldots,v_d$ of $L \otimes \BdR^+$ such that the mod $\mf m$ reductions $\ol{v}_1,\ldots, \ol{v}_d$ in $L \otimes C$ satisfy the following condition: For all $i \neq j$, 
    \begin{align}
        \label{eqn:basiscondition}
         \proj_{\gr^{r_i}(\MT\otimes C)}(\Span_{C}(\ol v_i)) \neq \proj_{\gr^{r_j}(\MT \otimes C)}(\Span_{C}(\ol v_j)),
    \end{align}
    where we write $r_i := r_{\max}(\Span_C(\ol v_i))$. Using the natural grading on $\ext^i \mf M(T ) \otimes C$, we then have that 
    $$
    r_{\max}(\ext^i (L \otimes C))=\sum_{i=1}^d r_{\max}(\Span_C(\ol v_i))
    $$
    for this choice of basis. 
    The submodule $\BdR^+ \cdot v_i$ either is contained in $\imag(\vphi_{\mf M(T)})\otimes \BdR^+$ or generates a torsion $\BdR^+$-submodule of $Q(T)$ of length at least $r_{\max}(\Span_C(\ol v_i))$.
    In other words, for each $i=1,\ldots,d$, we have the inequality
    \[
        \ell_{\BdR^+}\left(\imag(\BdR^+ \cdot v_i \to Q(T))\right) \ge r_{\max}(\Span_C(\ol{v}_i)). 
    \]
    The condition (\ref{eqn:basiscondition}) ensures that $\sum_{i=1}^d  \ell_{\BdR^+}\left(\imag(\BdR^+ \cdot v_i \to Q(T))\right) = \ell_{\BdR^+}(Q(L))$. 
\end{proof}

\subsubsection{}
\label{subsubsec:twogenericvalues}
By \Cref{lemma:lengthrmaxineq}, we may consider $L$ to have rank $1$, generated by an element $v \in T$, which has an action of $G$ defined over $\Zp$. The $\BdR^+$-length of the module generated by the image of $g \cdot v$ under the fixed map $T \otimes \BdR^+ \xto{\sim} \mf M(T) \otimes \BdR^+ \onto Q(T)$ divides $g \in G$ into finitely many strata. There is a generic stratum which is Zariski open in $G$, and this corresponds to the maximum value of $\ell_{\BdR^+}(Q(g\cdot L))$.
Using the interpretation of \Cref{rmk:rmaxandanglemu}, the generic value of $r_{\max}(g \cdot L \otimes C)$ for $g \in G(C)$ is also the maximum value over all $g \in G(C)$ and is achieved by $g$ in a Zariski open subset of $G$.

Let $W$ be the $C$-vector space of $\MT \otimes C$ spanned by $g \cdot L \otimes C$ for $g \in G(C)$. Let $q_0$ be the least $q$ such that $\proj_{\gr^{w-q}(\MT \otimes C)}(W)\neq 0$. Recall that $w$ comes from the weight of the Hodge structure of the motive that $T$ is associated to.
Let $p_0$ be the least $p$ such that $\proj_{\gr^{p}(\MT \otimes C)}(W)\neq 0$.

Recall that the cocharacter $\bm{\mu}_v$ which defines the grading on $\MT \otimes C$ has image inside of $G$, and so $W$ is stable under the action of $\bm{\mu}_v$, hence inherits a grading from $\bm \mu_v$ compatible with that of $\MT \otimes C$. 
Therefore $W \subset \bigoplus_{r = p_0}^{w-q_0} \gr^r (\MT \otimes C)$, and $W$ is \emph{not} contained in $\bigoplus_{r = p_0+1}^{w-q_0} \gr^r (\MT \otimes C)$ or $\bigoplus_{r = p_0}^{w-q_0-1} \gr^r (\MT \otimes C)$.

Using this notation, we have the following. 

\begin{corollary}\label{cor:genericlengthQ}
    The generic length $\ell_{\BdR^+}(Q_v(g\cdot L)) $ for $g \in G(\BdR^+)$ is $\ge w - q_0$.
\end{corollary}
\begin{proof}
    By definition of $w-q_0$ in \ref{subsubsec:twogenericvalues}, a generic $g \in G(\BdR^+)$ will satisfy $r_{\max}(g \cdot L \otimes C) = w- q_0$. By the generic behavior discussed in \ref{subsubsec:twogenericvalues} and \Cref{lemma:lengthrmaxineq}, $\ell_{\BdR^+}(Q_v(g\cdot L)) \ge w - q_0$. 
\end{proof}

\subsubsection{}
Let $\bm{\mu}'_v: \Gm \to G_C$ denote the cocharacter given by the Hodge decomposition (see \cite[Proof of Lemma 2.4.4]{kisinmocznorthcott}). 
Consider instead the conjugate embedding $c \circ v: \ol F \to C$. By \cite[Lemma 2.4.4(3)]{kisinmocznorthcott}, $\bm{\mu}_v \sim \bm \mu_v'$ and $\bm \mu_{c\circ v} \sim \bm \mu_{c \circ v}'$. 
By Hodge symmetry, the product 
\begin{align}\label{eqn:cocharHodgesymmetry}
    \bm \mu_v'  \bm \mu_{c\circ v}' = \bm w^{-1},
\end{align}
where $\bm w$ is the weight cocharacter. In particular, $\bm w(z)$ acts by $ \times z^{-w}$ on singular cohomology. 

For the two $G$-conjugate cocharacters $\bm \mu_v \sim \bm \mu_v'$ acting on the $G$-stable vector space $W \otimes C$, the maximum degree of a nonzero graded piece of $W \otimes C$ is the same for both cocharacters, namely $w-q_0$. The same is true for $\bm \mu_{c\circ v} \sim \bm \mu_{c \circ v}'$ acting on $W \otimes C$, whose maximum degree of nonzero graded piece of $W \otimes C$ is $w-p_0$ by the relation in (\ref{eqn:cocharHodgesymmetry}). 

\begin{prop}\label{prop:ineqconjpair}
    The following inequality holds when $\rk L = 1$.
    \[
        \max_{g \in G(\BdR^+)}\left( \ell_{\BdR^+}(Q_v(g\cdot L)) \right) +\max_{g \in G(\BdR^+)} \left(\ell_{\BdR^+}(Q_{c \circ v}(g\cdot L))\right) \ge w.
    \]
    If equality is achieved, then $\bm \mu_v$ is stable on $L \otimes C$.
\end{prop}
\begin{remark}
    If $\rk L > 1$, then the inequality is true if we replace $L$ with $\ext^{\rk L} L$, and we replace $w$ on the right-hand side of the inequality with $\rk (L) \cdot w$, which is the weight of the Hodge structure formed by the $\rk(L)$-th exterior product. 
\end{remark}
\begin{proof}
    By \Cref{cor:genericlengthQ}, 
    $$
        \max_{g \in G(\BdR^+)} \ell_{\BdR^+}(Q_v(g\cdot L)) +\max_{g \in G(\BdR^+)} \ell_{\BdR^+}(Q_{c \circ v}(g\cdot L)) \ge (w-q_0) + (w-p_0) \ge w,
    $$
    where the final inequality comes from the fact that $p_0 + q_0 \le w$. If equality holds, then $p_0 + q_0 = w$, which implies that $W$ consists of only one graded piece. Therefore $\bm \mu_v$ acts by a scalar on $W$, and $L\otimes C \subset W$ is preserved by the action of $\bm \mu_v$. 
\end{proof}

\subsection{Proof of the $p$-power isogeny case of the theorem}\label{subsec:fixedprimeheightbound}

Now we can compute what happens to the change in height as $n \to \infty$. 
\begin{align}
    h(M'_n)-h(M) &= \log_p \#(T/T_n') \left( \frac{w}{2} - \nu(T,T_n')\right)
\end{align}
Continuing from \ref{subsubsec:limnuusingcokerlength}, we have that
\begin{align}
    \lim_{N \to \infty }\nu(T,T_N') &\le  \sum_{v\mid p} \frac{[F_v:\Qp]}{[F:\Q]} \int_{\Gal_F} \frac{1}{(\rk T - \rk L)} \left(  \frac{w}{2} \rk(T) - \ell_{\BdR^+}(Q_v(\sigma L))\right) d\sigma \\
    & = \frac{1}{\rk(T)-\rk(L)}
    \sum_{v\mid p} \frac{[F_v:\Qp]}{[F:\Q]} \int_{G(\Zp)} \frac{w}{2} \rk(T) - \ell_{\BdR^+}( Q_v(g\cdot L))  dg,
\end{align}
where we use the fact that the image of $\Gal_F$ is open in $G(\Zp)$. From the generic value discussion in \ref{subsubsec:twogenericvalues}, we can replace the integral with the integrand's generic value, namely
\begin{align}
    \frac{1}{\rk(T)-\rk(L)}
    \sum_{v\mid p} \frac{[F_v:\Qp]}{[F:\Q]}\left( \frac{w}{2} \rk(T) - \max_{g \in G(\Zp)} \ell_{\BdR^+}( Q_v(g\cdot L)) \right).
\end{align}
We may replace $F$ with its normal closure so that $F/\Q$ is Galois, hence $[F_v:\Qp]$ is constant for all $v \mid p$. Let $s(p)$ denote the number of places of $F$ above $p$, so that $[F_v:\Qp]/[F:\Q] = 1/s(p)$. Then we may apply \Cref{prop:ineqconjpair} as follows:
\begin{align}
    \lim_{N \to \infty }\nu(T,T_N') 
    &\le \frac{1}{\rk(T)-\rk(L)} \cdot \frac{1}{s(p)}
    \sum_{v\mid p} \left( \frac{w}{2} \rk(T) - \max_{g \in G(\Zp)} \ell_{\BdR^+}( Q_v(g\cdot L)) \right) \\
    &= \frac{1}{\rk(T)-\rk(L)} \cdot \frac{1}{2 s(p)}
    \sum_{v\mid p}\left(  w \rk(T) - \max_{g \in G(\Zp)}\ell_{\BdR^+}( Q_v(g\cdot L)) -  \max_{g \in G(\Zp)}\ell_{\BdR^+}( Q_{c \circ v}(g\cdot L))  \right) \\
    &\le \frac{1}{\rk(T)-\rk(L)} \cdot \frac{1}{2 s(p)}
    \sum_{v\mid p}\left(  w \rk(T) - w \rk(L) \right) \\
    &= \frac{w}{2}.
\end{align}
By \Cref{prop:ineqconjpair}, if the inequality is an equality, then $L \otimes C$ is stable under $\bm \mu_v$ for every embedding $v: \ol{F} \into C$.

By the Mumford--Tate conjecture, the images of $\mu_v$ for all $v: \ol{F} \into C$ generate $G$ as an algebraic group over $C$. Therefore $L \otimes C$ is preserved by the action of $G(C)$ on $T \otimes C$. This action agrees with the action of $G(\Zp)$ on $T$, so $G(\Zp)$ (and thus a finite-index open subgroup of $\Gal_F$) preserves $L \subset T$. 

In particular, if the inequality is an equality, then there exists a finite extension $E/F$ of number fields such that $L$ is a $\Gal_E$-subrepresentation of $T$. The $p$-power isogenies of Koshikawa motives defined by the finite-index inclusion $L + p^nT \subset T$ of the $p$-part of the \'etale realization thus are isogenies defined over $E$. 
By \cite[Theorem 9.8]{koshikawa2015heights}, these may only attain finitely many values for their heights.

%% file: case2.tex
\section{Main result: varying prime case}
\label{section:varyingprime}

\begin{prop} \label{prop:differentprimes}
    Fix a pure weight $w$ Koshikawa motive ${ M} = { H}^w(X)$. Assume the adelic Mumford--Tate conjecture, absolute Hodge conjecture, and $D_{\dR}$ compatibility conjecture (\Cref{conj:adelicMT}, \ref{conj:blasius}, \ref{conj:absolutehodge}) hold for $ M$. For any constant $c > 0$, there exists a constant $N = N(c)$ such that the set 
    \[
        \mscr{I}({ M})_{\ol{F}}^{<c}:=  \{ M' \in \mathscr{I}({ M})_{\ol{F}} \; : \; h({ M}')< c\}
    \]
    consists of isogenies $\iota: \hat{T}' \into \hat{T}$ with $\#(\hat{T}/\hat{T}')$ divisible by only primes $p < N(c)$. 
\end{prop}

Then, for each of those finitely many primes $p$, \Cref{thm:fixedprimecase} shows that there can only be finitely many height values of elements in $\mscr{I}({ M})_{\ol{F}}^{<c}$ which come from $p$-power isogenies, and we have the main result (\Cref{mainthmB}). 

The rest of this section is dedicated to proving \Cref{prop:differentprimes}. We give an overview here of the method of Kisin--Mocz and note where we have to make changes for the setting of Kato--Koshikawa heights. 

\td{Establish notation conventions which are different from before.}

\td{I do not understand the remark 3.5.7 about Serre's conjecture recovering the HT cocharacter from the inertial weights. Need to look at Caruso.}

\nts{Outline of the proof: For large enough $p$, we are in the Fontaine-Laffaille case by Gee--Liu--Savitt (3.3.8). Moreover, you know the FL weights and the structure of $\mf M / \vphi^* \mf M$. The FL case is what allows you to get bounds on the slope of any line $\mf N(L) \subset \mf M$ (3.4.13). 
}
\nts{
Being in $G^{\bm \mu}\text - \Mod ^\vphi_{/k[[u]]}$ can be interpreted to mean that $\vphi_{\mf M_V}$ can be written as a diagonal matrix such that $\vphi(\vec{e}_i) = u^{p_i}\vec{e}_i$, where the $\{p_i\}$ are exactly the HT weights given by $\bm \mu$. This is basically the condition for being FL type as a mod $p$ BK module with the FURTHER condition that the weights showing up in the FL definition have to match the weights of $\bm \mu$. So this is forcing the action of the Frobenius on the mod $p$ BK module to match up with the HT weights.
}
\nts{
3.5.6 seems to be boosting 3.3.8 to the situation of adding in the expectation that the Galois representation factors through $G$. It says that when the HT weights of a crystalline Galois rep are small relative to $p$, then the mod $p$ BK module coming from the crystalline Galois rep will be of FL type and the structure of $\mf M / \vphi^* \mf M $ will match the HT weights.
Moreover, when the crystalline Galois rep factors through $G$, this might impose that $p$ needs to be a bit larger depending on $G$. I am not sure about this part of the argument. 
}

\nts{The HT cocharacter has a representative which is $G_k$-valued. This is the cocharacter used in the definition of $G^{\bm \mu}\text - \Mod$. The result of 3.5.6 allows you to say that the FL }

\subsubsection{Dependence of $p \gg 0$} \label{sssec:dependenceonp}
For clarity, we list the instances in the proof of \Cref{prop:differentprimes} where we place a condition on $p$ being sufficiently large.
\begin{itemize}
    \item to make sure $X$ has good reduction at all places above $p$ (so that the Galois representation is crystalline)
    \item so that $p$ is bigger than all the Hodge--Tate weights, so that the BK module is of Fontaine--Laffaille type (\cite{kisinmocznorthcott}, 3.3.8)
    \item to make sure $F$ is totally unramified at all places above $p$, which is required for the Fontaine--Laffaille condition to be satisfied
    \item to be in the $G^\mu\text-\Mod$ category (which is a different condition from the previous condition in the case $G_{MT}\neq GL_n$)
    \item to satisfy some bound depending on the action of $G$ on the $\Z[1/N]$-module $V$ (\cite{kisinmocznorthcott}, 3.5.9)
    \item to ensure that the representation $V \otimes \Fp$ is a semisimple representation and irreducible $\Q$-subspaces of $V_{\Q}$ stay irreducible mod $p$ (\cite{kisinmocznorthcott}, 3.5.10)
    \item such that (in the statement of the adelic MT conjecture) the lift $\Gal_E \to G'(\Zp)$ is surjective. 
\end{itemize}


\subsection{Switching to subobjects instead of quotients of BK modules}
\label{subsec:subobjectsnotquotients}
Suppose that $f: { M}' \to { M}$ is a $p$-power isogeny such that $T_p / T_p'$ is $p$-torsion. 
There is a map of Koshikawa motives $g: { M} \to { M}'$ such that $f \circ g: { M} \to { M}$ is the $p$-power isogeny defined by the injection $T_p \xto{\times p} T_p$. We can thus compute 
\begin{align*}
    h({ M}') - h({{M}}) 
        &= \log \#(T_p/T_p')  \left(  \frac{w}{2} - \nu(T_p, T_p')\right) \\
    h({ M}) - h({{M}}') 
        &= \log \#(T_p'/pT_p)  \left(  \frac{w}{2} - \nu(T_p', pT_p)\right).
\end{align*}

For a given prime $p$, there is a finite Galois extension $F'/F$ such that all of the finitely many $\Zp$-submodules $T_0$ such that $pT_p \subset T_0 \subset T_p$ are stable under the action of $\Gal_{F'}$. In other words, the action of $\Gal_F$ on such sublattices $T_0$ factors through the finite quotient $\Gal(F'/F)$.
From \Cref{lemma:nuintegral}, we have
\begin{align}
    \nu(T_p', pT_p) = \sum_{v\mid p} \frac{[F_v:\Qp]}{[F:\Q]} \cdot \frac{1}{|\Gal(F'/F)|} \sum_{\sigma \in \Gal(F'/F)}\mu \left(\frac{\mf M_{v}(\sigma T_p')}{\mf M_{v}(pT_p)}\right) 
\end{align}

The above expression involves slopes of various submodules of $\mf M(T_p/pT_p) \otimes_{\mf S, \vphi} \Ainf = \mf M(T_p)$. Given the data of an $\Fp$-subspace $L$ of $T_p/pT_p$, one can associate the submodule
\begin{align} \label{eqn:NLdef}
    \mf N(L) := (L \otimes_{\Fp} C^\tilt) \cap (\mf M(T) \otimes \mcO_{C^\tilt}). 
\end{align}
Observe that this is analogous to the definition of the BKF module $\mf M(L)$ in \ref{ss:constructionML}, as both come from the canonical identification of $\mf M(T) \otimes_{\Ainf} W(C^\tilt) \cong T \otimes_{\Zp} W(C^\tilt)$. In particular, $\mf N(\sigma T_p'/pT_p) = \frac{\mf M(\sigma T_p')}{\mf M(pT_p)}$, where the right-hand side is viewed as a torsion BKF module. To summarize, we have the following.
\begin{align}\label{eqn:dualnu}
    \nu(T_p', pT_p) = \sum_{v\mid p} \frac{[F_v:\Qp]}{[F:\Q]} \cdot \frac{1}{|\Gal(F'/F)|} \sum_{\sigma \in \Gal(F'/F)}\mu \left( \mf N(\sigma \cdot L ) \right),
\end{align}
where $L := T_p'/pT_p \subset T_p/pT_p$.

\subsection{Replacing Galois group action with an algebraic group action}\label{subsec:galoistoalgebraicgroup}

The adelic Mumford--Tate assumption (\Cref{conj:nonHodgemaxadelicMT}) ensures that there is a finite extension $E/F$ such that for $p \gg 0$,
\[
    \imag(\Gal_E \to G_{MT}(\Zp)) = \imag(G'(\Zp) \xto{\pi} G_{MT}(\Zp))
\]
In particular, the degree $[E:F]$ provides a bound independent of $p$ on the index of the subgroup
\[
    \imag(G'(\Fp) \xto{\pi} G_{MT}(\Fp)) = \imag(\Gal_E \to G_{MT}(\Fp)) \subseteq \imag(\Gal_F \to G_{MT}(\Fp)) .
\]

Therefore we can replace the summation over a Galois orbit in (\ref{eqn:dualnu}) with summation over a $G'(\F_p)$-orbit: choose $t \le [E:F]$ coset representatives $h_1,\ldots,h_t$ for $\pi(G'(\Fp))$ in $\imag(\Gal_F \to G_{MT}(\Fp))$. We may also assume without loss of generality that $F/\Q$ is Galois, so that $[F_v:\Qp]/[F:\Q] = 1/s(p)$, where $s(p)$ is the number of places of $F$ above $p$. Then
\begin{align}
    \label{eqn:nuG-orbit}
    \nu(T_p', pT_p) = \frac{1}{s(p)t} \sum_{\substack{v\mid p \\ j=1,\ldots,t}} \frac{1}{|G'(\Fp)|} \sum_{g' \in G'(\Fp)}\mu \left( \mf N(g' \cdot h_j L ) \right).
\end{align}
The problem is then reduced to studying the variation of slopes of torsion BK submodules $\mf N(g' \cdot h_j L)$ in a given BK module $\mf M(T_p) \bmod p$ for all $g' \in G'(\Fp)$, where $h_j L \subset T_p/p$ is a fixed $\Fp$-vector subspace.

In order to obtain lower bounds on such slopes which depend on $p$, Kisin--Mocz work with $p$-torsion BK modules of Fontaine--Laffaille type, whose structure and slopes are more tractable to study. 
In particular, Fontaine--Laffaille BK modules have a good structure theory as extensions of irreducible FL modules, and moreover, the Hodge--Tate grading on $\mb V(\mf M)$ can be matched with the structure of the cokernel of $\vphi_{\mf M}$. This allows us to link the slope of a torsion BK module of FL type to the Hodge filtration, which ultimately provides the bounds on the slope to prove the main result. 

In addition, replacing the Galois group action with the action of an algebraic group allows us to use results from algebraic geometry and representations of algebraic groups to study the average slope $\mu(\mf N(g' L))$ as $g' \in G'(\Fp)$ varies. 

\subsection{Fontaine--Laffaille modules}
Let $K$ be an algebraic extension of $\Qp$ with a discrete valuation. In this subsection, we will assume that $k$, the residue field of $K$, is algebraically closed, and that $K$ is totally unramified, i.e.~$K=K_0$, and we let $E(u)=u-p$.

\begin{defn}
    Let $\mf M \in \Mod^{\vphi}_{/k[[u]]}$, i.e.~a torsion Breuil-Kisin module which is a finite free $k[[u]]$-module. We say that $\mf M$ is of \emph{Fontaine--Laffaille (FL)} type if $\mf M$ is effective and satisfies the following condition: there exist bases of $\vphi^* \mf M$ and $\mf M$ such that the map $\vphi_{\mf M}$ is given by a diagonal matrix with entries $u^{i_1},\ldots, u^{i_r}$ with $i_j \in [0,p-1]$.

    We write $\Mod^{\vphi, FL}_{/k[[u]]} \subset \Mod^{\vphi}_{/k[[u]]}$ for the full subcategory of modules of FL type. 
\end{defn}

\begin{rmk}
    Note that the definition implies that $\Mod^{\vphi, FL}_{/k[[u]]} \subset \Mod^{\vphi, \le p-1}_{/k[[u]]} $. We will denote by $\Mod^{\vphi,\le h, FL}_{/k[[u]]} \subset \Mod^{\vphi, FL}_{/k[[u]]}$ the full subcategory of objects in $\Mod^{\vphi, \le h}_{/k[[u]]}$.
\end{rmk}

The following lemma is due to Kisin--Mocz. Recall the definition of the multiset $\mc H(\mf M)$ from \ref{subsubsec:mumaxHdefn}.
\begin{lemma}
    [\cite{kisinmocznorthcott}, Lemma 3.3.5]\label{lemma:Hextension}
    The subcategory $\Mod^{\vphi,FL}_{/k[[u]]} \subset \Mod^{\vphi,\le p-1}_{/k[[u]]}$ is closed under quotients and saturated subobjects. Moreover, for extensions in $\Mod^{\vphi,FL}_{/k[[u]]}$ of the form
    \[
    0 \to \mf M_1 \to \mf M \to \mf M_2 \to 0,
    \]
    we have that $\mc H(\mf M) = \mc H(\mf M_1) \sqcup \mc H(\mf M_2)$ as multisets. 
\end{lemma}

This property of $\mc H(\mf M)$ allows one to reduce studying the valuations of irreducible modules of FL type. 
All modules in $\Mod^{\vphi,FL}_{/k[[u]]}$, being Noetherian, can be written as a finite number of extensions of irreducible modules of FL type. These irreducible modules $\mf M_{irr}$ all take an explicit form. 
Using this, Kisin--Mocz provide a lower bound of the slope of any saturated rank 1 $\vphi$-stable submodule $\mf N \subset \mf M_{irr}$ in terms of $\mu_{\max}(\mf M_{irr})$. 
They then combine this with \Cref{lemma:Hextension} to give the following lower bound on the slope of a generic saturated rank 1 $\vphi$-stable submodule $\mf N(L)\subset \mf M$ coming from an $\Fp$-line $L \subset \mb V(\mf M)$ for any $\mf M$ of FL type. 

\begin{corollary}[\cite{kisinmocznorthcott}, Corollary 3.4.13]
    Suppose $\mf M \in \Mod^{\vphi, FL}_{/k[[u]]}$. Then there exists a proper subspace $W \subset \mb V(\mf M)$ such that for any $\Fp$-line $L \subset \mb V(\mf M)$ with $L \not\subset W$, we have
    \[
        \mu(\mf N(L)) \ge \left( 1 - \frac{1}{p} \right) \mu_{\max}(\mf M). 
    \]
\end{corollary}

Kisin and Mocz are able to work with the relatively well-behaved FL modules for the varying prime case due to the following result of \cite{gee2015weight} (see \cite[Proposition 3.3.8]{kisinmocznorthcott}): For sufficiently large primes $p$, if an effective BK module $\mf M$ comes from the mod $p$ reduction of a lattice in a crystalline Galois representation with Hodge--Tate weights bounded by $[-p+1,0]$, then $\mf M$ will be of FL type and the valuations of $\mf M/\vphi^* \mf M$ are $\mc H(\mf M) = \{ i_1,\ldots, i_r\}$, matching the Hodge--Tate weights with multiplicity, up to flipping the sign of the weights. 

\subsubsection{}
This section closely follows \cite[\S 3.5]{kisinmocznorthcott}. 
Recall that we assume that our Galois representation takes a certain form, namely $\rho: \Gal_{F_v} \to G(\Zp)$, where $G$ is a reductive group defined over $\Z$. In our situation, $G= G_{MT}$ is the Mumford--Tate group. Composing $\rho$ with any representation $V$ in $\Rep_{\Zp}(G)$ would define a lattice in a crystalline representation, which has an associated BK module.
We can then view this process as a functor $\mf M(-): \Rep_{\Zp}(G) \to \Mod^{\vphi}_{/\mf S}$. 
The same can be repeated in the $p$-torsion case to define $\mf M(-): \Rep_{\Fp}(G) \to \Mod^{\vphi}_{/k[[u]]}$. 
Within the category $G \text - \Mod^{\vphi}_{/k[[u]]}$ of all such functors $\mf M(-)$ in the mod $p$ case, there are subcategories $G^{\bm \mu} \text - \Mod^{\vphi}_{/k[[u]]}$ defined by each cocharacter $\bm \mu: \Gm \to G_k$, where the action of $\vphi$ on $\mf M(V)$ must be compatible in some sense with the specified $\bm \mu$.
In fact, one can check that this condition amounts to saying that this mod $p$ BK module is of FL type.

Suppose that we are consider a $p$-torsion functor $\mf M(-) \in G \text - \Mod^{\vphi}_{/k[[u]]}$ coming from the choice of a crystalline representation $\rho: \Gal_{F_v} \to G(\Zp)$ with Hodge--Tate cocharacter $\bm \mu^{-1}$, which has a conjugacy class representive defined over the residue field $k_v$ of $F_v$. In this case, Kisin--Mocz adjust the result of Gee--Liu--Savitt to show that when $p$ is large enough depending on $G$ and the Hodge--Tate weights of $\bm \mu$, $\mf M(-)$ is in fact in $G^{\bm \mu} \text - \Mod^{\vphi}_{/k[[u]]}$. This ensures that for any choice of representation $V \in \Rep_{\Fp}(G)$, the resulting mod $p$ reduction of a crystalline representation has associated $p$-torsion BK module of FL type.

\subsubsection{}
The upshot of the above paragraph is that for $p \gg 0$, the mod $p$ reduction of the crystalline Galois representation will have associated torsion BK module which is of FL type, and thus we can make use of the bounds on slopes in FL type modules.

\begin{defn}[\cite{kisinmocznorthcott}, 3.5.11]
    Let $\bm \mu: \Gm \to G_k$ be a cocharacter. For $V \in \Rep_{\Fp}G$, let $\Fil^{\circ \bm \mu}$ be the increasing filtration that $\bm \mu$ defines on $V \otimes_{\Fp} C^\tilt$. 
    Let $L \subset V \otimes C^\tilt$ be a nonzero $C^\tilt$-linear subspace. If $\dim L = 1$, define $\angles{\bm \mu, L}^\circ$ as the greatest integer such that $\Fil^{\circ \bm \mu, m-1}(L) =0$, i.e. $L \not\subseteq \bigoplus_{i \le m-1} \gr^i_{\bm \mu}(V_{C^\tilt})$ but $L \subseteq \bigoplus_{i \le m} \gr^i_{\bm \mu}(V_{C^\tilt})$.
    For $\dim L = d$, let $\angles{\bm \mu, L}^\circ := \frac{\angles{\bm \mu, \ext^d L}^\circ}{d}$, where $\ext^d L \subseteq \ext^d V_{C^\tilt}$ has the induced grading. Define 
    \[
    \angles{\bm \mu, L}^\circ_{G} := \max_{g \in G(C^\tilt)} \angles{\bm \mu, g \cdot L}^\circ.
    \]
\end{defn}
\begin{remark}
    Suppose that $T$ is a $\Gal_{F_v}$-stable lattice in a crystalline representation which factors as $\Gal_{F_v} \to G(\Zp) \to \GL(T)$, and let $\mf M := \mf M(T)$ be the corresponding finite free BK module. Suppose that $\mf M \bmod{p}$ is of FL type. If $p$ is sufficiently large, one can choose a conjugacy class representative of the Hodge--Tate cocharacter of $T$ which is defined over $\mcO_{F_v}$ and hence can be reduced mod $v$ to a character $\bm \mu^{-1}: \Gm \to G_k$. The result of Gee--Liu--Savitt \cite[Prop.~3.3.8]{kisinmocznorthcott} ensures that the multiplicity of nonzero graded degrees of $\bm \mu$ matches $\mc H(\mf M \bmod{p})$, therefore both sets have the same maximum element. 

    Now suppose that $G(\Fp)$ acting on $T/p$ is irreducible. 
    If $L \subset T \otimes C^\tilt$ is a line, then 
    $$ \mu_{\max}(\mf M \bmod{p})=\angles{\bm \mu, L}^\circ_{G},$$
    because the $G(C^\tilt)$-orbit of $L$ will span $T \otimes C^\tilt$. 
\end{remark}

Recall that $G_{MT}$ is an algebraic group scheme over $\Z$. 
\begin{proposition}[\cite{kisinmocznorthcott}, 3.5.14]
    \label{prop:avgslopebound}
    Choose a finite free $\Z[1/N]$-module $V$ with an algebraic representation $G \to \GL(V)$. There exists a constant $c_V$ such that for $p \gg 0$, if $\rho: \Gal_{F_v} \to G(\Zp)$ is a crystalline representation with Hodge--Tate cocharacter $\bm \mu^{-1}$, and $W \subset V_{\Fp}$ is any nonzero $\Fp$-subspace, then
    \[
        \frac{1}{|G(\Fp)|} \sum_{g \in G(\Fp)} \mu (\mf N(g \cdot W)) \ge \left(1 - \frac{c_V}{p}\right)\angles{\bm \mu, W}^\circ_G. 
    \]
\end{proposition}

\subsection{Bound on slopes}
We once again use the adelic MT conjecture, absolute Hodge conjecture and compatibility with $D_{\dR}$ assumption (\Cref{conj:nonHodgemaxadelicMT}, \ref{conj:blasius}, \ref{conj:absolutehodge}) to relate mod $p$ Hodge--Tate cocharacters for conjugate embeddings of $\ol{F}$ to the weight cocharacter. This relation allows us to lower-bound the height change by an expression in terms of $p$.
\begin{lemma}
    [\cite{kisinmocznorthcott}, Lemma 3.6.5]
    \label{lemma:conjpairmodpineq}
    Fix a place $v: \ol{F} \to C$ extending $v \mid p$ of $F$. 
    Let $\bm \mu_v$ be the Hodge--Tate cocharacter for the $\Gal_{F_v}$-representation $T_p$. It has a $G_{MT}$-conjugacy class representative which can be defined over the residue field $k$, which we also denote $\bm \mu_v$.
    Let $c \circ v$ denote the embedding $\ol{F} \to C$ using the fixed isomorphism $C \cong \C$ and complex conjugation $c \in \Aut(\C)$. Then for any $C^\tilt$-subspace $W \subset T_p/pT_p \otimes C^\tilt$, we have that 
    \[
        \angles{\bm \mu_v^{-1}, W}^\circ_{G_{MT}} + \angles{\bm \mu_{c \circ v}^{-1}, W}^\circ_{G_{MT}} \ge \angles{\bm w, W} = w,
    \]
    and equality holds only if $W$ is stable by $\bm \mu_v$. 
\end{lemma}
\td{Might need to apply this to $G'$ instead of $G_{MT}$.}

\begin{remark}\label{rmk:conjpairmodpineq}
    \begin{enumerate}[(1)]
    \item Note the switch in notation between the Hodge--Tate cocharacter in \Cref{prop:avgslopebound} and in \Cref{lemma:conjpairmodpineq}. 
    \item The left-hand side of the inequality is always an integer. So if equality does not hold, than the left-hand side is at least $w + 1$. 
\end{enumerate}
\end{remark}

By applying \Cref{prop:avgslopebound} to the $G'(\Fp)$-summation in the expression for $\nu(T'_p, pT_p)$ in (\ref{eqn:nuG-orbit}), and then applying \Cref{lemma:conjpairmodpineq} to the conjugate places $v \mid p$, we get the following bound on $\nu$, analogously to \cite[Proposition 3.6.3]{kisinmocznorthcott}.

\begin{corollary}\label{cor:ptorsionheightbound}
    Assume the adelic Mumford--Tate conjecture, absolute Hodge conjecture, and $D_{\dR}$ compatibility in \Cref{conj:adelicMT}, \ref{conj:absolutehodge}, and \ref{conj:blasius} for ${ M} = ({ M}_{\Q}, \hat{T})$. Up to replacing $F$ with a finite extension, assume that $F$ contains the field $E$ from \S \ref{subsec:galoistoalgebraicgroup}, which was independent of $p$. Choose $p \gg 0$.
    Assume that $T_p' \subset T_p$ is a finite-index sublattice such that $\Gal_F$ does not preserve $T_p'$, and that the $\Gal_F$-orbit of $T_p'$ generates $T_p$. 
    Then there is a constant $c_{{ M}} = c({ M})$, not depending on the choice of $p$ or $T_p'$ above, such that the following inequality holds.
    \begin{align}
        \frac{w}{2} - \nu(T_p,T'_p) \ge \frac{1}{2[F:\Q](\rk(T_p)-1)} \left( 1 - \frac{c_{ M}}{p} \right).
    \end{align}
    Observe that the condition on $T'_p$ forces $\rk T > 1$, if such a $T'_p$ exists. 
\end{corollary}

\begin{proof}
    Since we are only working with $\Zp$-lattices, we will use $T$ and $T'$ to denote $T_p$ and $T_p'$, respectively. 
    Let's first assume that $pT \subset T' \subset T$.
    Suppose that for every $v \mid p$, we have that the inequality in \Cref{lemma:conjpairmodpineq} is an equality. Then the Hodge--Tate cocharacter $\bm \mu_v$ and all $G'$-conjugates are stable on $W:= T'/pT \otimes C^\tilt$. These cocharacters generate $G'$ as an algebraic group, $G'(\Fp)$ preserves $W$. Hence $\Gal_F$ also preserves $W$, contradicting the choice of $T'$. 
    Therefore, for at least one $v \mid p$, the inequality in \Cref{lemma:conjpairmodpineq} is strict. Then, using
   \Cref{rmk:conjpairmodpineq}(2), we have that 
    \begin{align}\label{eqn:ptorsionnuineq}
        \nu(T',pT) 
        \ge \frac{1}{2} \left(1 - \frac{c_V}{p}\right)\left( w + \frac{1}{s(p)} \right) 
        \ge \frac{1}{2} \left(1 - \frac{c_V}{p}\right)\left( w + \frac{1}{[F:\Q]} \right).
    \end{align}
    
    Now, assume that $T'$ is as in the statement. Let $n$ be minimal such that $p^nT \subset T' \subset T$. For $i-0,\ldots, n$, define
    \[
    T^i := (p^iT) \cap T'.
    \]
    This defines a fiiltration $T' = T^0 \supseteq T^1 \supseteq T^2 \supseteq \cdots \supseteq T^{n-1} \supseteq T^n = p^nT$. Consider the composite map
    \[
    \frac{T^{i}}{T^{i+1}} \xhookrightarrow{\times p^{n-1-i}} \frac{T^{n-1}}{ T^n } = \frac{p^{n-1}T \cap T'}{p^nT } \subseteq \frac{p^{n-1}T}{p^n T} \xrightarrow[\sim]{\times p^{-(n-1)}}\frac{T}{pT},
    \]
    whose image is
    \[
    T_i'' := \frac{1}{p^i}T^i + pT.
    \]
    One can check that this yields inclusions 
    \[
        pT \subseteq T' + pT = T_0'' \subseteq T_1'' \subseteq \cdots \subseteq T_{n-1}'' \subseteq T_n'' = T.
    \]
    Because of the natural isomorphisms $\mf M(T) [1/\mu]\cong T \otimes_{\Zp}\Ainf[1/\mu]$ for any BKF module $\mf M(T)$, the map
    \[
        \frac{\mf M(T^i)}{\mf M(T^{i+1})} \to \frac{\mf M(T_i'')}{\mf M(pT)}
    \]
    becomes an isomorphism after inverting $\tilde{\xi}$, as $(\Ainf/p)[1/\mu] = (\Ainf/p)[1/\tilde{\xi}] = C^\tilt$. In particular, $\rk(\frac{\mf M(T^i)}{\mf M(T^{i+1})}) = \rk( \frac{\mf M(T_i'')}{\mf M(pT)})$ for all $i$. By the same proof as for \Cref{lemma:slopeproperties}(2), we get the degree inequality
    \begin{align} \label{eqn:Tislopeineq}
            \deg \left( \frac{\mf M(T^{i})}{\mf M(T^{i+1})}
            \right)
            \ge 
            \deg \left(
                \frac{\mf M(T''_i)}{\mf M(pT)}
            \right), 
    \end{align}
    where we use the notions of degree and rank of an effective torsion BKF module from \Cref{def:slope}.
    
    By construction, $\mf M(T')/\mf M(p^nT)$ is a successive extension of the modules $\mf M(T^{i})/ \mf M(T^{i+1})$ for $i = 0,\ldots,n-1$.
    Let $d' := \log_p \# (T'/p^nT)$. By the same reasoning as above, $\rk (\mf M(\sigma T')/\mf M(p^nT)) =d'$ for all $\sigma \in \Gal_F$, we have
    \begin{align}
        \nu(T', p^nT) &= \frac{1}{s(p)} \sum_{v\mid p} \int_{\Gal_F} \mu\left( \frac{\mf M_v(\sigma T')}{\mf M_v(p^nT)} \right) d\sigma \\
        &= \frac{1}{s(p)d'} \sum_{v\mid p} \int_{\Gal_F} \deg\left( \frac{\mf M_v(\sigma T')}{\mf M_v(p^nT)} \right) d\sigma.
    \end{align}
     The useful step is that this quantity is additive in short exact sequences. In particular, 
    \begin{align}
        \nu(T', p^nT) &= \frac{1}{s(p)d'} \sum_{v\mid p} \int_{\Gal_F} \sum_{i=0}^{n-1}\deg\left( \frac{\mf M_v(\sigma T^i)}{\mf M_v(\sigma T^{i+1})} \right) d\sigma.
    \end{align}
    Then we can apply (\ref{eqn:Tislopeineq}) to get the following inequality, which we can then rewrite in terms of slopes, because $\rk (\mf M_v(\sigma T_i'')/\mf M_v(pT)) = \log_p \# (T_i''/ pT)$ for any $\sigma \in \Gal_F$ and $v \mid p$. 
    \begin{align}
        \nu(T', p^nT) 
        &\ge \frac{1}{s(p)d'} \sum_{v\mid p} 
            \int_{\Gal_F} \sum_{i=0}^{n-1}
                \deg \left( \frac{\mf M_v(\sigma T_i'')}{\mf M_v(pT)} \right)
            d\sigma \\
        &= \frac{1}{d'} \sum_{i=0}^{n-1} 
        \log_p \# (T_i''/ pT)
         \left( 
         \frac{1}{s(p)}\sum_{v \mid p}
            \int_{\Gal_F}
                \mu \left( \frac{\mf M_v(\sigma T_i'')}{\mf M_v(pT)} \right)
            d\sigma 
        \right) \\
        &= \frac{1}{d'} \sum_{i=0}^{n-1} \log_p \# (T_i''/ pT) \; \nu(T_i'', pT) 
    \end{align}
    Finally, we can apply the inequality for $p$-torsion isogenies from (\ref{eqn:ptorsionnuineq}) to each term $\nu(T_i'', pT)$.
    \begin{align*}
        \nu(T', p^nT) \ge \frac{1}{2} \left(1 - \frac{c_V}{p}\right)\left( w + \frac{1}{[F:\Q]} \right) \frac{1}{d'} \sum_{i=0}^{n-1} \log_p \# (T_i''/ pT).
    \end{align*}
    By construction, $\sum_{i=0}^{n-1} \log_p \# (T_i''/ pT) = d'$, and we get
    \begin{align}
        \nu(T', p^nT) \ge \frac{1}{2} \left(1 - \frac{c_V}{p}\right)\left( w + \frac{1}{[F:\Q]} \right).
    \end{align}
    Suppose that $d$ and $d'$ satisfy $\# T/T' = p^d$, $\# T'/p^nT = p^{d'}$, so that $d + d' = \rk(T)n$. Then 
    \begin{align}
        \frac{w}{2} - \nu(T,T') &= \frac{d'}{d} \left( \nu(T', p^nT) - \frac{w}{2} \right) \\
        &\ge \frac{d'}{d} \left( \frac{1}{2}\left( 1 - \frac{c_V}{p} \right)\left( w + \frac{1}{[F:\Q]}\right) - \frac{w}{2} \right) \\
        &= \frac{d'}{2d[F:\Q]}\left(1 - \frac{c_V(w[F:\Q] + 1)}{p}\right) \\
        &\ge \frac{1}{2(\rk(T) - 1)[F:\Q]}\left(1 - \frac{c_{ M}}{p}\right)
    \end{align}
    where we use the fact that $d, d' \ge n$ and we let $c_{ M} := c_V(w[F:\Q] + 1)$. 
\end{proof}

\subsection{Proof of the main theorem in the varying prime case}
Now we can conclude the proof of \Cref{prop:differentprimes}. By \cite[Theorem 9.8]{koshikawa2015heights}, there are only finitely many heights coming from isogenies defined over $F$. Moreover, since the height change due to a $p$-power isogeny will always be a rational multiple of $\log p$, this will only account for finitely many primes. We can thus consider only the situation given in \Cref{cor:ptorsionheightbound}, which shows that for sufficiently large primes $p$, any $p$-power isogeny $M' \to M$ must have $h(M') > c$.